\documentclass[12pt,reqno]{NumPDEsArticle}

\usepackage[T1]{fontenc}
\usepackage[english]{babel}

\usepackage{mathtools}

\DeclareFontEncoding{LS2}{}{}
\DeclareFontSubstitution{LS2}{stix}{m}{n}
\DeclareSymbolFont{stixGlyphs}{LS2}{stixex}{m}{n}
\DeclareMathDelimiter{\lVert}{\mathopen}{stixGlyphs}{"F1}{stixGlyphs}{"F4}
\DeclareMathDelimiter{\rVert}{\mathclose}{stixGlyphs}{"F1}{stixGlyphs}{"F4}
\DeclareMathDelimiter{\langle}{\mathopen}{stixGlyphs}{"EA}{stixGlyphs}{"10}
\DeclareMathDelimiter{\rangle}{\mathclose}{stixGlyphs}{"EB}{stixGlyphs}{"11}

\usepackage{csquotes}
\usepackage{enumitem}
\usepackage{amsmath,amssymb}
\usepackage{microtype}
\usepackage{bbm}
\usepackage{bm}
\usepackage{tikz}
\usepackage{subfig}
\usepackage{graphicx}
\usepackage{stackengine}
\usepackage{xcolor}
\usepackage{float}

\title{Convergence of adaptive boundary element methods driven by functional a~posteriori error estimates}

\author{Alexander Freiszlinger}
\author{Dirk Pauly}
\author{Dirk Praetorius}

\address{TU Wien, Institute of Analysis and Scientific Computing, Wiedner Hauptstraße 8--10, 1040 Wien, Austria}
\email{alexander.freiszlinger@tuwien.ac.at \quad \rm (corresponding author)}
\email{dirk.praetorius@tuwien.ac.at}

\address{TU Dresden, Institute of Analysis, Zellescher Weg 12--14, 01069 Dresden, Germany}
\email{dirk.pauly@tu-dresden.de}

\thanks{This research was funded by the Austrian Science Fund (FWF) projects
\href{https://www.fwf.ac.at/en/research-radar/10.55776/F65}{10.55776/F65} (SFB
F65 ``Taming complexity in PDE systems''),
\href{https://www.fwf.ac.at/en/research-radar/10.55776/I6802}{10.55776/I6802}
(international project I6802 ``Functional error estimates for PDEs on unbounded
domains''), and
\href{https://www.fwf.ac.at/en/research-radar/10.55776/P33216}{10.55776/P33216}
(standalone project P33216 ``Computational nonlinear PDEs'').}
\begin{document}
\maketitle

\begin{abstract}
The recent work [Kurz \textsl{et al.}, Numer.\ Math., 147 (2021)] proposed functional \textsl{a~posteriori} error estimates for boundary element methods (BEMs) together with a related adaptive mesh-refinement strategy. Unlike most \textsl{a~posteriori} BEM error estimators, the proposed functional error estimators cover Galerkin as well as collocation BEM and, more importantly, do not control the error in the integral density on the boundary, but the error of the potential approximation in the domain, which is of greater relevance in practice. The estimates rely on the numerical solution of auxiliary problems on auxiliary strip domains along the boundary, where the strips are affected by the adaptive mesh-refinement and hence vary. For Galerkin BEM, we prove that the proposed adaptive mesh-refinement algorithm yields convergence of the potential error to zero. Due to the structural difference to residual-based estimators, the proof requires new ideas.
\end{abstract}

%%%%%%%%%%%%%%%%%%%%%%%%%%%%%%%%%%%%%%%%%%%%%%%%%%%%%%%%%%%%%%%%%%%%%%%%%%%%%%%%%%%
%%%%%%%%%%%%%%%%%%%%%%%%%%%%%%%%%%%%%%%%%%%%%%%%%%%%%%%%%%%%%%%%%%%%%%%%%%%%%%%%%%%
\section{Introduction}
\label{section:intro}
%%%%%%%%%%%%%%%%%%%%%%%%%%%%%%%%%%%%%%%%%%%%%%%%%%%%%%%%%%%%%%%%%%%%%%%%%%%%%%%%%%%
%%%%%%%%%%%%%%%%%%%%%%%%%%%%%%%%%%%%%%%%%%%%%%%%%%%%%%%%%%%%%%%%%%%%%%%%%%%%%%%%%%%

%!TEX root = fabem_convergence.tex
\noindent
Let $\Omega \subset \mathbb{R}^d$, $d \geq 2$, be a bounded Lipschitz domain with polyhedral boundary $\Gamma \coloneqq \partial\Omega$. 
We consider the Poisson problem with inhomogeneous Dirichlet boundary data $g$, i.e.,
\begin{align}\label{eq:poisson}
 \Delta u^\star = 0
 \quad \text{in } \Omega
 \quad \text{subject to} \quad
 u^\star = g
 \quad \text{on } \Gamma.
\end{align}
For the numerical solution of~\eqref{eq:poisson}, we employ the boundary element method (BEM).
For the ease of presentation in this introduction, we consider an indirect ansatz based on the single-layer potential
\begin{align}\label{eq:slp}
 (\widetilde V\phi^\star)(x) 
 \coloneqq \int_\Gamma G(x-y) \phi^\star(y) \, \mathop{}\!\mathrm{d}y 
 = u^\star(x) 
 \quad \text{for all } x \in \Omega
\end{align}
with unknown density $\phi^\star$, where 
\begin{equation} \label{eq:fundamental}
 G(x) = \begin{cases} - \frac{1}{2\pi} \log|x|  & \text{for } d = 2 \\
                       \frac{1}{(d-2)\lvert \mathbb{S}_{d-1}\rvert} \lvert x\rvert^{2-d} & \text{for } d \geq 3
 \end{cases}
\end{equation}
denotes the fundamental solution of the Laplacian. 
Here, $\lvert \mathbb{S}_{d-1}\rvert$ denotes the surface area of the unit sphere in $\mathbb{R}^d$. 
Taking the trace of~\eqref{eq:slp} on $\Gamma$, we obtain the weakly-singular boundary integral equation 
\begin{align}\label{eq:bie}
 (V\phi^\star)(x) = g(x)
 \quad \text{for almost all } x \in \Gamma,
\end{align}
where the integral representation of $(V\phi^\star)(x)$ coincides with that of $(\widetilde V\phi^\star)(x)$ (at least for bounded densities) but is now evaluated on $\Gamma$ (instead of the interior of $\Omega$). 
Supposing that $\operatorname{\mathrm{diam}}(\Omega) < 1$ in case of $d = 2$ (which can always be ensured by scaling of $\Omega$), the boundary integral equation~\eqref{eq:bie} admits a unique solution $\phi^\star \in H^{-1/2}(\Gamma)$; see, e.g., \cite{Mclean2000}.
Given a triangulation $\mathcal{T}_\ell^\Gamma$ of the boundary $\Gamma$, the boundary element method (BEM) approximates $\phi^\star \approx \phi_\ell \in \mathcal{P}^p(\mathcal{T}_\ell^\Gamma)$ by a $\mathcal{T}_\ell^\Gamma$-piecewise polynomial of degree $p \in \mathbb{N}_0$, where the index $\ell \in \mathbb{N}_0$ indicates the sequential step of a mesh-refining adaptive algorithm. 
Independently of how $\phi_\ell$ is computed (e.g., Galerkin BEM or collocation BEM with/without inexact iterative solver), this yields an approximation
\begin{align}\label{eq:approximation}
 u_\ell \coloneqq \widetilde V \phi_\ell
 \quad \text{satisfying} \quad
 \Delta u_\ell = 0 
 \quad \text{in } \Omega.
\end{align}
Unlike usual notation, we stress that $u_\ell$ is \emph{not} a discrete function but computed by an integral operator applied to the discrete density $\phi_\ell$, i.e., $u_\ell$ is only data sparse. 
Since $\nabla u^\star \approx \nabla u_\ell$ is usually of physical relevance and thus targeted by numerical simulation, the work~\cite{Kurz2020} proposed the functional \textsl{a~posteriori} error estimate 
\begin{align}\label{eq:estimation}
 \begin{split}
 \lVert \boldsymbol{\tau}_\ell\rVert_{L^2(\omega_\ell)}
 &\le 
% \max_{\substack{\boldsymbol{\tau}\in L^2(\Omega) \\ \operatorname{\mathrm{div}}\boldsymbol{\tau}= 0}}
% \Big( 2 \, \langle g - u_\ell|_\Gamma \:,\: \boldsymbol{n}\cdot \boldsymbol{\tau}|_\Gamma\rangle_\Gamma - \lVert \boldsymbol{\tau}\rVert_{L^2(\Omega)}^2 \Big)^{1/2}
% \\&
%  = 
 \lVert \nabla(u^\star - u_\ell)\rVert_{L^2(\Omega)}
 %\\&
 %= 
 %\min_{\substack{w \in H^1_(\Omega) \\ w|_\Gamma = g - u_\ell|_\Gamma}} 
 %\lVert \nabla w\rVert_{L^2(\Omega)}
 %\\&
 \le
 \lVert \nabla w_\ell\rVert_{L^2(\omega_\ell)} 
 + C_{\textnormal{osc}} \, \mathrm{osc}_\ell,
 \end{split}
\end{align}
where $\boldsymbol{\tau}_\ell \in \boldsymbol{H}({\rm div}, \omega_\ell)$ and $w_\ell \in H^1(\omega_\ell)$ are computable finite element functions that solve auxiliary problems on an $\ell$-dependent and shrinking strip $\omega_\ell \subseteq \Omega$ and $\mathrm{osc}_\ell$ are appropriate data approximation terms; see Section~\ref{section:mainresult} below.
The proposed adaptive strategy in~\cite{Kurz2020} relies only on the upper bound in~\eqref{eq:estimation} and adapts the strip $\omega_\ell$ in each step of the adaptive loop so that the necessary degrees of freedoms to compute $\boldsymbol{\tau}_\ell$ and $w_\ell$ appear to be proportional to the number of boundary elements $\#\mathcal{T}_\ell^\Gamma$. 
Of particular relevance is the fact that the upper and lower bound in~\eqref{eq:estimation} hold with \emph{known constant} $1$, where the adaptive algorithm, employed for Galerkin BEM with discrete density $\phi_\ell^\star = \phi_\ell \in \mathcal{P}^p(\mathcal{T}_\ell)$ and induced potential $u_\ell^\star \coloneqq \widetilde V\phi_\ell^\star$, empirically led to optimal convergence rates for the error and accurate error estimation $\lVert \nabla(u^\star - u_\ell^\star)\rVert_{L^2(\Omega)} \approx \lVert \nabla w_\ell\rVert_{L^2(\omega_\ell)}$ in the experiments of~\cite{Kurz2020}.

The purpose of the present work is to show that the proposed adaptive strategy applied for Galerkin BEM indeed guarantees plain convergence 
\begin{align}\label{eq:convergence}
 \lVert \nabla(u^\star - u_\ell^\star)\rVert_{L^2(\Omega)} \to 0
 \quad \text{as } \ell \to \infty, 
\end{align}
where the index $\ell \in \mathbb{N}_0$ indicates the sequential nature of the adaptive algorithm proposed in~\cite{Kurz2020}. 
While similar and even stronger results are available for adaptive Galerkin methods that aim for optimal adaptive approximation of the integral density $\phi^\star$ in terms of rates (see, e.g.,~\cite{Feischl2013,Gantumur2013,Feischl2014a,Gantner2022,Gantner2022a}) or even in terms of computational complexity (see, e.g.,~\cite{Dahmen2006,Gantumur2006,Fuehrer2019}), we note that the present analysis is impacted (and challenged) by the fact that the strip domain $\omega_\ell \subseteq \Omega$ is changed in each step of the adaptive loop.

We note that our proof of convergence~\eqref{eq:convergence} hinges on the use of Galerkin BEM as this allows to guarantee \textsl{a~priori} convergence of the computed integral densities $\phi_\ell^\star \to \phi_\infty^\star$ in the natural Sobolev space $H^{-1/2}(\Gamma)$ as well as on the assumption of additional regularity $g \in H^1(\Gamma)$ (despite only $g \in H^{1/2}(\Gamma)$) for the given Dirichlet data in~\eqref{eq:poisson} that, as in other works~\cite{Bartels2004,afkpp2013,Feischl2014a,Feischl2014}, is exploited to localize the data approximation terms.

\textbf{Outline.} Section~\ref{section:preliminaries} collects all necessary preliminaries and notation on Sobolev spaces on domains and boundaries, meshes and patches, discrete spaces, and Galerkin BEM.
In Section~\ref{section:mainresult}, we first recall the main idea of functional-based BEM error estimation from~\cite{Kurz2020} and then formulate the related adaptive algorithm (Algorithm~\ref{algo:adap}). 
Theorem~\ref{theorem:mainresult} is the main result of this work. 
It states, first, that the functional upper bound provides also a lower bound for the potential error (as well as the BEM error) and, second, that Algorithm~\ref{algo:adap} guarantees convergence. 
The proof of Theorem~\ref{theorem:mainresult} is given in Section~\ref{section:proof-mainresult}--\ref{section:extension}.
Section~\ref{section:numerics} concludes the work with some numerical experiments that underline the theory and show that, at least empirically, the proposed strategy leads to optimal convergence rates.

%%%%%%%%%%%%%%%%%%%%%%%%%%%%%%%%%%%%%%%%%%%%%%%%%%%%%%%%%%%%%%%%%%%%%%%%%%%%%%%%%%%
%%%%%%%%%%%%%%%%%%%%%%%%%%%%%%%%%%%%%%%%%%%%%%%%%%%%%%%%%%%%%%%%%%%%%%%%%%%%%%%%%%%
\section{Preliminaries and Notation}
\label{section:preliminaries}
%%%%%%%%%%%%%%%%%%%%%%%%%%%%%%%%%%%%%%%%%%%%%%%%%%%%%%%%%%%%%%%%%%%%%%%%%%%%%%%%%%%
%%%%%%%%%%%%%%%%%%%%%%%%%%%%%%%%%%%%%%%%%%%%%%%%%%%%%%%%%%%%%%%%%%%%%%%%%%%%%%%%%%%

%!TEX root = fabem_convergence.tex

\setcounter{subsection}{0}

%%%%%%%%%%%%%%%%%%%%%%%%%%%%%%%%%%%%%%%%%%%%%%%%%%%%%%%%%%%%%%%%%%%%%%%%%%%%%%%%%%%
\subsection{General notation}
%%%%%%%%%%%%%%%%%%%%%%%%%%%%%%%%%%%%%%%%%%%%%%%%%%%%%%%%%%%%%%%%%%%%%%%%%%%%%%%%%%%

Throughout this paper, let $\Omega$ be a bounded domain in $\mathbb{R}^d$, $d \geq 2$, with polyhedral Lipschitz boundary $\Gamma \coloneqq \partial \Omega$. 
By $\lvert \cdot\rvert$, we denote, without any ambiguity, the absolute value of a scalar, the Euclidean norm of a vector in $\mathbb{R}^d$, the $(d-1)$-dimensional Hausdorff measure of a surface (piece), or the $d$-dimensional Lebesgue volume of a measurable set.

Again, without any ambiguity, we write $\int_\Xi \cdot \mathop{}\!\mathrm{d}x$ for integration over a Lebesgue measurable set $\Xi \subseteq \Omega$ and over a Hausdorff measurable set $\Xi \subseteq \Gamma$.

For Lipschitz domains $\Xi \subseteq \Omega$, we denote by $\bm{n}_{\partial \Xi}$ the outward pointing unit normal vector on $\partial \Xi$.

For $1 < r < +\infty$, $\lVert \cdot\rVert_{L^r(\Xi)}$ denotes the usual $L^r$-norm and $1 < r' < +\infty$ is the Hölder conjugate, i.e., $1/r + 1/r' = 1$. 
For $r=2$, we additionally write $\langle \cdot \:,\: \cdot\rangle_{\Xi}$ for the $L^2(\Xi)$-inner product. 
For vector-valued spaces $[L^r(\Xi)]^n$ with $n \in \mathbb{N}$, we will omit the superscript if it is clear from the context. 

Throughout, discrete (or data-sparse) quantities are indicated by an index, e.g., $\mathcal{T}_\ell^\Gamma$ is a conforming triangulation of $\Gamma$ with corresponding Galerkin BEM approximation $\phi_\ell^\star$ and induced potential $u_\ell^\star$. 
Moreover, we always use the same index for related quantities, e.g., $\phi_\ell$ is a BEM approximation corresponding to $\mathcal{T}_\ell^\Gamma$ etc.
Finally, we abbreviate notation in proofs and write $A \lesssim B$, if there holds $A \le cB$ with a constant $c > 0$ which is independent of the adaptivity index $\ell \in \mathbb{N}_0$ and clear from the context. 
We write $A \simeq B$ to abbreviate $A \lesssim B \lesssim A$.

%%%%%%%%%%%%%%%%%%%%%%%%%%%%%%%%%%%%%%%%%%%%%%%%%%%%%%%%%%%%%%%%%%%%%%%%%%%%%%%%%%%
\subsection{Integer-order Sobolev spaces}
%%%%%%%%%%%%%%%%%%%%%%%%%%%%%%%%%%%%%%%%%%%%%%%%%%%%%%%%%%%%%%%%%%%%%%%%%%%%%%%%%%%

For Lipschitz domains $\Xi \subseteq \Omega$, we define the usual integer-order Sobolev space
\begin{equation*}
 W^{1,r}(\Xi) 
 \coloneqq 
 \{w \in L^r(\Xi) \:\vert\: \nabla w \in L^r(\Xi)\}
\end{equation*}
with corresponding norm
\begin{equation*}
 \lVert w\rVert_{W^{1,r}(\Xi)} 
 \coloneqq 
 \bigl( \lVert w\rVert_{L^r(\Xi)}^r + \lVert \nabla w\rVert_{L^r(\Xi)}^r \bigr)^{1/r},
\end{equation*}
where $\nabla$ denotes the distributional gradient. 
Furthermore, defining $C_0^\infty(\Xi)$ as the space of smooth functions with compact support in $\Xi$, we set
\begin{equation*}
 W^{1,r}_{\partial \Xi}(\Xi) 
 \coloneqq 
 \overline{C_0^\infty(\Xi) \cap W^{1,r}(\Xi)}^{W^{1,r}(\Xi)}.
\end{equation*}
\begin{remark}
Usually, one employs the notation $W^{1,r}_0(\Omega) = W^{1,r}_{\partial \Omega}(\Omega) = W^{1,r}_\Gamma(\Omega)$ for the space of $W^{1,r}$-functions with zero boundary conditions. 
However, our approach involves subdomains $\Xi \subseteq \Omega$ and corresponding boundary conditions. 
We therefore decided to use the notation $W^{1,r}_{\partial \Xi}(\Xi)$ to be explicit on these boundary conditions.
\end{remark}
For boundary pieces $\Xi \subseteq \Gamma$, we define 
\begin{equation*}
 W^{1,r}(\Xi) 
 \coloneqq 
 \{\varphi \in L^r(\Xi) \:\vert\: \nabla_{\Gamma} \varphi \in L^r(\Xi)\}
\end{equation*}
with corresponding norm 
\begin{equation*}
 \lVert \varphi\rVert_{W^{1,r}(\Xi)} 
 \coloneqq 
 \bigl( \lVert \varphi\rVert_{L^r(\Xi)}^r + \lVert \nabla_{\Gamma} \varphi\rVert_{L^r(\Xi)}^r \bigr)^{1/r},
\end{equation*}
where $\nabla_\Gamma$ denotes the surface gradient on the boundary $\Gamma$. 
For $r=2$, we write $H^1(\Xi) = W^{1,2}(\Xi)$ for either $\Xi \subseteq \Omega$ or $\Xi \subseteq \Gamma$. 
Additionally, we also write $H^1_{\partial \Xi}(\Xi) = W^{1,2}_{\partial \Xi}(\Xi)$ for $\Xi \subseteq \Omega$ being a Lipschitz domain. 
Moreover, we use the convention that $W^{0,r}(\Xi) = L^r(\Xi)$ and $H^0(\Xi) = L^2(\Xi)$ as well as $\lVert \cdot\rVert_{W^{0,r}(\Xi)} = \lVert \cdot\rVert_{L^r(\Xi)}$ and $\lVert \cdot\rVert_{H^0(\Xi)} = \lVert \cdot\rVert_{L^2(\Xi)}$ for $\Xi \subseteq \Omega$ or $\Xi \subseteq \Gamma$.

%%%%%%%%%%%%%%%%%%%%%%%%%%%%%%%%%%%%%%%%%%%%%%%%%%%%%%%%%%%%%%%%%%%%%%%%%%%%%%%%%%%
\subsection{Fractional-order Sobolev spaces}
%%%%%%%%%%%%%%%%%%%%%%%%%%%%%%%%%%%%%%%%%%%%%%%%%%%%%%%%%%%%%%%%%%%%%%%%%%%%%%%%%%%

For $0 < s < 1$, we define the $W^{s,r}(\Xi)$-seminorm as
\begin{equation*}
 |w|_{W^{s,r}(\Xi)} 
 \coloneqq 
 \begin{cases} \left( \int_{\Xi} \int_{\Xi} \frac{|w(x) - w(y)|^r}{\lvert x-y\rvert^{d + sr}} \mathop{}\!\mathrm{d}x \mathop{}\!\mathrm{d}y \right)^{1/r} & \text{if } \Xi \subseteq \Omega, \\[1ex]
 \left( \int_{\Xi} \int_{\Xi} \frac{|w(x) - w(y)|^r}{\lvert x-y\rvert^{d - 1 + sr}} \mathop{}\!\mathrm{d}x \mathop{}\!\mathrm{d}y \right)^{1/r} & \text{if }\Xi \subseteq \Gamma,
 \end{cases}
\end{equation*} 
and, additionally,
\begin{equation*}
 |w|_{W^{1,r}(\Xi)} 
 \coloneqq 
 \begin{cases} \lVert \nabla w\rVert_{L^2(\Xi)} & \text{if } \Xi \subseteq \Omega, \\[1ex]
 \lVert \nabla_\Gamma w\rVert_{L^2(\Xi)}  & \text{if }\Xi \subseteq \Gamma.
 \end{cases}
\end{equation*}
For $0 < s < 2$ and $s \neq 1$ we define the $W^{s,r}(\Xi)$-norm as
\begin{equation*}
 \lVert w\rVert_{W^{s,r}(\Xi)} 
 \coloneqq 
 (\lVert w\rVert_{W^{\lfloor s \rfloor,r}(\Xi)}^r + |w|_{W^{s - \lfloor s \rfloor,r}(\Xi)}^r)^{1/r},
\end{equation*}
establishing the fractional-order Sobolev space
\begin{equation*}
 W^{s,r}(\Xi) 
 \coloneqq 
 \{w \in L^r(\Xi) \:\vert\: \lVert w\rVert_{W^{s,r}(\Xi)} < +\infty\},
\end{equation*}
where $\mathbb{N}_0 \ni \lfloor s \rfloor \leq s$ denotes the largest possible integer and $W^{\lfloor s \rfloor,r}(\Xi)$ is the usual integer-order Sobolev space. 
In accordance with the definition of integer-order Sobolev spaces, we set $H^s(\Xi) \coloneqq W^{s,2}(\Xi)$ for $0 < s < 2$ and $s \neq 1$. 
Given a Lipschitz domain $\Xi \subseteq \Omega$, the trace operator $(\cdot) |_{\partial \Xi} \colon W^{s,r}(\Xi) \rightarrow W^{s - 1/r,r}(\partial \Xi)$ is well-defined and bounded for $1/r < s < 2$. 
Furthermore, there holds
\begin{equation*}
 W^{1,r}_{\partial \Xi}(\Xi) 
 = 
 \rm{ker}((\cdot) |_{\partial \Xi}) \cap W^{1,r}(\Xi).
\end{equation*}
Lastly, we note that the $W^{s,r}$-norm and $W^{s,r}$-seminorm are absolutely continuous with respect to the measure $\lvert \cdot\rvert$ of the underlying space, i.e., for all $w \in W^{s,r}(\Xi)$, there holds
\begin{equation} \label{eq:abscont}
 \lim_{\lvert T\rvert \to 0} \lVert w\rVert_{W^{s,r}(T)} 
 = 
 0 
 = 
 \lim_{\lvert T\rvert \to 0} \lvert w\rvert_{W^{s,r}(T)}.
\end{equation}
%

%%%%%%%%%%%%%%%%%%%%%%%%%%%%%%%%%%%%%%%%%%%%%%%%%%%%%%%%%%%%%%%%%%%%%%%%%%%%%%%%%%%%%
\subsection{Negative-order Sobolev spaces on the boundary}
%%%%%%%%%%%%%%%%%%%%%%%%%%%%%%%%%%%%%%%%%%%%%%%%%%%%%%%%%%%%%%%%%%%%%%%%%%%%%%%%%%%%%

Let $\Xi \subseteq \Omega$ be a Lipschitz domain. 
For $0 \leq s \leq 1$, we define the negative-order Sobolev spaces on $\partial \Xi$ as
\begin{equation*}
 W^{-s,r}(\partial \Xi) 
 \coloneqq 
 (W^{s,r'}(\partial \Xi))',
\end{equation*}
where the duality pairing is defined by continuous extension of the pairing between $L^{r'}(\partial \Xi)$ and $L^{r}(\partial \Xi)$ and, without any ambiguity, is denoted by $\langle \cdot \:,\: \cdot\rangle_{\partial \Xi}$. 
The space $H^{-s}(\partial \Xi)$ and the associated duality pairing is defined accordingly.

\subsection{The space \texorpdfstring{\boldmath $H(\operatorname{\mathrm{div}},\Xi)$}{H(div, Xi)}}

For $\Xi \subseteq \Omega$ being a Lipschitz domain, we set
\begin{equation*}
 H(\mathrm{div},\Xi) 
 \coloneqq 
 \{\bm{\sigma} \in [L^2(\Xi)]^d \:\vert\:\mathopen{} \operatorname{\mathrm{div}}\bm{\sigma} \in L^2(\Xi)\}
\end{equation*}
equipped with the usual graph norm
\begin{equation*}
 \lVert \bm{\sigma}\rVert_{H(\mathrm{div},\Xi)} 
 \coloneqq 
 \left( \lVert \bm{\sigma}\rVert_{L^2(\Xi)}^2 + \lVert \operatorname{\mathrm{div}}\bm{\sigma}\rVert_{L^2(\Xi)}^2 \right)^{1/2}.
\end{equation*}
We note that the normal traces $\bm{\sigma}|_{\partial \Xi} \cdot \bm{n}_{\partial \Xi}$ of $\bm{\sigma} \in H(\operatorname{\mathrm{div}},\Xi)$ are well-defined and belong to $H^{-1/2}(\partial \Xi)$ by Green's formula$\colon$ 
For given $\varphi \in H^{1/2}(\partial \Xi)$, we choose $w \in H^1(\Xi)$ such that $w|_{\partial \Xi} = \varphi$. 
Then, we define
\begin{equation*}
 \langle \bm{\sigma}|_{\partial \Xi} \cdot \bm{n}_{\partial \Xi} \:,\: \varphi\rangle_{\partial \Xi} 
 \coloneqq 
 \langle \bm{\sigma} \:,\: \nabla w\rangle_{\Xi} + \langle w \:,\: \operatorname{\mathrm{div}}\bm{\sigma}\rangle_{\Xi}.
\end{equation*}

%%%%%%%%%%%%%%%%%%%%%%%%%%%%%%%%%%%%%%%%%%%%%%%%%%%%%%%%%%%%%%%%%%%%%%%%%%%%%%%%%%%
\subsection{Meshes}
%%%%%%%%%%%%%%%%%%%%%%%%%%%%%%%%%%%%%%%%%%%%%%%%%%%%%%%%%%%%%%%%%%%%%%%%%%%%%%%%%%%

A $k$-dimensional simplex $T \subseteq \mathbb{R}^d$ is a compact set satisfying $T = \mathrm{conv} (\mathcal{N}_T)$ with vertices $\mathcal{N}_T = \{z_0, \ldots , z_k\}$ for some $z_0, \ldots , z_k \in \mathbb{R}^d$ and $\lvert T\rvert > 0$ for the $k$-dimensional measure. 
Let $\mathcal{T}_\ell$ be a conforming mesh on $\Omega$, i.e., a partition of $\overline{\Omega}$ into $d$-dimensional simplices $T$ such that the intersection of two different elements $T \cap T'$ is either empty or the joint $k$-dimensional compact simplex $T \cap T' = \mathrm{conv}(\mathcal{N}_T \cap \mathcal{N}_{T'})$. 
Throughout, we assume that all considered meshes $\mathcal{T}_{\ell}$ are obtained by a finite amount of newest-vertex bisection refinements of some initial conforming mesh $\mathcal{T}_0$ of $\Omega$. 
We refer to~\cite{Stevenson2008} for newest vertex bisection for $d \geq 2$ with admissible $\mathcal{T}_0$, as well as to~\cite{Karkulik2012} for $d=2$ and to~\cite{Diening2025} for $d \geq 2$ with non-admissible $\mathcal{T}_0$. 
It is well-known that newest vertex bisection guarantees uniform $\kappa$-shape regularity, i.e., there exists $\kappa > 0$ depending only on $\mathcal{T}_0$ such that each refined mesh $\mathcal{T}_\ell$ satisfies
\begin{equation*}
 \sup_{T \in \mathcal{T}_\ell} \frac{\mathrm{diam}(T)}{\lvert T\rvert^{1/d}} 
 \leq 
 \kappa.
\end{equation*}
Moreover, $\mathcal{T}_\ell$ induces a conforming boundary mesh on $\Gamma$ by
\begin{equation} \label{eq:boundarymesh}
 \mathcal{T}_\ell^\Gamma 
 \coloneqq 
 \mathcal{T}_\ell|_\Gamma 
 \coloneqq 
 \{T \cap \Gamma \:\vert\: T \in \mathcal{T}_\ell \text{ with }(d-1) \text{-dimensional measure } \lvert T \cap \Gamma\rvert > 0\},
\end{equation}
which is also $\kappa$-shape regular in the sense of
\begin{equation*}
 \sup_{F \in \mathcal{T}_\ell^\Gamma} \frac{\mathrm{diam}(F)}{\lvert F\rvert^{d-1}} 
 \leq 
 \kappa.
\end{equation*}
We denote by $\mathcal{N}_\ell$, $\mathcal{E}_\ell$ and $\mathcal{F}_\ell$ the sets of ($0$-dimensional) nodes, ($1$-dimensional) edges and ($(d-1)$-dimensional) facets of the mesh $\mathcal{T}_\ell$, respectively.
We define the patch of a set $U \subseteq \overline{\Omega}$ by
\begin{subequations} \label{eq:patch}
\begin{equation} \label{eq:1patch}
 \Omega_\ell[U] 
 \coloneqq 
 \mathrm{interior} \Bigl(\bigcup \mathcal{T}_\ell[U]\Bigr) \quad \text{with} \quad \mathcal{T}_\ell[U] \coloneqq \{T \in \mathcal{T}_\ell \:\vert\: T \cap \overline{U} \neq \emptyset\}.
\end{equation}
For $k \in \mathbb{N}$, we define the $k$-patch of $U$ as
\begin{equation} \label{eq:kpatch}
 \Omega_\ell^{[k]}[U] 
 \coloneqq 
 \mathrm{interior} \Bigl(\bigcup \mathcal{T}_\ell^{[k]}[U] \Bigr)\quad \text{with} \quad \mathcal{T}_\ell^{[k]}[U] 
 \coloneqq 
 \begin{cases} 
  \mathcal{T}_\ell[\Omega_\ell^{[k-1]}[U]] & \text{for } k > 1, \\
  \mathcal{T}_\ell[U] & \text{for } k = 1
 \end{cases}.
\end{equation}
\end{subequations}
To ease notation, we write $\Omega_\ell^{[k]}[z] \coloneqq \Omega_\ell^{[k]}(\{z\})$ and $\mathcal{T}_\ell^{[k]}[z] \coloneqq \mathcal{T}_\ell^{[k]}(\{z\})$ for any $z \in \mathcal{N}_\ell$. 
The $k$-patch $\Gamma_\ell^{[k]}[U]$ on the boundary is defined accordingly, with the volume mesh $\mathcal{T}_\ell$ replaced by the induced boundary mesh $\mathcal{T}_\ell^\Gamma$.

We define the piecewise constant mesh-size function $h_\ell \in L^\infty(\Omega)$ by $h_\ell|_T \coloneqq h_T \coloneqq |T|^{1/d}$ for every $T \in \mathcal{T}_\ell$. 
The quantities $h_E \coloneqq \lvert E\rvert = \mathrm{diam}(E)$ and $h_F \coloneqq \lvert F\rvert^{1/(d-1)}$ for edges $E \in \mathcal{E}_\ell$ and facets $F \in \mathcal{F}_\ell$, respectively, are defined analogously. 

%%%%%%%%%%%%%%%%%%%%%%%%%%%%%%%%%%%%%%%%%%%%%%%%%%%%%%%%%%%%%%%%%%%%%%%%%%%%%%%%%%%
\subsection{Discrete spaces}
%%%%%%%%%%%%%%%%%%%%%%%%%%%%%%%%%%%%%%%%%%%%%%%%%%%%%%%%%%%%%%%%%%%%%%%%%%%%%%%%%%%

Given $q \in \mathbb{N}$, let $\mathbb{P}^{q-1}(U)$ be the space of polynomials of degree at most $q-1$ defined on some set $U$ and denote by
\begin{equation*}
 \mathcal{P}^{q-1}(\mathcal{T}_\ell) 
 \coloneqq 
 \{v \in L^2(\Omega) \:\vert\: v|_T \in \mathbb{P}^{q-1}(T) \text{ for all } T \in \mathcal{T}_\ell\}
\end{equation*}
the space of $\mathcal{T}_\ell$-piecewise polynomials of degree at most $q-1$. 

We define the usual FEM-space of globally continuous $\mathcal{T}_\ell$-piecewise polynomials by
\begin{equation*}
 \mathcal{S}^q(\mathcal{T}_\ell) 
 \coloneqq 
 \{v \in \mathcal{P}^q(\mathcal{T}_\ell) \:\vert\: v \text{ is continuous on } \Omega\}  
 = 
 \mathcal{P}^q(\mathcal{T}_\ell) \cap H^1(\Omega).
\end{equation*}
Given $\Upsilon \subseteq \Gamma$, we set
\begin{equation*}
 \mathcal{S}_\Upsilon^q(\mathcal{T}_\ell) 
 \coloneqq 
 \{v \in \mathcal{S}^q(\mathcal{T}_\ell) \:\vert\: v|_{\Upsilon} = 0\}.
\end{equation*}
Furthermore, we define the $H(\operatorname{\mathrm{div}},\Omega)$-conforming Raviart--Thomas space
\begin{equation*}
 \mathcal{R}\mathcal{T}^{q-1}(\mathcal{T}_\ell) 
 \coloneqq 
 \{\bm{\sigma} \in H(\operatorname{\mathrm{div}},\Omega) \:\vert\: \bm{\sigma}|_T \in [\mathbb{P}^{q-1}(T)]^d + \mathbb{P}^{q-1}(T) \cdot \mathrm{id} \text{ for all } T \in \mathcal{T}_\ell\},
\end{equation*}
where $\mathrm{id}(x) = x$ denotes the identity. 
Given $\Upsilon \subseteq \Gamma$, we set
\begin{equation*}
 \mathcal{R}\mathcal{T}^{q-1}_\Upsilon(\mathcal{T}_\ell) 
 \coloneqq 
 \{\bm{\sigma} \in \mathcal{R}\mathcal{T}^{q-1}(\mathcal{T}_\ell) \:\vert\: \bm{\sigma}|_{\Upsilon} \cdot \bm{n}_{\Upsilon} = 0\}.
\end{equation*}
Discrete spaces on $\Gamma$ and $\omega_\ell$ are defined accordingly. 
Lastly, we introduce the $L^2$-orthogonal projections
\begin{equation*}
 \begin{split}
  Q_\ell^\Omega 
  &\colon 
  L^2(\Omega) \to \mathcal{P}^0(\mathcal{T}_\ell), 
  \quad Q_\ell^\Omega(w)|_T 
  \coloneqq 
  \langle w \:,\: 1\rangle_T/\lvert T\rvert, \\
  Q_\ell^\Gamma 
  &\colon
  L^2(\Gamma) \to \mathcal{P}^0(\mathcal{T}_\ell^\Gamma), 
  \quad Q_\ell^\Gamma(\varphi)|_F 
  \coloneqq 
  \langle \varphi \:,\: 1\rangle_F/\lvert F\rvert.
 \end{split}
\end{equation*}
Given $0 < s \leq 1$, $Q_\ell^\Omega$ and $Q_\ell^\Gamma$ satisfy
\begin{equation} \label{eq:avgprop}
 \begin{split}
  \lVert w - Q_\ell^\Omega w\rVert_{L^2(T)} 
  &\leq 
  C h_T^s \lvert w\rvert_{H^s(T)} 
  \quad \text{for all } T \in \mathcal{T}_\ell \text{ and } w \in H^s(T), \\
  \lVert \varphi - Q_\ell^\Gamma \varphi\rVert_{L^2(F)} 
  &\leq C h_F^s \lvert \varphi\rvert_{H^s(F)} 
  \quad \text{for all } F \in \mathcal{T}_\ell^\Gamma \text{ and } \varphi \in H^s(F),
 \end{split}
\end{equation}
with a constant $C > 0$ depending only on $\kappa$-shape regularity.

%%%%%%%%%%%%%%%%%%%%%%%%%%%%%%%%%%%%%%%%%%%%%%%%%%%%%%%%%%%%%%%%%%%%%%%%%%%%%%%%%%%
\subsection{Bases of discrete spaces}
%%%%%%%%%%%%%%%%%%%%%%%%%%%%%%%%%%%%%%%%%%%%%%%%%%%%%%%%%%%%%%%%%%%%%%%%%%%%%%%%%%%

For $U \subseteq \overline{\Omega}$ we denote by $\mathcal{D}_\ell^q(U)$ the set of all Lagrange nodes of the space $\mathcal{S}^q(\mathcal{T}_\ell)$ in $U$. 
We abbreviate $\mathcal{D}_\ell^q \coloneqq \mathcal{D}_\ell^q(\overline{\Omega})$ and let $(\zeta_{\ell,z}^q)_{z \in \mathcal{D}_\ell^q}$ be the corresponding nodal basis of $\mathcal{S}^q(\mathcal{T}_\ell)$, i.e., $\zeta_{\ell,z}^q(z') = \delta_{zz'}$ for all $z,z' \in \mathcal{D}_\ell^q$ with Kroneckers's delta $\delta$. 
Hence, there holds
\begin{equation*}
 v_\ell 
 = 
 \sum_{z \in \mathcal{D}_\ell^q} \alpha_z \zeta_{\ell,z}^q 
 \quad \text{for all } v_\ell \in \mathcal{S}^q(\mathcal{T}_\ell)
\end{equation*}
with unique coefficients $\alpha_z = v_\ell(z) \in \mathbb{R}$. 
For every $T \in \mathcal{T}_\ell$, the functions $(\zeta_{\ell,z}^q|_T)_{z \in \mathcal{D}_{\ell}^q(T)}$ are related to the standard basis of $\mathbb{P}^q(T_{\mathrm{ref}})$, denoted by $(\widetilde{\zeta}_j)_{j=1}^{n_{q,d}}$, by an affine diffeomorphism $\Phi_T \colon T \to T_{\mathrm{ref}}$, i.e., for all $j = 1, \ldots , n_{q,d}$ there is a unique $z \in \mathcal{D}_\ell^q(T)$ such that $\widetilde{\zeta}_j = \zeta_{\ell,z}^q|_T \circ \Phi_T$.

From now on, we assume that the polynomial degree $q$ is fixed. 
To ease notation, we will write $\mathcal{D}_\ell \coloneqq \mathcal{D}_\ell^q$, $\zeta_{\ell,z} \coloneqq \zeta_{\ell,z}^q$, and $n_d = n_{q,d}$.

Clearly, it holds $\mathcal{D}_\ell(\Gamma) = \mathcal{D}_\ell \cap \Gamma$. 
Hence, every $f_\ell \in \mathcal{S}^q(\mathcal{T}_\ell^\Gamma)$ in the form
\begin{equation*}
 f_\ell 
 = 
 \sum_{z \in \mathcal{D}_\ell(\Gamma)} \alpha_z \zeta_z|_{\Gamma}
\end{equation*}
can be naturally lifted to a function $\widehat{f}_\ell \in \mathcal{S}^q(\mathcal{T}_\ell)$ by setting
\begin{equation} \label{eq:natext}
 \widehat{f}_\ell 
 \coloneqq 
 \sum_{z \in \mathcal{D}_\ell(\Gamma)} \alpha_z \zeta_z,
\end{equation}
and we note that the discrete lifting $\widehat{(\cdot)} \colon \mathcal{S}^q(\mathcal{T}_\ell^\Gamma) \to \mathcal{S}^q(\mathcal{T}_\ell)$ is a linear operator.

%%%%%%%%%%%%%%%%%%%%%%%%%%%%%%%%%%%%%%%%%%%%%%%%%%%%%%%%%%%%%%%%%%%%%%%%%
\subsection{Boundary integral equation}
%%%%%%%%%%%%%%%%%%%%%%%%%%%%%%%%%%%%%%%%%%%%%%%%%%%%%%%%%%%%%%%%%%%%%%%%%%%%%%%%%%%

\noindent For $-1/2 \leq s \leq 1/2$, we consider the single-layer integral operator $V \colon H^{s - 1/2}(\Gamma) \to H^{s + 1/2}(\Gamma)$ defined by
\begin{equation} \label{eq:singlelayer}
 (V\psi)(x) 
 \coloneqq 
 \int_\Gamma G(x-y) \psi(y) \mathop{}\!\mathrm{d}y \quad \text{for all } x \in \Gamma,
\end{equation}
where $G$ is the fundamental solution of the Laplace equation; see~\eqref{eq:fundamental}. 
The operator $V \colon H^{s - 1/2}(\Gamma) \rightarrow H^{s + 1/2}(\Gamma)$ is a bounded isomorphism for $-1/2 \leq s \leq 1/2$ and elliptic for $s = 0$, provided that $\mathrm{diam}(\Omega) < 1$ if $d=2$, which can always be achieved by scaling; see, e.g.,~\cite{Mclean2000}. 
Hence, for every $g \in H^{1/2}(\Gamma)$ there is a unique solution $\phi^\star \in H^{-1/2}(\Gamma)$ of the problem
\begin{equation*}
 \langle V \phi^\star \:,\: \psi\rangle_{\Gamma} 
 = 
 \langle g \:,\: \psi\rangle_{\Gamma} 
 \quad \text{for all } \psi \in H^{-1/2}(\Gamma),
\end{equation*}
which satisfies $\lVert \phi^\star\rVert_{H^{-1/2}(\Gamma)} \simeq \lVert g\rVert_{H^{1/2}(\Gamma)}$, where the hidden equivalence constants depend only on the operator norm of $V$ and its inverse and hence only on $\Omega$. 
Moreover, the single-layer operator $\widetilde{V} \colon H^{-1/2}(\Gamma) \to H^1(\Omega)$ with the same integral representation as in~\eqref{eq:singlelayer} is a continuous linear operator satisfying $-\Delta(\widetilde{V}\phi) = 0$ for all $\phi \in H^{-1/2}(\Gamma)$.

%%%%%%%%%%%%%%%%%%%%%%%%%%%%%%%%%%%%%%%%%%%%%%%%%%%%%%%%%%%%%%%%%%%%%%%%%%
\subsection{Galerkin BEM}
%%%%%%%%%%%%%%%%%%%%%%%%%%%%%%%%%%%%%%%%%%%%%%%%%%%%%%%%%%%%%%%%%%%%%%%%%%%

Given $p \in \mathbb{N}_0$, we seek $\phi_\ell^\star \in \mathcal{P}^p(\mathcal{T}_\ell^\Gamma)$ such that 
\begin{equation} \label{eq:weakdisc}
 \langle V\phi_\ell^\star \:,\: \psi_\ell\rangle_{\Gamma} 
 = 
 \langle g \:,\: \psi_\ell\rangle_{\Gamma} 
 \quad \text{for all } \psi \in \mathcal{P}^p(\mathcal{T}_\ell^\Gamma).
\end{equation}
By ellipticity, the Lax--Milgram lemma yields existence and uniqueness of the solution $\phi_\ell^\star$ of~\eqref{eq:weakdisc} and there holds the Céa lemma
\begin{equation} \label{eq:cea}
 \lVert \phi^\star - \phi_\ell^\star\rVert_{H^{-1/2}(\Gamma)} 
 \lesssim 
 \min_{\psi_\ell \in \mathcal{P}^p(\mathcal{T}_\ell^\Gamma)} \lVert \phi^\star - \psi_\ell\rVert_{H^{-1/2}(\Gamma)},
\end{equation}
where the hidden constant depends only on ellipticity and continuity of $V$ and hence only on $\Omega$.

%%%%%%%%%%%%%%%%%%%%%%%%%%%%%%%%%%%%%%%%%%%%%%%%%%%%%%%%%%%%%%%%%%%%%%%%%%%%%%%%%%%
%%%%%%%%%%%%%%%%%%%%%%%%%%%%%%%%%%%%%%%%%%%%%%%%%%%%%%%%%%%%%%%%%%%%%%%%%%%%%%%%%%%
\section{Functional error estimation and main result}
\label{section:mainresult}
%%%%%%%%%%%%%%%%%%%%%%%%%%%%%%%%%%%%%%%%%%%%%%%%%%%%%%%%%%%%%%%%%%%%%%%%%%%%%%%%%%%
%%%%%%%%%%%%%%%%%%%%%%%%%%%%%%%%%%%%%%%%%%%%%%%%%%%%%%%%%%%%%%%%%%%%%%%%%%%%%%%%%%%

%!TEX root = fabem_convergence.tex
\noindent

%%%%%%%%%%%%%%%%%%%%%%%%%%%%%%%%%%%%%%%%%%%%%%%%%%%%%%%%%%%%%%%%%%%%%%%%%%%%%%%%%%%%%
\subsection{General notation}
%%%%%%%%%%%%%%%%%%%%%%%%%%%%%%%%%%%%%%%%%%%%%%%%%%%%%%%%%%%%%%%%%%%%%%%%%%%%%%%%%%%%%

Throughout this section, we adhere to the following notation: 
Let $\Omega_0 \subseteq \Omega$ be an initial strip domain along the boundary, i.e., $\Omega_0$ is a bounded polyhedral Lipschitz domain with $\Gamma = \partial\Omega \subseteq \partial\Omega_0$.
Let $\mathcal{T}_\ell$ be a conforming triangulation of $\Omega_0$ into simplices.
Let $\mathcal{T}_\ell^\Gamma \coloneqq \mathcal{T}_\ell|_\Gamma$ be the induced conforming triangulation of the boundary $\Gamma = \partial\Omega$ into surface simplices defined in~\eqref{eq:boundarymesh}. 
For fixed $k \in \mathbb{N}$, let $\mathcal{T}_\ell^\omega \coloneqq \mathcal{T}_\ell^{[k]}[\Gamma] \subseteq \mathcal{T}_\ell$ be the $k$-patch of $\Gamma$, e.g., for $k = 2$, the patch $\mathcal{T}_\ell^{[2]}[\Gamma]$ consists of all volume simplices $T \in \mathcal{T}_\ell$ that touch an element $T' \in \mathcal{T}_\ell$ with $T' \cap \Gamma \neq \emptyset$; see~\eqref{eq:patch} for the formal definition.
Let $\omega_\ell \coloneqq {\rm interior}\big( \bigcup \mathcal{T}_\ell^\omega \big)$ denote the induced strip domain along $\Gamma$ and $\Gamma_\ell^{\mathsf{c}}\coloneqq \partial \omega_\ell \backslash \Gamma$ be the complementary boundary part, i.e., $\partial\omega_\ell = \Gamma \cup \Gamma_\ell^{\mathsf{c}}$. 
By definition of $\omega_\ell$, we note that $\Gamma_\ell^{\mathsf{c}}$ is indeed closed and well separated from $\Gamma$ in the sense that $\mathrm{dist}(\Gamma, \Gamma_\ell^{\mathsf{c}}) > 0$. 
We refer to Figure~\ref{fig:strip} for an illustration.

\begin{figure}[!ht]

 \includegraphics[scale=0.5]{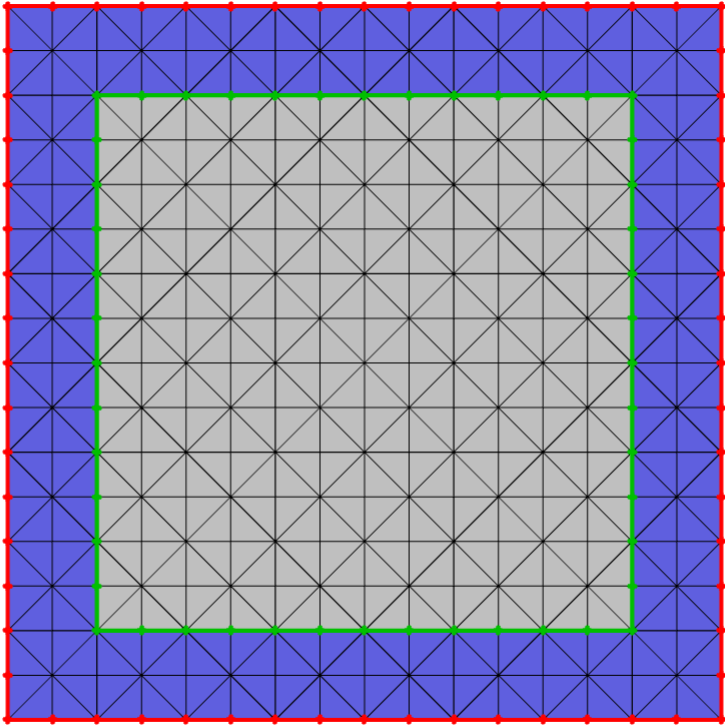}

 \caption{Boundary strip $\omega_\ell$ ({\color{blue} blue}) of strip depth $k=2$ of a square domain. 
  The outer boundary $\Gamma$ is depicted in {\color{red} red}, the complementary boundary $\Gamma_\ell^{\mathsf{c}}$ in {\color{green} green}.}

 \label{fig:strip}

\end{figure}
In the remainder of this section, we recall the ideas of~\cite{Kurz2020} for functional \textsl{a~posteriori} error estimation and formulate our main result. 
To this end, let $p\in \mathbb{N}_0$ und $q \in \mathbb{N}$ be polynomial degrees for the BEM as well as the auxiliary problems, respectively.

%%%%%%%%%%%%%%%%%%%%%%%%%%%%%%%%%%%%%%%%%%%%%%%%%%%%%%%%%%%%%%%%%%%%%%%%%%%%%%%%%%%
\subsection{Functional \textsl{a~posteriori} error estimation (upper bound)}
%%%%%%%%%%%%%%%%%%%%%%%%%%%%%%%%%%%%%%%%%%%%%%%%%%%%%%%%%%%%%%%%%%%%%%%%%%%%%%%%%%%

Suppose additional regularity $g \in H^1(\Gamma)$ for the inhomogeneous Dirichlet data in~\eqref{eq:poisson} (instead of only $g \in H^{1/2}(\Gamma)$).
Note that this leads to additional regularity $\phi^\star \in L^2(\Gamma)$ of the solution to~\eqref{eq:bie} (instead of only $\phi^\star \in H^{-1/2}(\Gamma)$).
Let $\phi_\ell \in \mathcal{P}^p(\mathcal{T}_\ell^\Gamma)$ be an approximation of $\phi^\star$.
Since $\Delta u_\ell = 0$ for $u_\ell \coloneqq \widetilde{V} \phi_\ell$, the Dirichlet principle yields
\begin{equation*}
 \lVert \nabla (u^\star - u_\ell)\rVert_{L^2(\Omega)} 
 \leq 
 \lVert \nabla w\rVert_{L^2(\Omega)} 
 \quad \text{for all } w \in H^1(\Omega) 
 \text{ with } w|_\Gamma 
 = 
 (u^\star-u_\ell)|_\Gamma
 = 
 g - u_\ell|_\Gamma.
\end{equation*}
We note that, in general, the boundary data $g - V\phi_\ell$ is not discrete and hence computable (discrete) functions $w_\ell$ cannot satisfy $w_\ell|_\Gamma = g - V\phi_\ell$. Therefore, we aim to discretize the boundary data by a projection $J_\ell \colon H^{1/2}(\Gamma) \rightarrow \mathcal{S}^q(\mathcal{T}_\ell^{\Gamma})$ that is $H^1(\Gamma)$-stable, i.e., there is a constant $C_{\textnormal{stab}}$ such that $\lVert J_\ell \psi\rVert_{H^1(\Gamma)} \leq C_{\textnormal{stab}} \lVert \psi\rVert_{H^1(\Gamma)}$ for all $\psi \in H^1(\Gamma)$.
Then, the discrete auxiliary problem reads:
Find $w_\ell \in \mathcal{S}^q(\mathcal{T}_\ell^\omega)$ such that
\begin{equation} \label{eq:defwh}
 \begin{split}
  \langle \nabla w_\ell \:,\: \nabla v_\ell\rangle_{L^2(\omega_\ell)} 
  &= 
  0 
  \quad \text{for all } v_\ell \in \mathcal{S}^q_{\partial \omega_\ell}(\mathcal{T}_\ell^\omega), \\
  w_\ell|_{\Gamma} 
  &= 
  g_\ell 
  \coloneqq 
  J_\ell(g - V\phi_\ell), \\ 
  w_\ell|_{\Gamma_\ell^{\mathsf{c}}} 
  &= 
  0.
 \end{split}
\end{equation}
It is well-known that~\eqref{eq:defwh} admits a unique solution $w_\ell \in \mathcal{S}^q(\mathcal{T}_\ell^\omega)$; see, e.g.,~\cite{Bartels2004,afkpp2013,Feischl2014}.
According to the homogeneous Dirichlet condition on $\Gamma_\ell^{\mathsf{c}}$, we can extend $w_\ell$ by zero to $w_\ell \in H^1(\Omega)$. 
Therefore,~\cite[Theorem~5]{Kurz2020} proves
\begin{align}\label{eq:upper-bound}
 \lVert \nabla (u^\star - u_\ell)\rVert_{L^2(\Omega)} 
 \le 
 \lVert \nabla w_\ell\rVert_{L^2(\Omega)} + \lVert (1-J_\ell)(g - V\phi_\ell)\rVert_{H^{1/2}(\Gamma)}.
\end{align}
By assumption on $g$ and the mapping properties of $V$, there holds $g - V\phi_\ell \in H^1(\Gamma)$. 
Moreover, the assumptions on the projection $J_\ell$ prove the Poincar\'e-type estimate
\begin{equation}\label{eq:rel}
 C_{\textnormal{osc}}^{-1} \, \lVert (1-J_\ell)(g - V\phi_\ell)\rVert_{H^{1/2}(\Gamma)} 
 \leq 
 \lVert h_\ell^{1/2} \nabla_\Gamma (1-J_\ell)(g - V\phi_\ell)\rVert_{L^2(\Gamma)}
 \eqqcolon 
 \mathrm{osc}_\ell,
\end{equation}
where the constant $C_{\textnormal{osc}}$ depends only on $C_{\textnormal{stab}}$, $\Gamma$, and $\kappa$-shape regularity of $\mathcal{T}_\ell$; see~\cite[Theorem~4]{kop2013} for $q = 1$ and~\cite[Lemma~7]{Aurada2015} for $q \ge 1$. 
Combining~\eqref{eq:upper-bound}--\eqref{eq:rel}, we prove
\begin{equation} \label{eq:compupperbound}
 \lVert \nabla (u^\star - u_\ell)\rVert_{L^2(\Omega)} 
 \leq
 \lVert \nabla w_\ell\rVert_{L^2(\Omega)} + C_{\textnormal{osc}} \, \mathrm{osc}_\ell 
 \eqqcolon 
 \eta_\ell + C_{\textnormal{osc}} \, \mathrm{osc}_\ell
\end{equation}
which is the computable upper bound of~\eqref{eq:estimation} with known constant $1$ up to data oscillations.

\begin{remark} \label{rem:whjustification}
The auxiliary problem~\eqref{eq:defwh} is motivated by the fact that, for $\omega_\ell = \Omega$ and $\Gamma_\ell^{\mathsf{c}}= \emptyset$, the discrete function $w_\ell$ is the unique FEM approximation of the potential error, i.e., $w_\ell \approx w \coloneqq u^\star - u_\ell$. 
As noted above, the homogeneous Dirichlet condition on $\Gamma_\ell^{\mathsf{c}}$ is chosen to allow for the conforming extension of $w_\ell$ from the strip domain $\omega_\ell$ to the whole of $\Omega$.
\end{remark}

\begin{remark}
The additional regularity $g \in H^1(\Gamma)$ is required for the definition of the localized oscillations
$\mathrm{osc}_\ell$ in~\eqref{eq:rel}. We note that even under this additional assumption, uniform mesh-refinement would not 
lead to the optimal rate of convergence $\mathcal{O}(h^{3/2})$ with respect to a uniform mesh-size. In general, this would 
require that $g$ is piecewise $H^2$ along $\Gamma$. The later adaptive strategy uses $\eta_\ell + \mathrm{osc}_\ell$ to resolve
both, potential geometric singularities of $\phi_\ell$ as well as non-smooth Dirichlet data.
\end{remark}

%%%%%%%%%%%%%%%%%%%%%%%%%%%%%%%%%%%%%%%%%%%%%%%%%%%%%%%%%%%%%%%%%%%%%%%%%%%%%%%%%%%
\subsection{Functional \textsl{a~posteriori} error estimation (lower bound)}
%%%%%%%%%%%%%%%%%%%%%%%%%%%%%%%%%%%%%%%%%%%%%%%%%%%%%%%%%%%%%%%%%%%%%%%%%%%%%%%%%%%

In order to obtain a computable lower bound, we aim to find a solution pair $(\boldsymbol{\tau}_\ell, p_\ell) \in \mathcal{R}^{q-1}_{\Gamma_\ell^{\mathsf{c}}}(\mathcal{T}_\ell^\omega) \times \mathcal{P}^{q-1}(\mathcal{T}_\ell^\omega)$ such that 
\begin{equation} \label{eq:mixed}
 \begin{split}
  \langle \boldsymbol{\tau}_\ell \:,\: \boldsymbol{\sigma}_\ell\rangle_{L^2(\omega_\ell)} + \langle \operatorname{\mathrm{div}}\boldsymbol{\sigma}_\ell \:,\: p_\ell\rangle_{L^2(\omega_\ell)} 
  &= 
  \langle g - V\phi_\ell \:,\: \boldsymbol{\sigma}_\ell \cdot \boldsymbol{n}\rangle_{L^2(\Gamma)} \\
  \langle \operatorname{\mathrm{div}}\boldsymbol{\tau}_\ell \:,\: q_\ell\rangle_{L^2(\omega_\ell)} 
  &= 
  0
 \end{split}
\end{equation}
for all $(\boldsymbol{\sigma}_\ell, q_\ell) \in \mathcal{R}^{q-1}_{\Gamma_\ell^{\mathsf{c}}}(\mathcal{T}_\ell^\omega) \times \mathcal{P}^{q-1}(\mathcal{T}_\ell^\omega)$.
It is well-known that this problem has a unique solution $(\boldsymbol{\tau}_\ell, p_\ell) \in \mathcal{R}^{q-1}_{\Gamma_\ell^{\mathsf{c}}}(\mathcal{T}_\ell^\omega) \times \mathcal{P}^{q-1}(\mathcal{T}_\ell^\omega)$; see, e.g.,~\cite{Bahriawati2005}. 
According to the homogeneous boundary condition $\boldsymbol{\tau}_\ell \cdot \boldsymbol{n}= 0$ on $\Gamma_\ell^{\mathsf{c}}$ included in the definition of $\mathcal{R}^{q-1}_{\Gamma_\ell^{\mathsf{c}}}(\mathcal{T}_\ell^\omega)$, we can extend $\boldsymbol{\tau}_\ell$ by zero to $\boldsymbol{\tau}_\ell \in \boldsymbol{H}(\operatorname{\mathrm{div}}, \Omega)$.
From~\cite[Theorem~5]{Kurz2020}, we then have that
\begin{equation*}
 2\langle g - V\phi_\ell \:,\: \boldsymbol{\tau}_\ell \cdot \boldsymbol{n}\rangle_{L^2(\Gamma)} - \lVert \boldsymbol{\tau}_\ell\rVert_{L^2(\omega_\ell)}^2 
 \leq 
 \lVert \nabla (u^\star - u_\ell)\rVert_{L^2(\Omega)}^2.
\end{equation*}
Plugging $(\boldsymbol{\sigma}_\ell, q_\ell) = (\boldsymbol{\tau}_\ell, p_\ell)$ into~\eqref{eq:mixed}, we obtain
\begin{equation*}
 \begin{split}
  \lVert \boldsymbol{\tau}_\ell\rVert_{L^2(\omega_\ell)}^2 + \langle \operatorname{\mathrm{div}}\boldsymbol{\tau}_\ell \:,\: p_\ell\rangle_{L^2(\omega_\ell)} 
  &= 
  \langle g - V\phi_\ell \:,\: \boldsymbol{\tau}_\ell \cdot \boldsymbol{n}\rangle_{L^2(\Gamma)} \\
  \langle \operatorname{\mathrm{div}}\boldsymbol{\tau}_\ell \:,\: p_\ell\rangle_{L^2(\omega_\ell)} 
  &= 
  0
 \end{split}
\end{equation*}
and hence
\begin{equation} \label{eq:lower-bound}
 \tau_\ell^2 \coloneqq \lVert \boldsymbol{\tau}_\ell\rVert_{L^2(\omega_\ell)}^2 
 = 
 2\langle g - V\phi_\ell \:,\: \boldsymbol{\tau}_\ell \cdot \boldsymbol{n}\rangle_{L^2(\Gamma)} - \lVert \boldsymbol{\tau}_\ell\rVert_{L^2(\omega_\ell)}^2 
 \leq 
 \lVert \nabla (u^\star - u_\ell)\rVert_{L^2(\Omega)}^2.
\end{equation}
Altogether, we proved the computable lower bound of~\eqref{eq:estimation} with known constant $1$.

\begin{remark} \label{rem:}
Similar to Remark~\ref{rem:whjustification}, the auxiliary problem~\eqref{eq:mixed} is motivated by the fact that, for $\omega_\ell = \Omega$ and $\Gamma_\ell^{\mathsf{c}}= \emptyset$, the discrete function $\boldsymbol{\tau}_\ell$ is the unique dual-mixed FEM approximation of the potential error, i.e., $\boldsymbol{\tau}_\ell \approx \boldsymbol{\tau}\coloneqq \nabla(u^\star - u_\ell)$.
\end{remark}

%%%%%%%%%%%%%%%%%%%%%%%%%%%%%%%%%%%%%%%%%%%%%%%%%%%%%%%%%%%%%%%%%%%%%%%%%%%%%%%%%%%
\subsection{Galerkin BEM, adaptive algorithm, and main result}
%%%%%%%%%%%%%%%%%%%%%%%%%%%%%%%%%%%%%%%%%%%%%%%%%%%%%%%%%%%%%%%%%%%%%%%%%%%%%%%%%%%

From now on, we restrict to Galerkin BEM\@. 
Given the triangulation $\mathcal{T}_\ell^\Gamma$ of $\Gamma$, let $\phi_\ell^\star \in \mathcal{P}^p(\mathcal{T}_\ell^\Gamma)$ be the unique solution to~\eqref{eq:weakdisc}. 
The adaptive algorithm is driven by the computable upper bound from~\eqref{eq:estimation}.
To this end, let $w_\ell^\star = w_\ell \in \mathcal{S}^q(\mathcal{T}_\ell^\omega)$ be the unique solution to~\eqref{eq:defwh} for $\phi_\ell = \phi_\ell^\star$, where we employ the Scott--Zhang projector $J_\ell \colon L^2(\Gamma) \to \mathcal{S}^q(\mathcal{T}_\ell^\Gamma)$ on the boundary as considered, e.g., in~\cite{Feischl2014} for $q = 1$ and~\cite{Aurada2015} for $q \ge 1$; see Section~\ref{section:proof-mainresult} below for the definition.
With this, we define the $\mathcal{T}_\ell$-elementwise contributions 
\begin{equation}\label{eq:indicators}
 \begin{split}
  \eta_\ell(T) 
  &\coloneqq 
  \lVert \nabla w_\ell^\star\rVert_{L^2(T)}, \\
  \mathrm{osc}_\ell(T) 
  &\coloneqq  
  \lVert h_\ell^{1/2} \nabla_\Gamma (1-J_\ell)(g - V\phi_\ell^\star)\rVert_{L^2(\partial T \cap \Gamma)}
  \quad \text{for all } T \in \mathcal{T}_\ell.
 \end{split}
\end{equation}
Based on these indicators and $\sum_{T \in \mathcal{T}_\ell^\omega} \eta_\ell(T)^2 = \lVert \nabla w_\ell^\star\rVert_{L^2(\omega_\ell)}^2$ and $\sum_{T \in \mathcal{T}_\ell^\omega} \mathrm{osc}_\ell(T)^2 = \mathrm{osc}_\ell^2$, the following adaptive algorithm has essentially been proposed by~\cite{Kurz2020}, where we note that $\eta_\ell(T) = 0 = \mathrm{osc}_\ell(T)$ for all $T \in \mathcal{T}_\ell \backslash \mathcal{T}_\ell^\omega$.

\begin{algorithm}[Adaptive algorithm]\label{algo:adap}
 \textbf{Input:}
 Initial strip domain $\Omega_0 \subseteq \Omega$ with $\Gamma = \partial\Omega \subseteq \partial\Omega_0$ and conforming triangulation $\mathcal{T}_0$ of $\Omega_0$, 
 polynomial degrees $p \in \mathbb{N}_0$ for the BEM in~\eqref{eq:weakdisc} and $q \in \mathbb{N}$ for the estimator in~\eqref{eq:compupperbound},
  patch order $k \in \mathbb{N}$,
 marking function $M \colon \mathbb{R}_{\ge0} \to \mathbb{R}_{\ge0}$ that is continuous at $0 = M(0)$,
 and right-hand side $g \in H^1(\Gamma)$.
 \\
 \textbf{Loop:} For all $\ell = 0,1,2,\dots$, repeat the following steps:
 \begin{enumerate}[label = \rm(\roman*)]
  \item Extract the boundary mesh $\mathcal{T}_{\ell}^{\Gamma} = \mathcal{T}_\ell|_\Gamma$ from $\mathcal{T}_{\ell}$ and solve~\eqref{eq:weakdisc} to obtain the Galerkin BEM solution $\phi_{\ell}^\star \in \mathcal{P}^p(\mathcal{T}_{\ell}^{\Gamma})$.
  \item Extract the boundary strip $\omega_\ell \coloneqq {\rm interior}\big( \bigcup \mathcal{T}_\ell^\omega \big)$ and the associated triangulation $\mathcal{T}_{\ell}^{\omega} \coloneqq \mathcal{T}_\ell^{[k]}[\Gamma]$,
  and compute the auxiliary FEM solution $w_\ell^\star \in \mathcal{S}^q(\mathcal{T}_\ell^\omega)$ by solving~\eqref{eq:defwh} with $\phi_\ell = \phi_\ell^\star$ and $J_\ell$ being the Scott--Zhang projector.
  \item Compute the error indicators $\eta_\ell(T)$ and $\mathrm{osc}_\ell(T)$ from~\eqref{eq:indicators} for all $T \in \mathcal{T}_\ell^\omega$.
  \item Determine a set $\mathcal{M}_\ell \subseteq \mathcal{T}_{\ell}^{\omega}$ (usually of minimal cardinality) such that
  \begin{align}\label{eq:marking}
   \max_{T \in \mathcal{T}_{\ell}^{\omega} \setminus \mathcal{M}_{\ell}} \big[ \eta_\ell(T)^2 + \mathrm{osc}_\ell(T)^2 \big] 
   \leq 
   M\Bigl(\max_{T \in \mathcal{M}_{\ell}} \big[ \eta_\ell(T)^2 + \mathrm{osc}_\ell(T)^2 \big]\Bigr).
  \end{align}
  \item Generate a new triangulation $\mathcal{T}_{\ell+1}$ by use of newest vertex bisection such that at least all elements in $\mathcal{M}_{\ell}$ are refined, i.e., $\mathcal{M}_\ell \subseteq \mathcal{T}_\ell \backslash \mathcal{T}_{\ell+1}$.
 \end{enumerate}
\end{algorithm}

\begin{remark}
The marking criterion~\eqref{eq:marking} is proposed and analyzed in the context of adaptive finite element methods in~\cite{Siebert2011}. 
Clearly, it is satisfied with $M(t) = t$ if 
\begin{align*}
 \max_{T \in \mathcal{T}_{\ell}^{\omega}} \big[ \eta_\ell(T)^2 + \mathrm{osc}_\ell(T)^2 \big] 
 =  
 \max_{T \in \mathcal{M}_\ell} \big[ \eta_\ell(T)^2 + \mathrm{osc}_\ell(T)^2 \big],
\end{align*}
i.e., $\mathcal{M}_\ell$ contains one element with largest overall indicator, which can easily be enforced algorithmically at linear cost.
Instead, the work~\cite{Kurz2020} uses the D\"orfler marking criterion from~\cite{doerfler1996} to select, for some fixed $0 < \theta \le 1$, the set $\mathcal{M}_\ell \subseteq \mathcal{T}_{\ell}^{\omega}$ via
\begin{align}\label{eq:doerfler}
 \theta \sum_{T \in \mathcal{T}_\ell^\omega} \big[ \eta_\ell(T)^2 + \mathrm{osc}_\ell(T)^2 \big]
 \le 
 \sum_{T \in \mathcal{M}_\ell} \big[ \eta_\ell(T)^2 + \mathrm{osc}_\ell(T)^2 \big].
\end{align}
Usually, $\mathcal{M}_\ell$ is obtained by sorting of the overall indicators so that
\begin{align*}
 \max_{T \in \mathcal{T}_{\ell}^{\omega} \backslash \mathcal{M}_\ell} \big[ \eta_\ell(T)^2 + \mathrm{osc}_\ell(T)^2 \big] 
 \le 
 \min_{T \in \mathcal{M}_\ell} \big[ \eta_\ell(T)^2 + \mathrm{osc}_\ell(T)^2 \big]
\end{align*}
and hence~\eqref{eq:marking} is satisfied with $M(t) = t$. 
If $\mathcal{M}_\ell$ in~\eqref{eq:doerfler} is obtained by inexact binsort as proposed in~\cite{stevenson2007}, then
\begin{align*}
 \max_{T \in \mathcal{T}_{\ell}^{\omega} \backslash \mathcal{M}_\ell} \big[ \eta_\ell(T)^2 + \mathrm{osc}_\ell(T)^2 \big] 
 \le 
 2 \min_{T \in \mathcal{M}_\ell} \big[ \eta_\ell(T)^2 + \mathrm{osc}_\ell(T)^2 \big]
\end{align*}
and hence\eqref{eq:marking} is satisfied with $M(t) = 2t$.
\end{remark}

\begin{remark}
Algorithm~\ref{algo:adap} could also include the computation of the estimator $\tau_{\ell}$ from~\eqref{eq:lower-bound} in order to determine the guaranteed lower bound~\eqref{eq:estimation}. 
However, our main result of Theorem~\ref{theorem:mainresult} exploits only the upper bound~\eqref{eq:compupperbound} of~\eqref{eq:estimation}.
\end{remark}

The following plain convergence theorem is the main result of this work. 
The subsequent sections will be concerned with its proof.

\begin{theorem}[Efficiency and plain convergence of functional upper bound]\label{theorem:mainresult}
Under the foregoing assumptions, Algorithm~\ref{algo:adap} guarantees the existence of a constant $C_{\textnormal{eff}} > 0$ such that
\begin{equation}\label{eq:theorem:mainresult}
 \begin{split}
  C_{\textnormal{eff}}^{-1}\eta_\ell 
  \leq 
  \lVert \phi^\star - \phi_\ell^\star\rVert_{H^{-1/2}(\Gamma)} 
  \leq 
  C_{\textnormal{V}} \lVert u^\star - u_\ell^\star\rVert_{H^1(\Omega)}
 \end{split}
\end{equation}
as well as plain convergence in the sense of
\begin{equation} \label{eq:plainconv}
 C_{\textnormal{P}}^{-1} \lVert u^\star - u_\ell^{\star}\rVert_{H^1(\Omega)} 
 \leq 
 \lVert \nabla(u^\star - u_\ell^\star)\rVert_{L^2(\Omega)} 
 \stackrel{\eqref{eq:compupperbound}}{\leq} 
 \eta_\ell +  C_{\textnormal{osc}}\mathrm{osc}_\ell
 \xrightarrow{\ell \to \infty} 
 0,
\end{equation}
where $C_{\textnormal{P}} > 0$ is a Poincaré constant and $C_V > 0$ is the operator norm of $\widetilde V \colon H^{-1/2}(\Gamma) \to H^1(\Omega)$, which both depend only on $\Omega$.
\end{theorem}
%\newpage

%%%%%%%%%%%%%%%%%%%%%%%%%%%%%%%%%%%%%%%%%%%%%%%%%%%%%%%%%%%%%%%%%%%%%%%%%%%%%%%%%%%
%%%%%%%%%%%%%%%%%%%%%%%%%%%%%%%%%%%%%%%%%%%%%%%%%%%%%%%%%%%%%%%%%%%%%%%%%%%%%%%%%%%
\section{Proof of Theorem~\ref{theorem:mainresult} (Efficiency and plain convergence)}
\label{section:proof-mainresult}
%%%%%%%%%%%%%%%%%%%%%%%%%%%%%%%%%%%%%%%%%%%%%%%%%%%%%%%%%%%%%%%%%%%%%%%%%%%%%%%%%%%
%%%%%%%%%%%%%%%%%%%%%%%%%%%%%%%%%%%%%%%%%%%%%%%%%%%%%%%%%%%%%%%%%%%%%%%%%%%%%%%%%%%

%!TEX root = fabem_convergence.tex

\noindent
The proof of Theorem~\ref{theorem:mainresult} crucially hinges on \textsl{a~priori} convergence of the solutions to the auxiliary problems~\eqref{eq:defwh}, which relies on a special construction of the Scott--Zhang projector $P_\ell \colon H^1(\Omega) \to \mathcal{S}^q(\mathcal{T}_\ell)$ from~\cite{Scott1990}. 
We briefly present this construction as well as properties of $P_\ell$. 
Subsequently, we prove \textsl{a~priori} convergence of the auxiliary solutions and efficiency of the functional error estimator. 
The proof of Theorem~\ref{theorem:mainresult} is then concluded by combining these results.

%%%%%%%%%%%%%%%%%%%%%%%%%%%%%%%%%%%%%%%%%%%%%%%%%%%%%%%%%%%%%%%%%%%%%%%%%%%%%
\subsection{Scott--Zhang projector} \label{subsec:SZ}
%%%%%%%%%%%%%%%%%%%%%%%%%%%%%%%%%%%%%%%%%%%%%%%%%%%%%%%%%%%%%%%%%%%%%%%%%%%%%

Recall that $\mathcal{D}_\ell = \mathcal{D}_\ell(\overline{\Omega})$ denotes the Lagrange nodes of $\mathcal{S}^q(\mathcal{T}_\ell)$. 
In essence, the construction of the Scott--Zhang projector $P_\ell \colon H^1(\Omega) \to \mathcal{S}^q(\mathcal{T}_\ell)$ in~\cite{Scott1990} is based on selecting $(K_{\ell,z})_{z \in \mathcal{D}_\ell}$ with $z \in K_{\ell,z}$, where $K_{\ell,z}$ is either an element $K_{\ell,z} \in \mathcal{T}_\ell$ or a facet $K_{\ell,z} \in \mathcal{F}_\ell$ for all $z \in \mathcal{D}_\ell$.
Then, one chooses $\psi_{\ell,z} \in \mathbb{P}^q(K_{\ell,z})$ such that
\begin{equation*}
 \int_{K_{\ell,z}} \psi_{\ell,z} \zeta_{\ell,z'} \, \mathop{}\!\mathrm{d}x 
 = 
 \delta_{zz'} 
 \quad \text{for all } z \in \mathcal{D}_\ell 
 \text{ and } z' \in \mathcal{D}_\ell\cap K_{\ell,z}
\end{equation*}
to define the projector $P_\ell$ by
\begin{equation} \label{eq:SZdef}
 P_\ell w 
 = 
 \sum_{z \in \mathcal{D}_\ell} \Bigl(\int_{K_{\ell,z}} \psi_{\ell,z} w \mathop{}\!\mathrm{d}x \Bigr) \zeta_{\ell,z}.
\end{equation}
The construction of $P_\ell$ thus provides considerable freedom in the choice of the sets $(K_{\ell,z})_{z \in \mathcal{D}_\ell}$, which can be exploited to ensure certain additional properties of $P_\ell$. 

\begin{lemma} \label{lem:SZconstruct}
 There exist sets $(K_{\ell,z})_{\ell \in \mathbb{N}_0,z \in \mathcal{D}_\ell}$ such that the following properties hold.
 \begin{enumerate}[label = \rm(\roman*)]
  \item If $z \in \mathrm{interior}(T_z) \cap \mathcal{D}_\ell$ for some $T_z \in \mathcal{T}_\ell$, we set $K_{\ell,z} \coloneqq T_z$.
  \item If $z \in \mathrm{interior}(F_z) \cap \mathcal{D}_\ell$ for some facet $F_z \in \mathcal{F}_\ell$, we set $K_{\ell,z} \coloneqq F_z$.
  \item If $z \in \mathcal{D}_\ell(\Gamma)$, then $K_{\ell,z} \subseteq \Gamma$.
  \item If $z \in \mathcal{D}_\ell(\omega_{\ell})$, then $\mathrm{interior}(K_{\ell,z}) \subseteq \omega_{\ell}$.
  \item If $z \in \mathcal{D}_\ell(\Omega \setminus \omega_{\ell})$, then $K_{\ell,z} \subseteq \Omega \setminus \omega_{\ell}$.
  \item If $z \in \mathcal{D}_\ell \cap \mathcal{D}_{\ell-1}$ and $\zeta_{\ell,z} = \zeta_{\ell-1,z}$, then $K_{\ell,z} = K_{\ell-1,z}$.
 \end{enumerate}
\end{lemma}

\begin{proof}
In the following, we prove that all the properties \rm (i)--(vi) can hold simultaneously. 
The proof is split into $3$ steps.

\textbf{Step~1 (Properties of \boldsymbol{$\omega_\ell$}).}
We first note that, due to the construction of $\omega_\ell$, whenever $z \in \mathcal{D}_\ell \cap \omega_\ell$, there holds $\mathrm{interior}(F_z) \subset \omega_\ell$ for all facets $F_z \in \mathcal{F}_\ell$ with $z \in F_z$. 
Additionally, if $z \in \mathrm{interior}(T_z)$ for some element $T_z \in \mathcal{T}_\ell$, there holds $\mathrm{interior}(T_z) \subseteq \omega_\ell$. 
By a similar argument, an analogous statement holds for $\Omega \setminus \overline{\omega_\ell}$. 
Moreover, if $z \in \mathcal{D}_\ell \cap \Gamma_\ell^{\mathsf{c}}$, there exists a facet $F_z \in \mathcal{F}_\ell$ with $z \in F_z \subset \Omega \setminus \omega_\ell$. 
Hence, whenever $z \in \mathcal{D}_\ell(\Omega \setminus \omega_\ell)$, there exists either an element $T_z \in \mathcal{T}_\ell$ with $z \in \mathrm{int}(T_z) \subset \Omega \setminus \overline{\omega_\ell}$ or a facet $F_z \in \mathcal{F}_\ell$ with $z \in F_z \subset \Omega \setminus \omega_\ell$. 
We further note that the properties \rm (iii)--(v) are independent of each other since $\Gamma$, $\omega_\ell$, and $\Omega \setminus \omega_\ell$ are pairwise disjoint.

In the following, we will construct $(K_{\ell,z})_{\ell \in \mathbb{N}_0,z \in \mathcal{D}_\ell}$ recursively starting with $\ell = 0$. 
 
\textbf{Step~2 (Case \boldsymbol{$\ell = 0$}).} 
We first note that $\ell=0$ excludes \rm (vi). 
Let $z \in \mathrm{int}(T_z) \cap \mathcal{D}_0$ for some $T_z \in \mathcal{T}_0$ and define $K_{0,z} \coloneqq T_z$. 
Note that this excludes \rm (iii). 
Hence, it only remains to prove \rm{(iv)--(v)}. 
By Step~1, $z \in \omega_0$ implies $T_z \in \mathcal{T}_0^\omega$ and hence $\mathrm{interior}(K_{0,z}) =  \mathrm{interior}(T_z) \subseteq \omega_0$, yielding \rm (iv). 
Similarly, Step~1 also establishes \rm (v).
 
Let $z \in \mathrm{interior}(F_z) \cap \mathcal{D}_0$ for some facet $F_z \in \mathcal{F}_\ell$ and define $K_{0,z} \coloneqq F_z$.
Arguing analogously to the case $z \in \mathrm{interior}(T_z) \cap \mathcal{D}_0$, we show that \rm{(iv)--(v)} hold. 
It thus remains to prove \rm (iii), which is an easy consequence of the fact that $z \in \Gamma$ implies $K_{0,z} = F_z \subseteq \Gamma$.

If $z \in \mathcal{D}_0$ is neither contained in the interior of an element nor the interior of a facet, we have to choose a suitable facet such that \rm (iii)--(v) hold. 
If $z \in \Gamma$, there exists $F_z \in \mathcal{T}_0^\Gamma$ with $z \in F_z$. 
Choosing $K_{0,z} \coloneqq F_z$ proves \rm (iii). 
Given $z \in \omega_0$ and the fact that $z$ is not contained in the interior of an element, Step~1 yields the existence of $F_z \eqqcolon K_{0,z}$ with $z \in F_z$ and $\mathrm{int}(F_z) \subset \omega_0$, which thus proves \rm (iv). 
Lastly, if $z \in \Omega \setminus \omega_0$, due to Step~1, there exists $F_z \eqqcolon K_{0,z}$ with $z \in F_z$ and $F_z \subseteq \Omega \setminus \omega_0$.

\textbf{Step~3 (Induction on \boldsymbol{$\ell$}).}
Suppose that $(K_{\ell-1,z})_{z \in \mathcal{D}_{\ell-1}}$ is chosen such that \rm{(i)--(vi)} hold for $\ell-1$. 
We now choose $K_{\ell,z}$ such that \rm (i)--(vi) also hold for $\ell$.
 
Let $z \in \mathcal{D}_\ell$. 
If either $z \notin \mathcal{D}_{\ell-1}$ or $\zeta_{\ell,z} \neq \zeta_{\ell-1,z}$, \rm (vi) is excluded, and an analogous choice of $K_{\ell,z}$ as in the case $\ell=0$, guarantees \rm (i)--(v).

Let $z \in \mathcal{D}_\ell \cap \mathcal{D}_{\ell -1}$ and $\zeta_{\ell,z} = \zeta_{\ell-1,z}$ and define $K_{\ell,z} \coloneqq K_{\ell-1,z}$. 
This validates \rm (vi). 
If $z \in \mathrm{interior}(T_z)$ for some element $T_z \in \mathcal{T}_\ell$ or $z \in \mathrm{interior}(F_z)$ for some facet $F_z \in \mathcal{F}_\ell$, $\zeta_{\ell,z} = \zeta_{\ell-1,z}$ implies that $K_{\ell,z} = K_{\ell-1,z} = T_z \in \mathcal{T}_{\ell-1} \cap \mathcal{T}_\ell$ or $K_{\ell,z} = K_{\ell-1,z} = F_z \in \mathcal{F}_{\ell-1} \cap \mathcal{F}_\ell$, respectively. 
This proves \rm (i)--(ii). 
If $z \in \Gamma$, the induction hypothesis yields $K_{\ell,z} = K_{\ell-1,z} \subset \Gamma$. 
This proves \rm (iii). If $z \in \omega_\ell$, Step~1 yields $K_{\ell,z} \subset \omega_\ell$. 
This proves \rm (iv). 
By analogous arguments, there holds \rm (v). 
This concludes the proof.
\end{proof}

From now on, we assume that the Scott--Zhang projector~\eqref{eq:SZdef} is constructed with sets $(K_{\ell,z})_{z \in \mathcal{D}_\ell}$ from Lemma~\ref{lem:SZconstruct}. 
We note that the Scott--Zhang projectors on boundaries $J_\ell \colon L^2(\Gamma) \to \mathcal{S}^q(\mathcal{T}_\ell^\Gamma)$ from~\cite{Feischl2014, Aurada2015} are related to $P_\ell$ by the trace operator $(P_\ell v)|_\Gamma = J_\ell(v|_\Gamma)$ for all $v \in H^1(\Omega)$.
The following result states certain stability properties of the Scott--Zhang projector, which are immediate consequences of the considerations of, e.g.,~\cite{Scott1990,Ciarlet2013}. 
They rely only on the properties Lemma~\ref{lem:SZconstruct}\rm (i)--(iii).

\begin{lemma}[Stability of Scott--Zhang projectors] \label{lem:SZ}
 There exists a constant $C_{\textnormal{stab}}>0$ depending only on $\kappa$-shape regularity of $\mathcal{T}_\ell$, $s$, and $t$  such that the following statements \rm (a)--(b) hold.
 \begin{enumerate}[label = \rm(\roman*)]
  \item[(a)] For given $1/2 < t \leq 1$, $0 < s \leq t$, $T \in \mathcal{T}_\ell$, and $w \in H^t(\Omega_\ell[T])$, there holds
  \begin{equation} \label{eq:SZvolstab}
   \lvert P_\ell w\rvert_{H^s(T)}
   \leq
   C_{\textnormal{stab}} h_T^{t-s} \lvert w\rvert_{H^t(\Omega_\ell[T])}.
  \end{equation}
  In particular, $P_\ell w$ is the quasi-best approximation of $w$ in the sense of
  \begin{equation} \label{eq:SZvolbest}
   \lvert w - P_\ell w\rvert_{H^s(T)}
   \leq
   C_{\textnormal{stab}} h_T^{t-s} \min_{w_\ell \in \mathcal{S}^q(\mathcal{T}_\ell)} \lvert w - w_\ell\rvert_{H^t(\Omega_\ell[T])}.
  \end{equation}
  \item[(b)] For given $0 < s \leq t \leq 1$, $F \in \mathcal{T}_\ell^\Gamma$ and $\varphi \in H^t(\Gamma_\ell[F])$, there holds
  \begin{equation} \label{eq:SZtracestab}
   \lvert J_\ell \varphi\rvert_{H^s(F)}
   \leq
   C_{\textnormal{stab}} h_F^{t-s} \lvert \varphi\rvert_{H^t(\Gamma_\ell[F])}.
  \end{equation}
  In particular, $J_\ell w$ is the quasi-best approximation of $w$ in the sense of
  \begin{equation} \label{eq:SZtracebest}
   \lvert \varphi - J_\ell \varphi\rvert_{H^s(F)}
   \leq
   C_{\textnormal{stab}} h_F^{t-s} \min_{\varphi_\ell \in \mathcal{S}^q(\mathcal{T}_\ell^\Gamma)} \lvert \varphi - \varphi_\ell\rvert_{H^t(\Gamma_\ell[F])}
  \end{equation}
  Additionally, for $\varphi \in L^2(\Gamma_\ell[F])$ there holds 
  \begin{equation} \label{eq:SZtraceL2}
   \lVert J_\ell \varphi\rVert_{L^2(F)} 
   \leq 
   C_{\textnormal{stab}} \lVert \varphi\rVert_{L^2(\Gamma_\ell[F])}
  \end{equation}
  and
  \begin{equation} \label{eq:SZtraceL2best}
   \lVert \varphi - J_\ell \varphi\rVert_{L^2(F)} 
   \leq 
   C_{\textnormal{stab}} \min_{\varphi_\ell \in \mathcal{S}^q(\mathcal{T}_\ell^\Gamma)} \lVert \varphi - \varphi_\ell\rVert_{L^2(\Gamma_\ell[F])}.
  \end{equation}
 \qed
 \end{enumerate}
\end{lemma} 

The properties \rm (i)--(iii) and (vi) of Lemma~\ref{lem:SZconstruct} ensure \textsl{a~priori} convergence of the Scott--Zhang projectors, e.g.,~\cite{Feischl2014}. 
We note that~\cite{Feischl2014} states the result for $q = 1$, but the proof transfers verbatim to any fixed $q \in \mathbb{N}$.

\begin{lemma}[{\textsl{A~priori} convergence of Scott--Zhang projector on $\boldsymbol{\Omega}$ \cite{Feischl2014}}]\label{lem:SZconv}
 For every $v \in H^{1}(\Omega)$, there exists a unique limit $P_{\infty}v \in H^1(\Omega)$ such that
 \begin{equation}
  \lim_{\ell \rightarrow \infty} \lVert P_{\infty}v - P_{\ell}v\rVert_{H^{1}(\Omega)} 
  = 
  0 
  \quad \text{for all } v \in H^{1}(\Omega).
  \qquad \qed
 \end{equation}
\end{lemma}

Lastly, the choices of $K_{\ell,z}$ in Lemma~\ref{lem:SZconstruct} ensure validity ofthe following statement.
\begin{lemma}[Scott--Zhang projector is compatible with boundary strip] \label{lem:SZrest}
 Let $\ell \in \mathbb{N}_0$ and $w \in H^1(\Omega)$. 
 Then, $\mathrm{supp}(w) \subseteq \overline{\omega_\ell}$ yields $\mathrm{supp}(P_\ell w) \subseteq \overline{\omega_\ell}$.
\end{lemma} 

\begin{proof}
Given $w \in H^1(\Omega)$ with $\mathrm{supp}(w) \subseteq \overline{\omega_\ell}$ and $z \in \mathcal{D}_\ell \setminus \omega_\ell$, there holds $w|_{K_{\ell,z}} = 0$ due to the properties \rm{(i)}, \rm{(iii)}, and \rm{(v)} of Lemma~\ref{lem:SZconstruct}. 
Hence,
\begin{equation*}
 P_\ell w 
 = 
 \sum_{z \in \mathcal{D}_\ell(\omega_\ell)} \Bigl( \int_{K_{\ell,z}} \psi_{\ell,z} w \, \mathop{}\!\mathrm{d}x \Bigr) \zeta_{\ell,z},
\end{equation*}
which yields $\mathrm{supp}(P_\ell w) \subseteq \displaystyle{\bigcup_{z \in \mathcal{D}_\ell(\omega_\ell)} \mathrm{supp}(\zeta_{\ell,z})}$. 
Since $(\zeta_{\ell,z})_{z \in \mathcal{D}_\ell \cap \omega_\ell}$ is the nodal basis of the space $\mathcal{S}^q_{\partial \omega_\ell}(\mathcal{T}_\ell^\omega)$, we obtain $\mathrm{supp}(\zeta_{\ell,z}) \subseteq \overline{\omega_\ell}$ for all $z \in \mathcal{D}_\ell(\omega_\ell)$. 
This implies $\mathrm{supp}(P_\ell w) \subseteq \overline{\omega_\ell}$ and concludes the proof.
\end{proof}

%%%%%%%%%%%%%%%%%%%%%%%%%%%%%%%%%%%%%%%%%%%%%%%%%%%%%%%%%%%%%%%%%%%%%%%%%%%%%
\subsection{Proof of Theorem~\ref{theorem:mainresult}} \label{subsec:proofmainresult}
%%%%%%%%%%%%%%%%%%%%%%%%%%%%%%%%%%%%%%%%%%%%%%%%%%%%%%%%%%%%%%%%%%%%%%%%%%%%%

As in other works~\cite{Morin2008, Siebert2011, Gantner2021} on plain convergence of adaptive finite element methods, one crucial ingredient of our proof is \textsl{a~priori} convergence in Hilbert spaces that already goes back to the seminal work~\cite{bv1984} and results from the validity of the C\'ea lemma in the context of the Lax--Milgram lemma; see also~\cite{Morin2008} for the generalization to uniformly $\inf$-$\sup$-stable problems.

\begin{lemma}[{\textsl{A~priori} convergence in Hilbert spaces~\cite{bv1984}}]\label{lemma:apriori}
Let $H$ be a Hilbert space and $X_\ell \subseteq H$ a sequence of closed subspaces with $X_\ell \subseteq X_{\ell+1}$ for all $\ell \in \mathbb{N}_0$. 
Let $a(\cdot,\cdot)$ be an elliptic and continuous bilinear form and $F \in H'$. 
Then, for all $\ell \in \mathbb{N}_0$, the Lax--Milgram yields the existence and uniqueness of $x_\ell \in X_\ell$ with
\begin{align}
 a(x_\ell, y_\ell) 
 = 
 F(y_\ell)
 \quad \text{for all } y_\ell \in X_\ell.
\end{align}
Moreover, there exists a unique limit $x_\infty \in X_\infty \coloneqq \overline{\bigcup_{\ell \in \mathbb{N}_0} X_\ell} \subseteq H$ such that
\begin{align}
 \lim_{\ell \to \infty} \lVert x_\infty - x_\ell\rVert_{H} 
 = 
 0.
 \qquad \qed
\end{align}
\end{lemma}

This allows to prove \textsl{a~priori} convergence of the discretized boundary residuals $g_{\ell} = J_{\ell}(g - V\phi_{\ell}^\star)$ from~\eqref{eq:defwh}.

\begin{corollary}[{\textsl{A~priori} convergence of discretized boundary residuals $\boldsymbol{g_\ell}$}] \label{cor:gellconv}
Recall $g_{\ell}$ from~\eqref{eq:defwh}. 
Then, there exists a limit $g_{\infty} \in H^{1/2}(\Gamma)$ such that 
\begin{equation}
 \lim_{\ell \rightarrow \infty} \lVert g_{\infty} - g_{\ell}\rVert_{H^{1/2}(\Gamma)} 
 = 
 0,
\end{equation}
where we note that $g_\ell$ is defined on $\partial\omega_\ell \supseteq \Gamma$.
\end{corollary}

\begin{proof}
According to Lemma~\ref{lemma:apriori} applied to $H = H^{-1/2}(\Gamma)$ and $X_\ell = \mathcal{P}^p(\mathcal{T}_\ell^\Gamma)$, there exists $\phi_\infty^\star \in H^{-1/2}(\Gamma)$ such that 
\begin{align}\label{eq:apriori:phi}
 \lVert \phi_\infty^\star - \phi_\ell^\star\rVert_{H^{-1/2}(\Gamma)} 
 \xrightarrow{\ell \to \infty} 
 0.
\end{align}
%
%According to~\cite{Morin2008}, nestedness of the spaces $(\mathcal{P}^0(\mathcal{T}_\ell^\Gamma))_{\ell \in \mathbb{N}_0}$ and the Céa-Lemma~\eqref{eq:cea} imply a priori convergence $\lim\limits_{\ell \rightarrow \infty} \phi_{\ell}^\star \eqqcolon \phi_\infty^\star \in H^{-1/2}(\Gamma)$ and therefore, 
With Lemma~\ref{lem:SZconv}, we define $g_\infty \coloneqq (P_\infty(u^\star - u_\infty^\star))|_\Gamma$, where $u_\infty^\star \coloneqq \widetilde V\phi_\infty^\star$. 
With $H^{1/2}(\Gamma)$-stability of $J_\ell$ from~\cite{Feischl2014, Aurada2015}, continuity of $V$, continuity of the trace operator, and~\eqref{eq:apriori:phi}, we obtain
\begin{align*}
 \lVert g_{\infty} - g_{\ell}\rVert_{H^{1/2}(\Gamma)} 
 &\leq 
 \lVert g_{\infty} - J_{\ell}(g - V\phi_\infty^\star)\rVert_{H^{1/2}(\Gamma)} 
 + \lVert J_{\ell}(g - V\phi_\infty^\star) - g_{\ell}\rVert_{H^{1/2}(\Gamma)} 
 \\&
 = 
 \lVert ((P_{\infty} - P_{\ell})(u^\star - u_{\infty}^\star))|_\Gamma\rVert_{H^{1/2}(\Gamma)} 
 + \lVert J_{\ell}V(\phi_\infty^\star - \phi_{\ell}^\star)\rVert_{H^{1/2}(\Gamma)} 
 \\&
 \lesssim  
 \lVert (P_{\infty} - P_{\ell})(u^\star - u_{\infty}^\star)\rVert_{H^{1}(\Omega)} 
 + \lVert \phi_\infty^\star - \phi_{\ell}^\star\rVert_{H^{-1/2}(\Gamma)} 
 \xrightarrow{\ell \to \infty} 0.
\end{align*}
This concludes the proof.
\end{proof}

The proof of Theorem~\ref{theorem:mainresult} requires the following theorem, which is proved in Section~\ref{section:extension} below but stated here to ease reading of the core arguments.

\begin{theorem}[\textsl{A~priori} convergent lifting of discretized boundary residual] \label{thm:extconv}
 Under the assumptions of Section~\ref{section:mainresult}, there exists a sequence $(\widehat{g}_\ell)_{\ell \in \mathbb{N}_0}$ in $H^1(\Omega)$ with $\widehat{g}_{\ell} \in \mathcal{S}^q(\mathcal{T}_\ell^\omega)$ and $\widehat{g}_{\ell}|_{\Gamma} = g_\ell$ and $\widehat{g}_{\ell|_{\Gamma_\ell^{\mathsf{c}}}} = 0$ for all $\ell \in \mathbb{N}_0$ as well as a limit function $\widehat{g}_{\infty} \in H^1(\Omega)$ such that 
 \begin{equation} \label{eq:extconv}
  \lim_{\ell \rightarrow \infty} \lVert \nabla(\widehat{g}_{\infty} - \widehat{g}_{\ell})\rVert_{L^2(\Omega)} 
  = 
  0.
 \end{equation}
 Moreover, there exists a constant $\widetilde{C}_{\textnormal{eff}} > 0$ depending only on the dimension $d$ and $\kappa$-shape regularity of $\mathcal{T}_\ell$ such that 
 \begin{equation} \label{eq:exteff}
  \lVert \nabla \widehat{g}_\ell\rVert_{L^2(\Omega)} 
  \leq 
  \widetilde{C}_{\textnormal{eff}} \, \lVert \phi^\star - \phi_\ell^\star\rVert_{H^{-1/2}(\Omega)} 
  \quad \text{for all } \ell \in \mathbb{N}_0.
 \end{equation}
\end{theorem}

Theorem~\ref{thm:extconv} is a crucial ingredient in the proofs of the following results stating efficiency of the functional error estimates and \textsl{a~priori} convergence of the auxiliary solutions.

\begin{proposition}[Efficiency of functional error estimator]\label{prop:efficiency}
Under the assumptions of Section~\ref{section:mainresult}, there exists a constant $C_{\textnormal{eff}} > 0$ depending only on $d$ and $\kappa$-shape regularity of $\mathcal{T}_\ell$ such that
\begin{equation} \label{eq:efficiency}
 \eta_\ell 
 = 
 \lVert \nabla w_\ell^\star\rVert_{L^2(\Omega)} 
 \leq C_{\textnormal{eff}} \, \lVert \phi^\star - \phi_\ell^\star\rVert_{H^{-1/2}(\Gamma)}.
\end{equation}
\end{proposition}

\begin{proof}
Let $(\widehat{g}_\ell)_{\ell \in \mathbb{N}_0}$ be the sequence from Theorem~\ref{thm:extconv}. 
Since $\widehat{g}_\ell \in \mathcal{S}^q(\mathcal{T}_\ell^\omega)$, $\widehat{g}_{\ell}|_{\Gamma} = g_\ell$, and $\widehat{g}_{\ell}|_{\Gamma_\ell^{\mathsf{c}}} = 0$, we have that $w_\ell^\star = w_{\ell,0}^\star + \widehat{g}_{\ell}$, where $w_{\ell,0}^\star \in \mathcal{S}^q_{\partial \omega_\ell}(\mathcal{T}_\ell^\omega)$ is the unique solution to 
\begin{equation} \label{eq:dirichletzeroprob}
 \langle \nabla w_{\ell,0}^\star \:,\: \nabla v_\ell\rangle_{L^2(\omega_\ell)} 
 = 
 -\langle \nabla \widehat{g}_{\ell} \:,\: \nabla v_\ell\rangle_{L^2(\omega_\ell)} \quad 
 \text{for all } v_\ell \in \mathcal{S}^q_{\partial \omega_\ell}(\mathcal{T}_\ell^\omega).
\end{equation}
Choosing $v_\ell = w_{\ell,0}^\star$ and extending $w_{\ell,0}^\star$ and $\widehat{g}_\ell$ by zero to the whole domain $\Omega$, we obtain
\begin{equation*}
 \lVert \nabla w_{\ell,0}^\star\rVert_{L^2(\Omega)}^2
 =
 \lVert \nabla w_{\ell,0}^\star\rVert_{L^2(\omega_\ell)}^2 
 = 
 -\langle \nabla \widehat{g}_{\ell} \:,\: \nabla w_{\ell,0}^\star\rangle_{L^2(\omega_\ell)} 
 \leq 
 \lVert \nabla \widehat{g}_{\ell}\rVert_{L^2(\Omega)} \lVert \nabla w_{\ell,0}^\star\rVert_{L^2(\Omega)}.
\end{equation*}
Due to Theorem~\ref{thm:extconv}, this implies
\begin{equation*}
 \begin{split}
  \lVert \nabla w_\ell^\star\rVert_{L^2(\Omega)} 
  \leq 
  \lVert \nabla w_{\ell,0}^\star\rVert_{L^2(\Omega)} + \lVert \nabla\widehat{g}_{\ell}\rVert_{L^2(\Omega)} 
  \leq 
  2 \, \lVert \nabla\widehat{g}_{\ell}\rVert_{L^2(\Omega)} 
  \stackrel{\eqref{eq:exteff}}{\leq} 
  2\widetilde{C}_{\textnormal{eff}} \, \lVert \phi^\star - \phi_\ell^\star\rVert_{H^{-1/2}(\Gamma)}.
 \end{split}
\end{equation*}
This concludes~\eqref{eq:efficiency} with $C_{\textnormal{eff}} = 2 \widetilde{C}_{\textnormal{eff}}$.
\end{proof}

\begin{proposition}[\textsl{A~priori} convergence of functional upper bound]\label{prop:auxconv}
Under the assumptions of Section~\ref{section:mainresult}, there exists $w_\infty^\star \in H^1(\Omega)$ such that
\begin{equation} \label{eq:auxconv}
 \lim_{\ell \to \infty} \lVert \nabla(w_\infty^\star - w_\ell^\star)\rVert_{L^2(\Omega)} 
 = 
 0.
\end{equation}
\end{proposition}

\begin{proof}
For $\ell \in \mathbb{N}_0$, let $\widehat{g}_\ell$ be the extension of $g_\ell$ from Theorem~\ref{thm:extconv}. 
Then, it follows that $w_\ell^\star = w_{\ell,0}^\star + \widehat{g}_\ell$, where $w_{\ell,0}^\star \in \mathcal{S}^q_{\partial \omega_\ell}(\mathcal{T}_\ell^\omega)$ solves~\eqref{eq:dirichletzeroprob} and satisfies $\lVert \nabla w_{\ell,0}^\star\rVert_{L^2(\Omega)} \leq \lVert \nabla \widehat{g}_{\ell}\rVert_{L^2(\Omega)}$.

\textbf{Step~1 (Problem formulation on limit space)}. 
For $\ell \in \mathbb{N}_0$, we define the space
\begin{equation*}
 \mathcal{X}_\ell 
 \coloneqq 
 \bigcup_{m \geq \ell} \{v_m \in \mathcal{S}^q_{\Gamma}(\mathcal{T}_m) \:\vert\: \mathrm{supp}(v_m) \subseteq \overline{\omega_\ell}\} \subseteq H^1_{\partial \omega_\ell}(\omega_\ell) \subseteq H_\Gamma^1(\Omega).
\end{equation*}
Since $\mathcal{X}_\ell$ is a union of nested subspaces, $\mathcal{X}_\ell\subseteq H_0^1(\Omega)$ is again a subspace. 
We note that $\mathcal{X}_\ell \supseteq \mathcal{X}_{\ell + 1}$ since $\omega_\ell \supseteq \omega_{\ell + 1}$. 
Hence, we may define the limit space
\begin{equation*}
 \mathcal{X}_\infty 
 \coloneqq 
 \bigcap_{\ell \in \mathbb{N}_0} \overline{\mathcal{X}_\ell}
 \subseteq
 H^1_{\partial \omega_j}(\omega_j)
 \subseteq
 H_{\Gamma}^1(\Omega)
 \quad \text{for all } j \in \mathbb{N}_0,
\end{equation*}
where the closures are taken in $H^1_{\partial \omega_\ell}(\omega_\ell)$. 
As an intersection of closed subspaces, $\mathcal{X}_\infty \subseteq H_{\Gamma}^1(\Omega)$ is a closed subspace itself. 
Hence, there exists a unique $w_{\infty,0}^\star \in \mathcal{X}_\infty$ such that
\begin{equation} \label{eq:infprob}
 \langle \nabla w_{\infty,0}^\star \:,\: \nabla v_\infty\rangle_{L^2(\Omega)} 
 = 
 -\langle \nabla \widehat{g}_\infty \:,\: \nabla v_\infty\rangle_{L^2(\Omega)} 
 \quad \text{for all } v_\infty \in \mathcal{X}_\infty.
\end{equation}

\textbf{Step~2 (Weak convergence of certain subsequences)}.
Since $(\widehat{g}_{\ell})_{\ell \in \mathbb{N}_0}$ is a convergent sequence in $H^1(\Omega)$, there is a constant $C > 0$ such that $\lVert \nabla w_{\ell,0}^\star\rVert_{L^2(\Omega)} \leq \lVert \nabla \widehat{g}_{\ell}\rVert_{L^2(\Omega)} \leq C$ for all $\ell \in \mathbb{N}_0$. 
Let $(w_{\ell_k,0}^\star)_{k \in N_0}$ be an arbitrary subsequence of $(w_{\ell,0}^\star)_{\ell \in \mathbb{N}_0}$. 
Since $(w_{\ell_k,0}^\star)_{k \in \mathbb{N}_0}$ is also bounded in $H_{\Gamma}^1(\Omega)$, the Banach--Alaoglu theorem yields a further subsequence $(w_{\ell_{k_j},0}^\star)_{j \in \mathbb{N}_0}$ which converges weakly to some $w_{\infty,0} \in H_{\Gamma}^1(\Omega)$. 
According to Mazur's theorem, the closed subspace $\mathcal{X}_\infty$ is also weakly closed and hence contains the limit $w_{\infty,0} \in \mathcal{X}_\infty$.

\textbf{Step~3 (Proof of \boldsymbol{$w_{\infty,0} = w_{\infty,0}^\star$})}. 
We want to show that $w_{\infty,0} = w_{\infty,0}^\star$. 
To ease notation, we abbreviate $w_{j,0}^\star \coloneqq w_{\ell_{k_j},0}^\star$ for all $j \in \mathbb{N}_0$. 
Let $\varepsilon > 0$ and $v_\infty \in \mathcal{X}_\infty$ be arbitrary but fixed. 
We denote by $R_j \colon H^1_{\partial \omega_j}(\omega_j) \to \mathcal{S}^q_{\partial \omega_j}(\mathcal{T}_j^\omega)$ the Ritz projector defined via
\begin{equation} \label{eq:orthproj}
 \langle \nabla (R_j v) \:,\: \nabla v_j\rangle_{L^2(\omega_j)} 
 = 
 \langle \nabla v \:,\: \nabla v_j\rangle_{L^2(\omega_j)} 
 \quad \text{for all } v \in H^1_{\partial \omega_j}(\omega_j) \text{ and all } v_j \in \mathcal{S}^q_{\partial \omega_j}(\mathcal{T}_j^\omega).
\end{equation}
Being the orthogonal projection with respect to the $H_{\partial \omega_j}^1(\omega_j)$-inner product, $R_j$ satisfies
\begin{equation} \label{eq:bestapprox}
 \lVert \nabla(v - R_j v)\rVert_{L^2(\omega_j)} 
 = 
 \min_{v_j \in \mathcal{S}^q_{\partial \omega_j}(\mathcal{T}_j^\omega)} \lVert \nabla(v - v_j)\rVert_{L^2(\omega_j)} 
 \quad \text{for all } v \in H^1_{\partial \omega_j}(\omega_j).
\end{equation}
For every $j \in \mathbb{N}_0$, we obtain
\begin{equation} \label{eq:probrep}
 \begin{split}
  &\lvert \langle \nabla(w_{\infty,0} + \widehat{g}_\infty) \:,\: \nabla v_\infty\rangle_{L^2(\omega_\infty)}\rvert \\
  &\qquad \leq 
  \lvert \langle \nabla(w_{j,0}^\star + \widehat{g}_j) \:,\: \nabla (R_j v_\infty)\rangle_{L^2(\omega_j)}\rvert 
  + \lvert \langle \nabla(w_{\infty,0} - w_{j,0}^\star) \:,\: \nabla v_\infty\rangle_{L^2(\omega_j)}\rvert \\
  &\qquad + \lvert \langle \nabla \widehat{g}_\infty - \nabla \widehat{g}_j \:,\: \nabla v_\infty\rangle_{L^2(\omega_j)}\rvert 
  + \lvert \langle \nabla w_{j,0}^\star \:,\: \nabla (v_\infty - R_j v_\infty)\rangle_{L^2(\omega_j)}\rvert \\
  &\qquad + \lvert \langle \nabla\widehat{g}_j \:,\: \nabla (v_\infty - R_j v_\infty)\rangle_{L^2(\omega_j)}\rvert.
 \end{split}
\end{equation}
The first summand on the right-hand side of~\eqref{eq:probrep} equals zero by~\eqref{eq:dirichletzeroprob}. 
Since $\widehat{g}_j \to \widehat{g}_\infty$ strongly in $H^1(\Omega)$ and $w_{j,0}^\star \rightharpoonup w_{\infty,0}$ weakly in $H^1_{\Gamma}(\Omega)$, there exists $j_0 \in N_0$ such that 
\begin{equation} \label{eq:strongweakesti}
 \lvert \langle \nabla(w_{\infty,0} - w_{j,0}^\star) \:,\: \nabla v_\infty\rangle_{L^2(\Omega)}\rvert 
 + 
 \lvert \langle \nabla \widehat{g}_\infty - \nabla \widehat{g}_j \:,\: \nabla v_\infty\rangle_{L^2(\Omega)}\rvert 
 \leq
 \varepsilon
 \quad \text{for all } j \geq j_0.
\end{equation}
The orthogonality~\eqref{eq:orthproj} of the Ritz projector and the fact that $w_{j,0}^\star \in \mathcal{S}_{\partial \omega_j}^q(\mathcal{T}_j^\omega)$ imply $\langle \nabla w_{j,0}^\star \:,\: \nabla(v_\infty - R_j v_\infty)\rangle_{L^2(\omega_j)} = 0$. 
It remains to estimate the last summand of $\eqref{eq:probrep}$. 
Let $m \geq j \geq j_0$ and consider the Scott--Zhang projector $P_j \colon H^1_{\Gamma}(\Omega) \to \mathcal{S}_{\Gamma}^q(\mathcal{T}_{j})$. 
Due to Lemma~\ref{lem:SZrest} and since $\mathrm{supp}(v_\infty) \subseteq \overline{\omega_j}$, there holds $\mathrm{supp}(P_jv_\infty) \subseteq \overline{\omega_j}$ and hence $P_jv_\infty \in \mathcal{S}_{\partial \omega_j}^q(\mathcal{T}_j^\omega)$. 
We thus obtain
\begin{equation} \label{eq:RitzSZ}
 \begin{split}
  \lVert \nabla(v_\infty - R_j v_\infty)\rVert_{L^2(\omega_j)} 
  &= 
  \min_{v_j \in \mathcal{S}_{\partial \omega_j}^q(\mathcal{T}_j^\omega)} \lVert \nabla(v_\infty - v_j)\rVert_{L^2(\omega_j)} \\
  &\leq 
  \lVert \nabla(v_\infty - P_m v_\infty)\rVert_{L^2(\omega_j)} \\
  &\qquad+ \min_{v_j \in \mathcal{S}_{\partial \omega_j}^q(\mathcal{T}_j^\omega)} \lVert \nabla(P_m v_\infty - v_j)\rVert_{L^2(\omega_j)} \\
  &\leq
  \lVert \nabla(v_\infty - P_m v_\infty)\rVert_{L^2(\Omega)} + \lVert \nabla(P_m - P_j)v_\infty\rVert_{L^2(\Omega)}.
 \end{split}
\end{equation}
We first consider the second term of the right-hand side of~\eqref{eq:RitzSZ}. 
By Lemma~\ref{lem:SZconv}, the sequence $(P_m v_\infty)_{m \in \mathbb{N}_0}$ is a Cauchy-sequence. 
This yields the existence of $j_1 \in \mathbb{N}_0$ such that 
\begin{equation*}
 \lVert \nabla(P_m - P_j)v_\infty\rVert_{L^2(\Omega)} 
 \leq
 \varepsilon
 \quad \text{for all } m \geq j \geq j_1.
\end{equation*}
For the first summand on the right-hand side of~\eqref{eq:RitzSZ}, we again refer to Lemma~\ref{lem:SZrest} and note that because of $\mathrm{supp}(v_\infty) \subseteq \overline{\omega_m}$ for all $m \in \mathbb{N}_0$, it follows that 
\begin{equation*}
 \mathrm{supp}(P_mv_\infty) 
 \subseteq 
 \overline{\omega_m} 
 \subset 
 \overline{\omega_j}.
\end{equation*}
Hence, Lemma~\ref{lem:SZ} yields
\begin{equation*}
 \lVert \nabla(v_\infty - P_m v_\infty)\rVert_{L^2(\Omega)} 
 =
 \lVert \nabla(v_\infty - P_m v_\infty)\rVert_{L^2(\omega_j)}
 \leq 
 C_{\textnormal{stab}}\min_{\substack{v_m^j \in \mathcal{S}_0^q(\mathcal{T}_m) \\ \mathrm{supp}(v_m^j) \subseteq \omega_j}} \lVert \nabla(v_\infty - v_m^j)\rVert_{L^2(\omega_j)}.
\end{equation*}
For given $j \geq j_1$, since $v_\infty \in \mathcal{X}_\infty \subset \overline{\mathcal{X}_m}$, there exists $m_j \geq m$ and $v_{m_j}^j \in \mathcal{S}_{\Gamma}^q(\mathcal{T}_{m_j})$ with $\mathrm{supp}(v_{m_j}^j) \subseteq \overline{\omega_j}$ such that
\begin{equation*}
 \lVert \nabla(v_\infty - v_{m_j}^j)\rVert_{L^2(\omega_j)} 
 =
 \lVert \nabla(v_\infty - v_{m_j}^j)\rVert_{L^2(\Omega)}
 \leq 
 \varepsilon.
\end{equation*}
This yields
\begin{equation*}
 \lVert \nabla(v_\infty - R_jv_\infty)\rVert_{L^2(\omega_j)} 
 \leq 
 (1 + C_{\textnormal{stab}})\varepsilon 
 \quad \text{for all } j \geq j_1.
\end{equation*}
In conclusion, we have that 
\begin{equation*}
 \lvert \langle \nabla \widehat{g}_{j} \:,\: \nabla(v_\infty - R_mv_\infty)\rangle_{L^2(\omega_{j})}\rvert 
 \leq 
 (1 + C_{\textnormal{stab}})\varepsilon \sup_{\ell \in \mathbb{N}_0} \lVert \nabla \widehat{g}_{\ell}\rVert_{L^2(\Omega)} 
 \quad \text{for all } j \geq j_1.
\end{equation*}
Together with~\eqref{eq:probrep}--\eqref{eq:strongweakesti} and the fact that $\lVert \nabla \widehat{g}_{\ell}\rVert_{L^2(\Omega)} \lesssim 1$, this yields
\begin{equation*}
 \lvert \langle \nabla(w_{\infty,0}^\star + \widehat{g}_\infty) \:,\: \nabla v_\infty\rangle_{L^2(\omega_\infty)}\rvert 
 \lesssim 
 \varepsilon.
\end{equation*} 
Since $\varepsilon > 0$ and $v_\infty \in \mathcal{X}_\infty$ were arbitrary, we conclude that the right-hand side of~\eqref{eq:probrep} vanishes and hence
\begin{equation*}
 \langle \nabla w_{\infty,0} \:,\: \nabla v_\infty\rangle_{L^2(\Omega)} 
 = 
 -\langle \nabla \widehat{g}_{\infty} \:,\: \nabla v_\infty\rangle_{L^2(\Omega)}
 \quad \text{for all } v_\infty \in \mathcal{X}_\infty.
\end{equation*}
Since~\eqref{eq:infprob} has a unique solution, we conclude $w_{\infty,0} = w_{\infty,0}^\star$.

\textbf{Step~4 (Weak convergence of full sequence)}. Step~2 and~3 prove that any subsequence $(w_{\ell_k,0}^\star)_{k \in \mathbb{N}_0}$ of $(w_{\ell,0}^{\star})_{\ell \in \mathbb{N}_0}$ admits a weakly convergent subsequence $(w_{\ell_{k_j},0}^\star)_{j \in \mathbb{N}_0}$ which converges to the unique limit $w_{\infty,0}^\star$. 
By elementary calculus, it follows that indeed the full sequence $(w_{\ell,0}^\star)_{\ell \in \mathbb{N}_0}$ satisfies weak convergence $w_{\ell,0}^\star \rightharpoonup w_{\infty,0}^\star$ in $H^1_{\Gamma}(\Omega)$.

\textbf{Step~5 (Strong convergence of full sequence)}. Since $\nabla \widehat{g}_{\ell} \to \nabla \widehat{g}_{\infty}$ in $L^2(\Omega)$ and $\nabla w_{\ell,0}^\star \rightharpoonup \nabla w_{\infty,0}^\star$ in $L^2(\Omega)$, it follows that
\begin{equation*}
 \begin{split}
 \lVert \nabla w_{\ell,0}^\star\rVert_{L^2(\Omega)}^2 
 &\stackrel{\eqref{eq:dirichletzeroprob}}{=} 
 -\langle \nabla \widehat{g}_{\ell} \:,\: \nabla w_{\ell,0}^\star\rangle_{L^2(\omega_\ell)}
 \xrightarrow{\ell \to \infty} 
 -\langle \nabla \widehat{g}_\infty \:,\: \nabla w_{\infty,0}^\star\rangle_{L^2(\Omega)} 
 \stackrel{\eqref{eq:infprob}}{=} 
 \lVert \nabla w_{\infty,0}^\star\rVert_{L^2(\Omega)}^2.
\end{split}
\end{equation*}
Recall that $\lVert \nabla(\cdot)\rVert_{L^2(\Omega)}$ is an equivalent norm on $H_{\Gamma}^1(\Omega)$. 
Together with weak convergence $w_{\ell,0}^\star \rightharpoonup w_{\infty,0}^\star$ in $H_{\Gamma}^1(\Omega)$, convergence of the norm suffices for strong convergence
\begin{equation*}
 \lim_{\ell \to \infty} \lVert \nabla (w_{\infty,0}^\star - w_{\ell,0}^\star)\rVert_{L^2(\Omega)} 
 = 
 0.
\end{equation*}
In particular, Theorem~\ref{thm:extconv} yields strong convergence
\begin{equation*}
 \nabla w_{\ell}^\star 
 = \nabla w_{\ell,0}^\star + \nabla \widehat{g}_{\ell} 
 \xrightarrow{\ell \to \infty} 
 \nabla w_{\infty,0}^\star + \nabla \widehat{g}_{\infty} 
 \eqqcolon \nabla w_{\infty}^\star
 \quad \text{in } L^2(\Omega).
\end{equation*}
This concludes the proof.
\end{proof}

Recall the notation and assumptions from Section~\ref{section:mainresult}.
As in~\cite{Morin2008, Siebert2011, Gantner2021}, we define the set $\mathcal{T}_\infty$ of all elements that remain unrefined after finitely many steps of Algorithm~\ref{algo:adap}, i.e.,
\begin{equation} \label{eq:Tinfty}
 \mathcal{T}_{\infty} 
 \coloneqq 
 \bigcup_{\ell' \in \mathbb{N}_0} \bigcap_{\ell \geq \ell'} \mathcal{T}_{\ell}.
\end{equation}
%
%and abbreviate $\mathcal{T}_{\ell,\infty} \coloneqq \mathcal{T}_{\ell} \cap \mathcal{T}_{\infty}$. \\ \\
%
%We begin by collecting useful results. 
With Corollary~\ref{cor:gellconv} and Proposition~\ref{prop:auxconv}, we can adapt arguments of~\cite{Gantner2021} to derive the first convergence statement of Theorem~\ref{theorem:mainresult}, i.e., $\mathrm{osc}_\ell \to 0$ in~\eqref{eq:theorem:mainresult}.
\begin{proposition} \label{thm:eta2conv}
For any $\ell' \in \mathbb{N}_0$, Algorithm~\ref{algo:adap} guarantees
\begin{align} \label{eq:oscconv}
 \lim_{\ell' \leq \ell \rightarrow \infty} \eta_{\ell}(\mathcal{T}_{\ell'} \cap \mathcal{T}_\infty) 
 = 
 0
 \quad \text{as well as} \quad
 \lim_{\ell \rightarrow \infty} \mathrm{osc}_{\ell} 
 = 
 0.
\end{align}
\end{proposition}
\begin{proof}
The proof is split into six steps.

\textbf{Step~1 (Stability of $\boldsymbol{\mathrm{osc}_\ell}$).}
We show that, for all $k, \ell \in \mathbb{N}_0$ with $k > \ell$, 
\begin{subequations}\label{eq:stability:osc}

\begin{equation}
 \mathrm{osc}_{k}(\mathcal{T}_k \cap \mathcal{T}_{\ell}) 
 \leq 
 \mathrm{osc}_{\ell}(\mathcal{T}_k \cap \mathcal{T}_{\ell}) + a_{\ell,k},
\end{equation}
where 
\begin{equation}
 a_{\ell,k} 
 \coloneqq 
 \lVert h_k^{1/2} \nabla_\Gamma \bigl( (1-J_{\ell})(g - V\phi_{\ell}^\star) - (1-J_k)(g-V\phi_k^\star) \bigr)\rVert_{L^2(\Gamma)}
 \xrightarrow{k,\ell \to \infty} 0.
\end{equation}
\end{subequations}
Let $\ell,k \in \mathbb{N}_0$ with $k > \ell$.
The triangle inequality proves that
\begin{align*}
 \mathrm{osc}_{k}(\mathcal{T}_k \cap \mathcal{T}_{\ell}) 
 &= 
 \Bigl( \sum_{T \in \mathcal{T}_k \cap \mathcal{T}_{\ell}} \lVert h_k^{1/2} \nabla_\Gamma (1 - J_k)(g - V \phi_k^\star)\rVert_{L^2(\partial T \cap \Gamma)}^2\Bigr)^{1/2} 
 \\&
 \leq 
 \Bigl( \sum_{T \in \mathcal{T}_k \cap \mathcal{T}_{\ell}} \lVert h_k^{1/2} \nabla_\Gamma (1 - J_\ell)(g - V \phi_\ell^\star)\rVert_{L^2(\partial T \cap \Gamma)}^2 \Bigr)^{1/2} 
 \\& \quad
 + \Bigl( \sum_{T \in \mathcal{T}_k \cap \mathcal{T}_{\ell}} \lVert h_k^{1/2} \nabla_\Gamma \bigl( (1 - J_\ell)(g - V \phi_\ell^\star) - (1 - J_k)(g - V \phi_k^\star)\rVert_{L^2(\partial T \cap \Gamma)}^2 \Bigr)^{1/2} 
 \\&
 \leq 
 \mathrm{osc}_{\ell}(\mathcal{T}_k \cap \mathcal{T}_{\ell}) + a_{\ell,k}.
\end{align*}
We recall the inverse estimates
\begin{align}\label{eq:invest:1/2}
 \lVert h_k^{1/2} \nabla_\Gamma v_k\rVert_{L^2(\Gamma)}
 \lesssim \lVert v_k\rVert_{H^{1/2}(\Gamma)}
 \quad \text{for all } v_k \in \mathcal{S}^q(\mathcal{T}_k^\Gamma)
\end{align}
as well as 
\begin{align}\label{eq:invest:V}
 \lVert h_k^{1/2} \nabla_\Gamma V \psi_k\rVert_{L^2(\Gamma)}
 \lesssim \lVert \psi_k\rVert_{H^{-1/2}(\Gamma)}
 \quad \text{for all } \psi_k \in \mathcal{P}^p(\mathcal{T}_k^\Gamma),
\end{align}
where the hidden constants depend only on the polynomial degrees $p \in \mathbb{N}_0$, $q \in \mathbb{N}$, and $\kappa$-shape regularity of $\mathcal{T}_k^\Gamma$; see~\cite[Proposition~5]{Aurada2015} and~\cite[Theorem~3.1]{Aurada2015a}, respectively.
With this and nestedness of the discrete spaces, we see that
\begin{align*}
 a_{\ell,k}
 &\le 
 \lVert h_k^{1/2} \nabla_\Gamma (J_{\ell}-J_k)g \rVert_{L^2(\Gamma)}
 + 
 \lVert h_k^{1/2} \nabla_\Gamma ( V\phi_{\ell}^\star - V\phi_k^\star )\rVert_{L^2(\Gamma)}
 \\& \quad
 + 
 \lVert h_k^{1/2} \nabla_\Gamma ( J_{\ell} V\phi_{\ell}^\star - J_k V\phi_k^\star )\rVert_{L^2(\Gamma)}
 \\&
 \lesssim 
 \lVert (J_{\ell}-J_k)g \rVert_{H^{1/2}(\Gamma)}
 + 
 \lVert \phi_{\ell}^\star - \phi_k^\star\rVert_{H^{-1/2}(\Gamma)}
 + 
 \lVert J_{\ell} V\phi_{\ell}^\star - J_k V\phi_k^\star\rVert_{H^{1/2}(\Gamma)}.
\end{align*}
Lemma~\ref{lem:SZconv} and continuity of the trace operator yield
\begin{equation*}
 \lVert (J_\ell-J_{k})g\rVert_{H^{1/2}(\Gamma)} 
 \lesssim 
 \lVert (P_\ell - P_{k})u^\star\rVert_{H^1(\Omega)} 
 \xrightarrow{k, \ell \to \infty} 0.
\end{equation*}
Moreover, the $H^{1/2}(\Gamma)$-stability of $J_{\ell}$, Lemma~\ref{lem:SZconv} for $u_\infty = \widetilde V\phi_\infty^\star$, and \textsl{a~priori} convergence~\eqref{eq:apriori:phi} together with continuity of $V$ yield
\begin{align*}
 &\lVert J_{\ell} V\phi_{\ell}^\star - J_{\infty} V\phi_\infty^\star\rVert_{H^{1/2}(\Gamma)} 
 \leq 
 \lVert J_{\ell} V\phi_{\ell}^\star - J_{\ell} V\phi_\infty^\star\rVert_{H^{1/2}(\Gamma)} 
 + 
 \lVert J_{\ell} V\phi_\infty^\star - J_{\infty} V\phi_\infty^\star\rVert_{H^{1/2}(\Gamma)} 
 \\& \qquad
 \lesssim 
 \lVert V\phi_{\ell}^\star - V\phi_\infty^\star\rVert_{H^{1/2}(\Gamma)}  
 + \lVert P_{\ell} u_{\infty} - P_{\infty} u_{\infty}\rVert_{H^{1}(\Omega)} 
  \xrightarrow{\ell \to \infty} 0.
\end{align*}
Altogether, we thus conclude that $a_{\ell,k} \to 0$ as $k,\ell \to \infty$.

\textbf{Step~2 (Reduction of $\boldsymbol{\mathrm{osc}_\ell}$).}
We show that, for all $k, \ell \in \mathbb{N}_0$ with $k > \ell$,
\begin{align}\label{eq:reduction:osc}
 \mathrm{osc}_{k}(\mathcal{T}_k \setminus  \mathcal{T}_{\ell})
 \le 
 \rho_{\textnormal{red}} \, \mathrm{osc}_{\ell}(\mathcal{T}_{\ell} \setminus \mathcal{T}_k) + a_{\ell,k}
 \quad \text{with} \quad
 \rho_{\textnormal{red}} \coloneqq 2^{-1/(2d)}.
\end{align}
To this end, we note that $\bigcup (\mathcal{T}_k \setminus \mathcal{T}_{\ell}) = \bigcup (\mathcal{T}_{\ell} \setminus \mathcal{T}_k)$ and that for $T \in \mathcal{T}_k \setminus \mathcal{T}_{\ell}$ and $T' \in \mathcal{T}_{\ell} \setminus \mathcal{T}_k$ with $T \subset T'$, there holds $h_T \leq 2^{-1/d} h_{T'}$ due to bisection and $h_T = |T|^{1/d}$.
Arguing as for stability~\eqref{eq:stability:osc}, we obtain
\begin{align*}
 \mathrm{osc}_{k}(\mathcal{T}_k \setminus  \mathcal{T}_{\ell}) 
 &= 
 \Bigl(\sum_{T \in \mathcal{T}_{k} \setminus \mathcal{T}_{\ell}} h_T \lVert \nabla_\Gamma (1 - J_k)(g - V \phi_k^\star)\rVert_{L^2(\partial T \cap \Gamma)}^2\Bigr)^{1/2} 
 \\&
 = 
 \Bigl(\sum_{T' \in \mathcal{T}_{\ell} \setminus \mathcal{T}_k} \sum_{\substack{T \in \mathcal{T}_k \setminus \mathcal{T}_{\ell} \\ T \subset T'}} h_T \lVert \nabla_\Gamma (1 - J_k)(g - V \phi_k^\star)\rVert_{L^2(\partial T \cap \Gamma)}^2\Bigr)^{1/2}
 \\&
 \leq 
 \Bigl(\sum_{T' \in \mathcal{T}_{\ell} \setminus \mathcal{T}_k} \sum_{\substack{T \in \mathcal{T}_k \setminus \mathcal{T}_{\ell} \\ T \subset T'}} h_T \lVert \nabla_\Gamma (1 - J_{\ell})(g - V \phi_{\ell}^\star)\rVert_{L^2(\partial T \cap \Gamma)}^2 
 \Bigr)^{1/2} + a_{\ell,k}
 \\&
 \leq 
 2^{-1/(2d)} \Bigl(\sum_{T' \in \mathcal{T}_{\ell} \setminus \mathcal{T}_k} h_{T'} \sum_{\substack{T \in \mathcal{T}_k \setminus \mathcal{T}_{\ell} \\ T \subset T'}}  \lVert \nabla_\Gamma (1 - J_{\ell})(g - V \phi_{\ell}^\star)\rVert_{L^2(\partial T \cap \Gamma)}^2 \Bigr)^{1/2} + a_{\ell,k} 
 \\&
 = 2^{-1/(2d)} \, \mathrm{osc}_{\ell}(\mathcal{T}_{\ell} \setminus \mathcal{T}_k) + a_{\ell,k}.
\end{align*}

\textbf{Step~3 (Convergence $\boldsymbol{\mathrm{osc}_{\ell_k}(\mathcal{M}_{\ell_k}) \to 0}$ along a subsequence).}
With reduction~\eqref{eq:reduction:osc}, the arguments from~\cite[Proof of Theorem~3.1, step~1--2]{Gantner2021} prove that there exists a subsequence $(\mathrm{osc}_{\ell_k})_{k \in \mathbb{N}_0}$ such that
\begin{equation} \label{eq:eta2fineconv}
 0 
 \le 
 \mathrm{osc}_{\ell_k}(\mathcal{M}_{\ell_k}) 
 \le 
 \mathrm{osc}_{\ell_k}(\mathcal{T}_{\ell_k} \backslash \mathcal{T}_{\ell_{k+1}}) 
 \xrightarrow{k \to \infty} 0.
\end{equation}

\textbf{Step~4 (Convergence $\boldsymbol{\max_{T \in \mathcal{M}_{\ell}} \eta_{\ell}(T) \to 0}$ along the full sequence).}
For $\ell \in \mathbb{N}_0$, let $T_{\ell} \in \mathcal{M}_{\ell}$ such that $\eta_{\ell}(T_{\ell}) = \max_{T \in \mathcal{M}_{\ell}} \eta_{\ell}(T)$.
Since $\mathcal{M}_{\ell} \subseteq \mathcal{T}_{\ell} \setminus \mathcal{T}_{\infty}$, \cite[Corollary~4.1]{Morin2008} proves
\begin{equation*}
 |T_{\ell}| 
 = 
 \lVert h_{\ell}^d\rVert_{L^{\infty}(T_{\ell})} 
 \leq 
 \lVert h_{\ell}\rVert_{L^{\infty}(\bigcup (\mathcal{T}_{\ell} \setminus \mathcal{T}_{\infty}))}^d \stackrel{\ell \rightarrow \infty}{\longrightarrow} 0.
\end{equation*}
By Proposition~\ref{prop:auxconv}, there exists $w_\infty^\star \in H^1(\Omega)$ such that $\lVert w_\infty^\star - w_\ell^\star\rVert_{H^1(\Omega)} \xrightarrow{\ell \to \infty} 0$. Hence, absolute continuity of the $H^1$-seminorm with respect to the measure $|\cdot|$ concludes
\begin{equation} \label{eq:eta1fineconv}
 \max_{T \in \mathcal{M}_{\ell}}  \eta_{\ell}(T)^2 
 =
 \eta_{\ell}(T_{\ell})^2 
 = 
 \lVert \nabla w_\ell^\star\rVert_{L^2(T_{\ell})}^2 
 \lesssim 
 \lVert \nabla w_\infty^\star\rVert_{L^2(T_\ell)}^2 + \lVert \nabla(w_\infty^\star - w_\ell^\star)\rVert_{L^2(\Omega)}^2
 \xrightarrow{\ell \to \infty} 0.
\end{equation}

\textbf{Step~5 (Convergence $\boldsymbol{\mathrm{osc}_\ell \to 0}$ along the full sequence).}
From Step~3--4, we obtain 
\begin{equation*}
 \max_{T \in \mathcal{M}_{\ell_k}} \big[ \eta_{\ell_k}(T)^2 + \mathrm{osc}_{\ell_k}(T)^2 \big]
 \le \max_{T \in \mathcal{M}_{\ell_k}} \eta_{\ell_k}(T_{\ell_k})^2 + \mathrm{osc}_{\ell_k}(\mathcal{M}_{\ell_k})^2
 \xrightarrow{k \to \infty} 0.
\end{equation*}
For all $\ell' \leq \ell_k$ and every $T \in \mathcal{T}_{\ell'} \cap \mathcal{T}_{\infty} \subseteq \mathcal{T}_{\ell_k} \setminus \mathcal{M}_{\ell_k}$, the marking strategy~\eqref{eq:marking} yields that
\begin{equation} \label{eq:2mark}
 \begin{split}
 \mathrm{osc}_{\ell_k}(T) ^2
 &\leq 
 \max_{T \in \mathcal{T}_{\ell_k} \setminus \mathcal{M}_{\ell_k}} \big[ \eta_{\ell_k}(T)^2 + \mathrm{osc}_{\ell_k}(T)^2 \big] 
 \\
 &\stackrel{\eqref{eq:marking}}{\leq} 
 M\Bigl(\max_{T \in \mathcal{M}_{\ell_k}} \big[ \eta_{\ell_k}(T)^2 + \mathrm{osc}_{\ell_k}(T)^2 \big] \Bigr) 
 \stackrel{k \rightarrow \infty}{\longrightarrow} 0.
 \end{split}
\end{equation}
Since $\mathcal{T}_{\ell'} \cap \mathcal{T}_\infty$ is finite, this proves
\begin{equation} \label{eq:eta2nfineconv}
 \lim_{k \rightarrow \infty} \mathrm{osc}_{\ell_k}(\mathcal{T}_{\ell'} \cap \mathcal{T}_{\infty}) = 0 
 \qquad \text{for all } \ell' \in \mathbb{N}_0,
\end{equation}	
which verifies the claim of~\cite[Proof of Theorem~3.1, step~3]{Gantner2021} in the present setting.
With stability~\eqref{eq:stability:osc} and reduction~\eqref{eq:reduction:osc}, the arguments from~\cite[Proof of Theorem~3.1, step~4--5]{Gantner2021} conclude that indeed
\begin{equation} \label{eq:eta2conv}
 \lim_{\ell \rightarrow \infty} \mathrm{osc}_{\ell} = 0.
\end{equation}

\textbf{Step~6 (Convergence $\boldsymbol{\eta_\ell(\mathcal{T}_{\ell'} \cap \mathcal{T}_\infty) \to 0}$ along the full sequence).}
Let $\ell, \ell' \in \mathbb{N}_0$ with $\ell' \le \ell$.
Since marked elements are refined, it holds that $\mathcal{T}_{\ell'} \cap \mathcal{T}_{\infty} \subseteq \mathcal{T}_{\ell} \setminus \mathcal{M}_{\ell}$.
For all $T \in \mathcal{T}_{\ell'} \cap \mathcal{T}_{\infty}$, the convergences~\eqref{eq:eta1fineconv} and~\eqref{eq:eta2conv} prove
\begin{align*}
 0 \le \max_{T \in \mathcal{M}_{\ell}} \big[ \eta_{\ell}(T)^2 + \mathrm{osc}_{\ell}(T)^2 \big]
 \le \max_{T \in \mathcal{M}_{\ell}} \eta_{\ell}(T)^2 + \mathrm{osc}_{\ell}^2
 \xrightarrow{\ell \to \infty} 0.
\end{align*}
Thus, the marking strategy~\eqref{eq:marking} yields
\begin{equation*}
 \eta_{\ell}(T)^2 
 \leq 
 \max_{T \in \mathcal{T}_{\ell} \setminus \mathcal{M}_{\ell}} \big[ \eta_{\ell}(T)^2 + \mathrm{osc}_{\ell}(T)^2 \big]
 \stackrel{\eqref{eq:marking}}{\leq}
  M\Big(\max_{T \in \mathcal{M}_{\ell}} \big[ \eta_{\ell}(T)^2 + \mathrm{osc}_{\ell}(T)^2 \big] \Big)
  \stackrel{\ell \rightarrow \infty}{\longrightarrow} 0.
\end{equation*}
Since $\mathcal{T}_{\ell'} \cap \mathcal{T}_\infty$ is finite, this gives
\begin{equation*}
 \lim_{\ell' \leq \ell \rightarrow \infty} \eta_{\ell}(\mathcal{T}_{\ell'} \cap \mathcal{T}_\infty) = 0
\end{equation*}
and concludes the proof.
\end{proof}

\begin{lemma}[Strong reliability] \label{lem:strongrel}
For all $\psi \in L^2(\Gamma)$ and $\ell \in \mathbb{N}_0$, there holds
\begin{equation}\label{eq:lem:strongrel}
 |\langle V\phi^\star - V\phi_{\ell}^\star \:,\: \psi\rangle_{H^{1/2}(\Gamma) \times H^{-1/2}(\Gamma)}| 
 \le 
 C_{\textnormal{rel}}
 \Big( \mathrm{osc}_\ell \, \lVert \psi\rVert_{L^2(\Gamma)} +
 \sum_{T \in \mathcal{T}_{\ell}} \eta_{\ell}(T) \lVert \psi\rVert_{L^2(\partial T \cap \Gamma)} \Big).
\end{equation}
The constant $C_{\textnormal{rel}} > 0$ depends only on $d$, $\Gamma$, and $\kappa$-shape regularity of $\mathcal{T}_\ell$.
\end{lemma} 

\begin{proof}
The proof is split into three steps.

\textbf{Step~1:} 
Let $Q_\ell^\Gamma \colon L^2(\Gamma) \to \mathcal{P}^0(\mathcal{T}_\ell^\Gamma)$ and $Q_\ell^\Omega \colon L^2(\Omega) \to \mathcal{P}^0(\mathcal{T}_\ell)$ be the integral mean operators, i.e., the $L^2$-orthogonal projections onto the piecewise constants. 
For $F \in \mathcal{T}_\ell^\Gamma$ and $T_F \in \mathcal{T}_\ell$ with $F \subset \partial T$, the best approximation property of $Q_\ell^\Gamma$ with respect to $\lVert \cdot\rVert_{L^2(F)}$ yields
\begin{align*}
 \lVert (1-Q_\ell^\Gamma) J_{\ell}(g - V\phi_{\ell}^\star)\rVert_{L^2(F)}
 = \lVert (1-Q_\ell^\Gamma) w_\ell^\star\rVert_{L^2(F)}
 \le \lVert (1-Q_\ell^\Omega) w_\ell^\star\rVert_{L^2(F)}.
\end{align*}
Moreover, the trace inequality and a Poincar\'e inequality prove
\begin{align*}
 \begin{split}
 \lVert (1-Q_\ell^\Omega) w_\ell^\star\rVert_{L^2(F)} 
% \le
% \lVert (1-Q_\ell^\Omega) w_\ell^\star\rVert_{L^2(\partial T_F)} 
% \nonumber
% \\& \quad
 &\lesssim 
 h_{T_F}^{-1} \lVert (1-Q_\ell^\Omega) w_\ell^\star\rVert_{L^2(T_F)} + \lVert \nabla (1-Q_\ell^\Omega) w_\ell^\star\rVert_{L^2(T_F)}
 \\&
 \lesssim 
 \lVert \nabla w_\ell^\star\rVert_{L^2(T_F)}
 =
 \eta_\ell(T_F),
 \end{split}
\end{align*}%
where the hidden constants depend only on $d$ and $\kappa$-shape regularity of $\mathcal{T}_\ell$.

\textbf{Step~2:} With $Q_\ell^\Gamma$ from Step~1 and the local $L^2$-approximation property of the projector $J_{\ell}$, it holds that
\begin{align*}% \label{eq:rholegeta2}
 \begin{split}
 \lVert (1-Q_\ell^\Gamma)(1-J_{\ell})(g - V\phi_{\ell}^\star)\rVert_{L^2(\Gamma)}
 &\le 
 \lVert (1-J_{\ell})(g - V\phi_{\ell}^\star)\rVert_{L^2(\Gamma)}
 \\&
 \lesssim 
 \lVert h_{\ell} \nabla_\Gamma (1-J_{\ell})(g-V\phi_{\ell}^\star)\rVert_{L^2(\Gamma)}
 \lesssim \mathrm{osc}_{\ell};
 \end{split}
\end{align*}
see, e.g.,~\cite[Lemma~3]{Aurada2015}. The hidden constant depends only on $p$, $\Gamma$, and $\kappa$-shape regularity of $\mathcal{T}_\ell^\Gamma$.

\textbf{Step~3:}
Let $\psi \in L^2(\Gamma)$. Since the dual pairing in $H^{1/2}(\Gamma) \times H^{-1/2}(\Gamma)$ extends the $L^2(\Gamma)$-inner product and $\eta_{\ell}(T) = 0$ for all $T \in \mathcal{T}_{\ell} \setminus \mathcal{T}_{\ell}^{\omega}$, the Galerkin orthogonality and the symmetry of the $L^2$-orthogonal projection prove
\begin{align*}
 &\langle V\phi^\star - V\phi_{\ell}^\star \:,\: \psi\rangle_{H^{1/2}(\Gamma) \times H^{-1/2}(\Gamma)} 
 = \langle V\phi^\star - V\phi_{\ell}^\star \:,\: (1-Q_\ell^\Gamma)\psi\rangle_{H^{1/2}(\Gamma) \times H^{-1/2}(\Gamma)}
 \\& \quad 
 = \langle (1-Q_\ell^\Gamma)(g - V\phi_{\ell}^\star) \:,\: \psi\rangle_{H^{1/2}(\Gamma) \times H^{-1/2}(\Gamma)}
 % = 
 % \langle V\phi^\star - V\phi_{\ell}^\star \:,\: \psi\rangle_{L^2(\Gamma)} \\
% = 
% \sum_{F \in \mathcal{T}_\ell^\Gamma} \langle g - V\phi_{\ell}^\star \:,\: \psi\rangle_{L^2(F)} 
% \\& \quad 
% = 
% \sum_{F \in \mathcal{T}_\ell^\Gamma} \big( \langle J_{\ell}(g - V\phi_{\ell}^\star) \:,\: \psi\rangle_{L^2(F)} + \langle (1-J_{\ell})(g - V\phi_{\ell}^\star) \:,\: \psi\rangle_{L^2(F)} \big)
 \\& \quad 
 \leq 
 \sum_{F \in \mathcal{T}_\ell^\Gamma} \big( \lVert (1-Q_\ell^\Gamma)J_{\ell}(g - V\phi_{\ell}^\star)\rVert_{L^2(F)} + \lVert (1-Q_\ell^\Gamma)(1-J_{\ell})(g - V\phi_{\ell}^\star)\rVert_{L^2(F)} \big) \lVert \psi\rVert_{L^2(F)} 
% \\& \quad  
% \leq 
% \sum_{F \in \mathcal{T}_\ell^\Gamma} \big( \lVert w_{\ell}|_{F}\rVert_{L^2(F)} + \mu_{2,\ell}(T_F) \big) \lVert \psi\rVert_{L^2(F)} 
% \\& \quad  
 %= \sum_{F \in \mathcal{T}_\ell^\Gamma} \big( \lVert w_\ell^\star\rVert_{L^2(F)} + \lVert (1-J_{\ell})(g - V\phi_{\ell}^\star)\rVert_{L^2(F)} \big) \lVert \psi\rVert_{L^2(F)} 
 \\& \quad  
 \le \Big( \sum_{F \in \mathcal{T}_\ell^\Gamma} \lVert (1-Q_\ell^\Gamma)w_\ell^\star\rVert_{L^2(F)} \lVert \psi\rVert_{L^2(F)} \Big)
 + \lVert (1-Q_\ell^\Gamma)(1-J_{\ell})(g - V\phi_{\ell}^\star)\rVert_{L^2(\Gamma)}  \lVert \psi\rVert_{L^2(\Gamma)} 
 \\& \quad  
 \lesssim \Big( \sum_{F \in \mathcal{T}_\ell^\Gamma} \eta_{\ell}(T_F) \lVert \psi\rVert_{L^2(F)} \Big) + \mathrm{osc}_\ell \, \lVert \psi\rVert_{L^2(\Gamma)}
 \le 
 \Big( \sum_{T \in \mathcal{T}_\ell} \eta_{\ell}(T) \lVert \psi\rVert_{L^2(T \cap \Gamma)} \Big) + \mathrm{osc}_\ell \, \lVert \psi\rVert_{L^2(\Gamma)}.
% \\& \quad  
% \leq \sqrt{2} \sum_{T \in \mathcal{T}_{\ell}} \mu_{\ell}(T) \, \lVert \psi\rVert_{L^2(\partial T \cap \Gamma)} 
\end{align*}
This concludes the proof.
\end{proof}

By means of the foregoing lemma, we are able to adapt some arguments of~\cite{Siebert2011} to the present setting in order to prove $\lVert \phi^\star - \phi_\ell^\star\rVert_{H^{-1/2}(\Gamma)} \to 0$.

\begin{proposition}\label{prop:convergence:norm}
Algorithm~\ref{algo:adap} guarantees that
\begin{equation}\label{eq:prop:convergence:norm}
 \lim_{\ell \rightarrow \infty} \lVert \phi^\star - \phi_{\ell}^\star\rVert_{H^{-1/2}(\Gamma)} = 0.
\end{equation}
\end{proposition} 

\begin{proof}
The proof is split into four steps.

\textbf{Step~1.} 
Since $\phi_{\ell}^\star$ converges strongly in $H^{-1/2}(\Gamma)$ to $\phi_\infty^\star$, it remains to show that $\phi_\infty^\star = \phi^\star$. To this end, it suffices to show weak convergence $\phi_\ell^\star \rightharpoonup \phi^\star$ in $H^{-1/2}(\Gamma)$. Since $V$ is an isomorphism and $H^1(\Gamma)$ is densely contained in $H^{-1/2}(\Gamma)$, this is equivalent to 
\begin{equation*}
% \langle V\phi^\star - V\phi_\infty^\star \:,\: \psi\rangle_{H^{1/2}(\Gamma) \times H^{-1/2}(\Gamma)} 
% = 
% 0 
 \langle V\phi^\star - V\phi_{\ell}^\star \:,\: \psi\rangle_{H^{1/2}(\Gamma) \times H^{-1/2}(\Gamma)} \xrightarrow{\ell \to \infty} 0
 \qquad \text{for all } \psi \in H^1(\Gamma).
\end{equation*}
Let $\psi \in H^1(\Gamma)$.
%Since 
%\begin{equation*}
% \lim_{\ell \rightarrow \infty} \langle V\phi_\infty^\star - V\phi_{\ell}^\star \:,\: \psi\rangle_{H^{1/2}(\Gamma) \times H^{-1/2}%(\Gamma)} = 0,
%\end{equation*}
%it remains to prove that
%\begin{equation*}
% \lim_{\ell \rightarrow \infty} \langle V\phi^\star - V\phi_{\ell}^\star \:,\: \psi\rangle_{H^{1/2}(\Gamma) \times H^{-1/2}(\Gamma)} = 0.
%\end{equation*}
Let $Q_{\ell}^\Gamma \colon L^2(\Gamma) \rightarrow \mathcal{P}^0(\mathcal{T}_\ell^\Gamma)$ be the integral mean operator.
Then, Galerkin orthogonality, Lemma~\ref{lem:strongrel}, and a Poincar\'e-type inequality prove that
\begin{align}\label{eq:step0}
 \begin{split}
 &|\langle V\phi^\star - V\phi_{\ell}^\star \:,\: \psi\rangle_{H^{1/2}(\Gamma) \times H^{-1/2}(\Gamma)}| 
 = |\langle V\phi^\star - V\phi_{\ell}^\star \:,\: \psi-Q_{\ell}^\Gamma\psi\rangle_{H^{1/2}(\Gamma) \times H^{-1/2}(\Gamma)}| 
 \\& \quad
 \stackrel{\eqref{eq:lem:strongrel}}{\lesssim} 
 \mathrm{osc}_\ell \, \lVert \psi-Q_{\ell}^\Gamma\psi\rVert_{L^2(\Gamma)} 
 + 
 \sum_{T \in \mathcal{T}_\ell} \eta_{\ell}(T) \, \lVert \psi-Q_{\ell}^\Gamma\psi\rVert_{L^2(\partial T \cap \Gamma)} 
 \\& \quad
 \lesssim
 \mathrm{osc}_\ell \, \lVert \psi\rVert_{L^2(\Gamma)} 
 +
 \sum_{T \in \mathcal{T}_\ell} h_T \, \eta_{\ell}(T) \, \lVert \nabla_\Gamma \psi\rVert_{L^2(\partial T \cap \Gamma)}.
% \\& \quad
% \lesssim
% \mathrm{osc}_\ell \, \lVert \psi\rVert_{L^2(\Gamma)} 
% +
% \sum_{T \in \mathcal{T}_\ell} \big( \lVert w_\ell^\star\rVert_{L^2(T)} + h_T \lVert \nabla w_\ell^\star\rVert_{L^2(T)} \big) \, \lVert \nabla_\Gamma \psi\rVert_{L^2(\partial T \cap \Gamma)}.
 \end{split}
\end{align}
Recall that $\mathrm{osc}_\ell \to 0$ as $\ell \to \infty$ according to Proposition~\ref{thm:eta2conv}. It therefore remains to prove that also the sum converges to zero as $\ell \to \infty$. To this end, let $\ell' \in \mathbb{N}_0$ and $\ell' \le \ell$. We split the sum into the two sums 
\begin{align}\label{eq:step1}
\begin{split}
 \sum_{T \in \mathcal{T}_\ell} h_T \, \eta_{\ell}(T) \, \lVert \nabla_\Gamma \psi\rVert_{L^2(\partial T \cap \Gamma)}
 &= \sum_{T \in \mathcal{T}_{\ell'} \cap \mathcal{T}_\infty} h_T \, \eta_{\ell}(T) \, \lVert \nabla_\Gamma \psi\rVert_{L^2(\partial T \cap \Gamma)}
 \\& \quad
 + \sum_{T \in \mathcal{T}_\ell \backslash (\mathcal{T}_{\ell'} \cap \mathcal{T}_\infty)} h_T \, \eta_{\ell}(T) \, \lVert \nabla_\Gamma \psi\rVert_{L^2(\partial T \cap \Gamma)}
% &\sum_{T \in \mathcal{T}_\ell} \big( \lVert w_\ell^\star\rVert_{L^2(T)} + h_T \lVert \nabla w_\ell^\star\rVert_{L^2(T)} \big) \, \lVert \nabla_\Gamma \psi\rVert_{L^2(\partial T \cap \Gamma)}
% \nonumber
% \\& \quad
% \le 
% \sum_{T \in \mathcal{T}_{\ell'} \cap \mathcal{T}_\infty} \lVert w_\ell^\star\rVert_{L^2(T)} \lVert \nabla_\Gamma \psi\rVert_{L^2(\partial T \cap \Gamma)}
% + 
% \sum_{T \in \mathcal{T}_\ell \backslash (\mathcal{T}_{\ell'} \cap \mathcal{T}_\infty)} \lVert w_\ell^\star\rVert_{L^2(T)} \lVert \nabla_\Gamma \psi\rVert_{L^2(\partial T %\cap \Gamma)}
% \\& \qquad
% \nonumber
% +
% \sum_{T \in \mathcal{T}_{\ell'} \cap \mathcal{T}_\infty} h_T \lVert \nabla w_\ell^\star\rVert_{L^2(T)} \lVert \nabla_\Gamma \psi\rVert_{L^2(\partial T \cap \Gamma)}
% + 
% \sum_{T \in \mathcal{T}_\ell \backslash (\mathcal{T}_{\ell'} \cap \mathcal{T}_\infty)} h_T \lVert \nabla w_\ell^\star\rVert_{L^2(T)} \lVert \nabla_\Gamma \psi\rVert_{L^2(\partial T \cap \Gamma)}
\end{split}
\end{align}
and we consider each of these sums individually.

\textbf{Step~2.}
We consider the first sum on the right-hand side of~\eqref{eq:step1}.
Together with Proposition~\ref{thm:eta2conv}, the Cauchy--Schwarz inequality proves that
\begin{align}\label{eq:step2}
 \nonumber
 \sum_{T \in \mathcal{T}_{\ell'} \cap \mathcal{T}_\infty} 
 h_T \, \eta_{\ell}(T) \, \lVert \nabla_\Gamma \psi\rVert_{L^2(\partial T \cap \Gamma)}
 %h_T \lVert \nabla w_\ell^\star\rVert_{L^2(T)} \lVert \nabla_\Gamma \psi\rVert_{L^2(\partial T \cap \Gamma)}
% \lVert h_\ell\rVert_{L^\infty(\Omega)} \, \lVert \nabla_\Gamma \psi\rVert_{L^2(\Gamma)} \, \lVert \nabla w_\ell^\star\rVert_{L^2(\bigcup(\mathcal{T}_{\ell'} \cap \mathcal{T}_\infty))} 
% \\& \qquad
% = 
 &\le 
 \lVert h_\ell\rVert_{L^\infty(\Omega)} \, \lVert \nabla_\Gamma \psi\rVert_{L^2(\Gamma)} \, \eta_\ell(\mathcal{T}_{\ell'} \cap \mathcal{T}_\infty)
 \\&
 \lesssim 
 \eta_\ell(\mathcal{T}_{\ell'} \cap \mathcal{T}_\infty)
 \xrightarrow{\ell \to \infty} 0.
\end{align} 

\textbf{Step~3.}
We consider the second sum on the right-hand side of~\eqref{eq:step1}.
To show convergence, we first prove that $\bigcup[\mathcal{T}_\ell \backslash (\mathcal{T}_{\ell'} \cap \mathcal{T}_\infty)] \subseteq \bigcup(\mathcal{T}_{\ell'} \backslash \mathcal{T}_\infty)$. 
To this end, let $T \in \mathcal{T}_\ell \backslash (\mathcal{T}_{\ell'} \cap \mathcal{T}_\infty)$ and recall that $\ell \ge \ell'$. 
If $T \in \mathcal{T}_{\ell'}$, then $T \in \mathcal{T}_{\ell'} \backslash \mathcal{T}_\infty$. 
If $T \not\in \mathcal{T}_{\ell'}$, then there exists $T' \in \mathcal{T}_{\ell'}$ with $T \subsetneqq T'$. 
This implies that $T' \not\in \mathcal{T}_\infty$ and hence $T \subsetneqq T' \in \mathcal{T}_{\ell'} \backslash \mathcal{T}_\infty$. 
Overall, we confirm that 
\begin{align*}
 \bigcup(\mathcal{T}_\ell \backslash (\mathcal{T}_{\ell'} \cap \mathcal{T}_\infty)) 
 \subseteq \bigcup(\mathcal{T}_{\ell'} \backslash \mathcal{T}_\infty)
\end{align*}
Second, we note that $h_\ell \le h_{\ell'}$ pointwise almost everywhere. 
By efficiency~\eqref{eq:efficiency} from Proposition~\ref{prop:efficiency} and the Céa lemma~\eqref{eq:cea}, it holds that 
\begin{equation*}
 \eta_\ell 
 = 
 \lVert \nabla w_\ell^\star\rVert_{L^2(\Omega)} 
 \stackrel{\eqref{eq:efficiency}}{\lesssim} 
 \lVert \phi^\star - \phi_\ell^\star\rVert_{H^{-1/2}(\Gamma)} 
 \stackrel{\eqref{eq:cea}}{\lesssim} 
 \lVert \phi^\star\rVert_{H^{-1/2}(\Gamma)}.
\end{equation*} 
Hence, the convergence result from~\cite[Corollary~4.1]{Morin2008} leads to
\begin{align}\label{eq:step3}
 \nonumber
 &\sum_{T \in \mathcal{T}_\ell \backslash (\mathcal{T}_{\ell'} \cap \mathcal{T}_\infty)} 
 %h_T \lVert \nabla w_\ell^\star\rVert_{L^2(T)} \lVert \nabla_\Gamma \psi\rVert_{L^2(\partial T \cap \Gamma)}
 h_T \, \eta_{\ell}(T) \, \lVert \nabla_\Gamma \psi\rVert_{L^2(\partial T \cap \Gamma)}
 \le
 \lVert h_\ell\rVert_{L^\infty(\bigcup[\mathcal{T}_\ell \backslash (\mathcal{T}_{\ell'} \cap \mathcal{T}_\infty)])} \, \lVert \nabla_\Gamma \psi\rVert_{L^2(\Gamma)} \, 
 \eta_\ell
 %\lVert \nabla w_\ell^\star\rVert_{L^2(\Omega)}
 \\& \qquad 
 \lesssim
 \lVert h_\ell\rVert_{L^\infty(\bigcup[\mathcal{T}_\ell \backslash (\mathcal{T}_{\ell'} \cap \mathcal{T}_\infty)])}
 \le 
 \lVert h_{\ell'}\rVert_{L^\infty(\bigcup(\mathcal{T}_{\ell'} \backslash \mathcal{T}_\infty))}
 \xrightarrow{\ell' \to \infty} 0.
\end{align}

\textbf{Step~4.} 
To formally conclude the proof, let $\varepsilon > 0$. 
Recall from~\eqref{eq:step0}--\eqref{eq:step3} that
\begin{align*}
 &|\langle V\phi^\star - V\phi_{\ell}^\star \:,\: \psi\rangle_{H^{1/2}(\Gamma) \times H^{-1/2}(\Gamma)}|
 \\& \quad
 \le 
 C \, \Big( 
 \lVert \psi\rVert_{L^2(\Gamma)} \, \mathrm{osc}_\ell
 + \lVert \nabla_\Gamma \psi\rVert_{L^2(\Gamma)} \, \eta_\ell(\mathcal{T}_{\ell'} \cap \mathcal{T}_\infty)
 + \lVert \nabla_\Gamma \psi\rVert_{L^2(\Gamma)} \, \lVert h_{\ell'}\rVert_{L^\infty(\bigcup(\mathcal{T}_{\ell'} \backslash \mathcal{T}_\infty))} \Big).
\end{align*}
We then choose $\ell' \in \mathbb{N}_0$ such that $C \, \lVert \nabla_\Gamma \psi\rVert_{L^2(\Gamma)} \, \lVert h_{\ell'}\rVert_{L^\infty(\bigcup(\mathcal{T}_{\ell'} \backslash \mathcal{T}_\infty))} \le \varepsilon$. 
Finally, we may choose $\ell_0 \ge \ell'$ such that
$C \, \big[ 
 \lVert \psi\rVert_{L^2(\Gamma)} \, \mathrm{osc}_\ell
 + \lVert \nabla_\Gamma \psi\rVert_{L^2(\Gamma)} \, \eta_\ell(\mathcal{T}_{\ell'} \cap \mathcal{T}_\infty) \big]
 \le \varepsilon$  for all $\ell \ge \ell_0$. Overall, we reveal that 
\begin{align*}
 |\langle V\phi^\star - V\phi_{\ell}^\star \:,\: \psi\rangle_{H^{1/2}(\Gamma) \times H^{-1/2}(\Gamma)}| \le 2 \, \varepsilon
 \quad \text{for all } \ell \ge \ell_0
\end{align*} 
 and hence conclude the proof.
\end{proof}

Finally, we are in the position to conclude also the proof of Theorem~\ref{theorem:mainresult}.

\begin{proof}[{\bfseries Proof of Theorem~\ref{theorem:mainresult}}]
The lower estimate in~\eqref{eq:theorem:mainresult} is the efficiency~\eqref{eq:efficiency} from Proposition~\ref{prop:efficiency}. 
The upper estimate in~\eqref{eq:theorem:mainresult} follows from the continuity of the single-layer potential $\widetilde{V} \colon H^{-1/2}(\Gamma) \to H^{1/2}(\Gamma)$ with $u^\star - u^\star_\ell = \widetilde{V}(\phi^\star - \phi_\ell^\star)$. 
According to 
\begin{equation*}
 (u^\star - u^\star_\ell)|_{\Gamma} 
 = 
 \bigl(\widetilde{V}(\phi^\star - \phi^\star_\ell)\bigr)|_{\Gamma} 
 = 
 V(\phi^\star - \phi^\star_\ell)
 \stackrel{\eqref{eq:bie}}{=}
 g - V\phi_\ell^\star
\end{equation*}
and hence
\begin{equation*}
 \int_\Gamma (u^\star - u_\ell^\star) \mathop{}\!\mathrm{d}x
 \stackrel{\eqref{eq:bie}}{=}
 \langle g - V\phi^\star_\ell \:,\: 1\rangle_\Gamma
 \stackrel{\eqref{eq:weakdisc}}{=}
 0,
\end{equation*}
a Poincaré-type inequality proves
\begin{equation*}
 \lVert u^\star - u^\star_\ell\rVert_{H^1(\Omega)}
 \simeq 
 \lVert \nabla(u^\star - u^\star_\ell)\rVert_{L^2(\Omega)} + \Bigl|\int_\Gamma (u^\star - u^\star_\ell) \mathop{}\!\mathrm{d}x \Bigr|
 =
  \lVert \nabla(u^\star - u^\star_\ell)\rVert_{L^2(\Omega)}.
\end{equation*}
This the lower bound in~\eqref{eq:plainconv} with a Poincaré constant $C_{\textnormal{P}}$ depending only on $\Omega$. 
Together with~\eqref{eq:compupperbound}, which is the main result of~\cite{Kurz2020}, this proves the estimate of~\eqref{eq:plainconv}. 
Finally, error convergence~\eqref{eq:prop:convergence:norm} and efficiency~\eqref{eq:efficiency} prove 
\begin{equation*}
 \eta_\ell
 \stackrel{\eqref{eq:efficiency}}{\lesssim}
 \lVert \phi^\star - \phi_\ell^\star\rVert_{H^{-1/2}(\Gamma)}
 \stackrel{\eqref{eq:prop:convergence:norm}}{\longrightarrow} 0
 \quad \text{as } \ell \to \infty.
\end{equation*}
Together with vanishing oscillations
\begin{equation*}
 \mathrm{osc}_\ell
 \stackrel{\eqref{eq:oscconv}}{\longrightarrow}
 0
 \quad \text{as } \ell \to \infty,
\end{equation*}
this concludes the proof of~\eqref{eq:plainconv}.
\end{proof}

%%%%%%%%%%%%%%%%%%%%%%%%%%%%%%%%%%%%%%%%%%%%%%%%%%%%%%%%%%%%%%%%%%%%%%%%%%%%%%%%%%%
%%%%%%%%%%%%%%%%%%%%%%%%%%%%%%%%%%%%%%%%%%%%%%%%%%%%%%%%%%%%%%%%%%%%%%%%%%%%%%%%%%%

%%%%%%%%%%%%%%%%%%%%%%%%%%%%%%%%%%%%%%%%%%%%%%%%%%%%%%%%%%%%%%%%%%%%%%%%%%%%%%%%%%%
%%%%%%%%%%%%%%%%%%%%%%%%%%%%%%%%%%%%%%%%%%%%%%%%%%%%%%%%%%%%%%%%%%%%%%%%%%%%%%%%%%%
\section{Proof of Theorem~\ref{thm:extconv} (lifting of discretized boundary residual)}
\label{section:extension}
%%%%%%%%%%%%%%%%%%%%%%%%%%%%%%%%%%%%%%%%%%%%%%%%%%%%%%%%%%%%%%%%%%%%%%%%%%%%%%%%%%%
%%%%%%%%%%%%%%%%%%%%%%%%%%%%%%%%%%%%%%%%%%%%%%%%%%%%%%%%%%%%%%%%%%%%%%%%%%%%%%%%%%%

\noindent A key difficulty in the analysis of Algorithm~\ref{algo:adap} is the behaviour of the discrete FEM solution $w_\ell^\star$, since it is defined on the strip domain $\omega_\ell$, which changes adaptively in each step of Algorithm~\ref{algo:adap}. 
The goal of this section is to prove Theorem~\ref{thm:extconv}, which is crucial for the proof of Proposition~\ref{prop:efficiency} (efficiency) and Proposition~\ref{prop:auxconv} (convergence of functional upper bound).

\begin{lemma} \label{lem:discnorm}
Let $1 < r < \infty$ and $1/r + 1/r' = 1$. 
Then, there exists a constant $C > 0$ such that for all $T \in \mathcal{T}_\ell$ and $f_\ell \in \mathcal{S}^q(\mathcal{T}_\ell^\Gamma)$, there holds that  
\begin{equation} \label{eq:liftcont}
 \lVert \widehat{f}_\ell\rVert_{W^{1,r}(T)}^r 
 \leq 
 C(h_T^{-r/r'} \lVert f_\ell\rVert_{L^r(\partial T \cap \Gamma)}^r + \lvert f_\ell\rvert_{W^{1/r',r}(\partial T \cap \Gamma)}^r),
\end{equation}
where $\widehat{f}_\ell$ is the natural lifting of $f_\ell$ to $\Omega$; see~\eqref{eq:natext}. 
The constant $C$ depends only on $r$, $d$, $\kappa$-shape regularity of $\mathcal{T}_\ell$, and $\mathrm{diam}(\Omega)$.      
\end{lemma}

\begin{proof}
The proof is split into 3 steps.

\textbf{Step~1 (Estimate of the discrete \boldsymbol{$W^{1,r}(T)$}-norm).} Let $v_\ell \in \mathcal{S}^p(\mathcal{T}_\ell)$ be arbitrary such that
\begin{equation*}
 v_\ell 
 = 
 \sum_{z \in \mathcal{D}_\ell} \alpha_z \zeta_{\ell,z}.
\end{equation*}
Let $T \in \mathcal{T}_\ell$ and $\mathcal{D}_\ell(T)$ the local degrees of freedom on $T$. 
The affine diffeomorphism $\Phi_T \colon T_{\mathrm{ref}} \to T$, $\Phi_T(x) \coloneqq B_Tx + c_T$ with $|\mathrm{det}(B_T)| \simeq |T| \simeq h_T^d \simeq \lVert B_T\rVert_{\mathrm{Frob}}^d \simeq \lVert B_T^{-1}\rVert_{\mathrm{Frob}}^{-d}$ yields for $\widetilde{v}_{\ell,T} \coloneqq v_\ell \circ \Phi_T$ that
\begin{equation*}
 \lVert v_\ell\rVert_{W^{1,r}(T)}^r 
 \simeq 
 h_T^d(\lVert \widetilde{v}_{\ell,T}\rVert_{L^r(T_{\mathrm{ref}})}^r 
 + 
 \lVert (B_T^\intercal)^{-1}\nabla \widetilde{v}_{\ell,T}\rVert_{L^r(T_{\mathrm{ref}})}^r)
\end{equation*} 
Recall the standard basis $(\zeta_j)_{j=1}^{n_{d}}$ of $\mathbb{P}^q(T_{\mathrm{ref}})$ from Section~\ref{section:preliminaries}.
For $j \in \{1, \ldots , n_{q,d}\}$ let $z_j \in \mathcal{D}_\ell(T)$ be the unique degree of freedom such that $\widetilde{\zeta}_j = \zeta_{z_j} \circ \Phi_T$. We can write $v_\ell|_T$ in the form
\begin{equation*}
 v_\ell|_T 
 = 
 \sum_{z \in \mathcal{D}_\ell(T)} \alpha_z \zeta_{\ell,z}|_T,
\end{equation*}
which yields
\begin{equation*}
 \widetilde{v}_{\ell,T} 
 = 
 \sum_{j=1}^{n_{d}} \alpha_{z_j} \widetilde{\zeta}_j.
\end{equation*}
We obtain
\begin{equation*}
 \lVert \widetilde{v}_{\ell,T}\rVert_{L^r(T_{\mathrm{ref}})} 
 \leq 
 \sum_{j = 1}^{n_d} \lvert \alpha_{z_j}\rvert \lVert \widetilde{\zeta}_j\rVert_{L^r(T_{\mathrm{ref}})} 
 \simeq
 \Bigl(\sum_{j = 1}^{n_{d}} \lvert \alpha_{z_j}\rvert^r\Bigr)^{1/r} 
\end{equation*}
and
\begin{equation*}
 \begin{split}
  \lVert (B_T^\intercal)^{-1}\nabla \widetilde{v}_{\ell,T}\rVert_{L^r(T_{\mathrm{ref}})} 
  &\lesssim 
  h_T^{-1} \sum_{j=1}^{n_{d}} \lvert \alpha_{z_j}\rvert \lVert \nabla \widetilde{\zeta}_j\rVert_{L^r(T_{\mathrm{ref}})} 
  \simeq 
  h_T^{-1} \Bigl(\sum_{j=1}^{n_{d}} \lvert \alpha_{z_j}\rvert^r\Bigr)^{1/r}.
 \end{split}
\end{equation*}
This proves
\begin{equation} \label{eq:discvolnorm}
 \lVert v_\ell\rVert_{W^{1,r}(T)}^r 
 \lesssim 
 h_T^d(1 + h_T^{-r}) \sum_{j = 1}^{n_{d}} \lvert \alpha_{z_j}\rvert^r
 =
 h_T^d(1 + h_T^{-r}) \sum_{z \in \mathcal{D}_\ell(T)} \lvert \alpha_{z}\rvert^r.
\end{equation}

\textbf{Step~2 (Estimate of the discrete \boldsymbol{$W^{1/r',r}(F)$}-norm).}
Next, we aim to prove a similar estimate for the $W^{1/r',r}(\Gamma)$-norm. 
Let $f_\ell \in \mathcal{S}^q(\mathcal{T}_\ell^\Gamma)$ be arbitrary such that
\begin{equation*}
 f_\ell
 = 
 \sum_{z \in \mathcal{D}_\ell(\Gamma)} \alpha_z \zeta_{\ell,z}|_\Gamma.
\end{equation*}
Let $F \in \mathcal{T}_\ell^\Gamma$ and let $\mathcal{D}_\ell(F)$ denote the local degrees of freedom on $F$. 
Let $T \in \mathcal{T}_\ell$ with $F \subseteq \partial T$. 
Similar to Step~1, the diffeomorphism $\Psi_F \colon F_{\mathrm{ref}} \to F$, $\Psi_F(x) \coloneqq A_Fx + b_F$ with $\sqrt{\mathrm{det}(A_F^\intercal A_F)} \simeq |F| \simeq h_F^{d-1} \simeq \lVert A_F\rVert_{\mathrm{Frob}}^{d-1} \simeq \lVert A_F^{-1}\rVert_{\mathrm{Frob}}^{1-d}$ yields for 
\begin{equation*}
 \widetilde{f}_{\ell,F} 
 \coloneqq 
 f_\ell \circ \Psi_F 
 = 
 \sum_{j = 1}^{n_{d-1}} \alpha_{z_j} \widetilde{\zeta}_j|_{F_{\mathrm{ref}}}
\end{equation*} 
that
\begin{equation*}
 \lVert f_\ell\rVert_{L^r(F)}^r 
 \simeq 
 h_F^{d-1} \lVert \widetilde{f}_{\ell,F}\rVert_{L^r(F_{\mathrm{ref}})}^r
\end{equation*}
and, with $r/r' = r - 1$ and hence $2(d-1) - (d - 1 + r/r') = d-r$,
\begin{equation*}
 \begin{split}
  \lvert f_\ell\rvert_{W^{1/r',r}(F)}^r 
  &= 
  \int_F \int_F \frac{\lvert f_\ell(x) - f_\ell(y)\rvert^r}{\lVert x-y\rVert^{d - 1 + r/r'}} \mathop{}\!\mathrm{d}x \mathop{}\!\mathrm{d}y 
  = 
  h_F^{2(d-1)} \int_{F_{\mathrm{ref}}} \int_{F_{\mathrm{ref}}} \frac{\lvert \widetilde{f}_{\ell,F}(x) - \widetilde{f}_{\ell,F}(y)\rvert^r}{\lVert A_F(x - y)\rVert^{d - 1 + r/r'}} \mathop{}\!\mathrm{d}x \mathop{}\!\mathrm{d}y \\
  &\simeq 
  h_F^{d-r} \int_{F_{\mathrm{ref}}} \int_{F_{\mathrm{ref}}} \frac{\lvert \widetilde{f}_{\ell,F}(x) - \widetilde{f}_{\ell,F}(y)\rvert^r}{\lVert x-y\rVert^{d-1 + r/r'}} \mathop{}\!\mathrm{d}x \mathop{}\!\mathrm{d}y 
  = 
  h_F^{d-r} \lvert \widetilde{f}_{\ell,F}\rvert_{W^{1/r',r}(F_{\mathrm{ref}})}^r.
 \end{split}
\end{equation*}
Let $j_0 \in \{1, \ldots , n_{d-1}\}$. 
Note that 
\begin{equation*}
 \inf\{\lVert v - \widetilde{\zeta}_{j_0}|_{F_{\mathrm{ref}}}\rVert_{W^{1/r',r}(F_{\mathrm{ref}})} 
 \:\vert\: 
 v \in \mathrm{span}\{\widetilde{\zeta}_j|_{F_{\mathrm{ref}}} \:\vert\: j \in \{1, \ldots , n_{d-1}\}  \setminus \{j_0\}\}\} \gtrsim 1.
\end{equation*}
If $\alpha_{z_{j_0}} \neq 0$, we obtain
\begin{equation*}
 \begin{split}
  1 
  &\lesssim 
  \Big\lVert\sum_{\substack{j = 1 \\ j \neq j_0}}^{n_{d-1}} (-\alpha_{z_j} / \alpha_{z_{j_0}}) \widetilde{\zeta}_j|_{F_{\mathrm{ref}}} - \widetilde{\zeta}_{j_0}\Big\rVert_{W^{1/r',r}(F_{\mathrm{ref}})}^r \\
  &= 
  \lvert \alpha_{z_{j_0}}\rvert^{-r} \lVert \widetilde{f}_{\ell,F}\rVert_{W^{1/r',r}(F_{\mathrm{ref}})}^r 
  \lesssim 
  \lvert \alpha_{z_{j_0}}\rvert^{-r} (h_F^{1-d} \lVert f_\ell\rVert_{L^r(F)}^r + h_F^{r-d} \lvert f_\ell\rvert_{W^{1/r',r}(F)}^r)
 \end{split}
\end{equation*}
and hence
\begin{equation*}
 \lvert \alpha_{z_{j_0}}\rvert^r 
 \lesssim 
 h_F^{1-d} \lVert f_\ell\rVert_{L^r(F)}^r + h_F^{r-d} \lvert f_\ell\rvert_{W^{1/r',r}(F)}^r.
\end{equation*}
Note that the latter estimate holds trivially if $\alpha_{z_{j_0}} = 0$.
Since $j_0 \in \{1, \ldots, n_{d-1}\}$ was arbitrary and $n_{d-1} \lesssim 1$, $\kappa$-shape regularity yields
\begin{equation} \label{eq:disctracenorm}
 \begin{split}
  \sum_{z \in \mathcal{D}_\ell(F)} \lvert \alpha_z\rvert^r 
  =
  \sum_{j = 1}^{n_{d-1}} \lvert \alpha_{z_j}\rvert^r
  &\lesssim 
  h_F^{1-d} \lVert f_\ell\rVert_{L^r(F)}^r + h_F^{r-d} \lvert f_\ell\rvert_{W^{1/r',r}(F)}^r \\
  &\simeq 
  h_T^{1-d} \lVert f_\ell\rVert_{L^r(F)}^r + h_T^{r-d} \lvert f_\ell\rvert_{W^{1/r',r}(F)}^r
 \end{split}
\end{equation}

\textbf{Step~3 (Estimate of \boldsymbol{$W^{1,r}(T)$}-norm of the natural lifting).}
Let $f_\ell \in \mathcal{S}^p(\mathcal{T}_\ell^\Gamma)$ be such that
\begin{equation*}
 f_\ell 
 = 
 \sum_{z \in \mathcal{D}_\ell(\Gamma)} \alpha_z \zeta_{\ell,z}|_{\Gamma}.
\end{equation*}
Denote by $\widehat{f}_\ell \in \mathcal{S}^p(\mathcal{T}_\ell)$ its natural lifting. Then, our previous considerations show
\begin{equation*}
 \begin{split}
  \lVert \widehat{f}_\ell\rVert_{W^{1,r}(T)}^r 
  &\stackrel{\eqref{eq:discvolnorm}}{\lesssim} 
  h_T^d(1 + h_T^{-r}) \sum_{z \in \mathcal{D}_\ell(T) \cap \mathcal{D}_\ell(\Gamma)} \lvert \alpha_z\rvert^r \\
  &\stackrel{\eqref{eq:disctracenorm}}{\lesssim} h_T^d(1 + h_T^{-r}) (h_T^{1-d} \lVert f_\ell\rVert_{L^r(\partial T \cap \Gamma)}^r + h_T^{r-d} \lvert f_\ell\rvert_{W^{1/r',r}(\partial T \cap \Gamma)}^r) \\
  &\lesssim
  h_T^{-r/r'} \lVert f_\ell\rVert_{L^r(\partial T \cap \Gamma)}^r + \lvert f_\ell\rvert_{W^{1/r',r}(\partial T \cap \Gamma)}^r.
 \end{split}
\end{equation*}
This concludes the proof.
\end{proof}

The following proposition states a uniform bound for the natural lifting of the discretized boundary residual $g_\ell$, which is imperative for the proof of Proposition~\ref{prop:efficiency}. 
To ease notation, we define the natural lifting of a discretized function $J_\ell f \in \mathcal{S}^q(\mathcal{T}_\ell)$ as $\widehat{J_\ell} f \coloneqq \widehat{J_\ell f}$.

\begin{proposition} \label{prop:extbound}
There exists $r_0 > 2$ such that for all $2 \leq r \leq r_0$ there exist $1/2 \leq s_r \leq 1$ and $C_r > 0$ such that
\begin{equation} \label{eq:boundarypoincare}
 \lVert  \nabla \widehat{J_\ell} (f - Q_\ell^\Gamma f)\rVert_{L^r(\Omega)} 
 \leq 
 C_r \lvert f\rvert_{H^{s_r}(\Gamma)}
 \quad \text{for any } f \in H^{s_r}(\Gamma).
\end{equation}
In particular, it holds that $s_r = 1/2$ for $r=2$ and
\begin{equation} \label{eq:boundarypoincare2}
\lVert \nabla \widehat{g}_\ell\rVert_{L^2(\Omega)} 
\leq C_2 
\lvert g - V \phi_\ell^\star\rvert_{H^{1/2}(\Gamma)} 
\leq 
C_2C_V \lVert \phi^\star - \phi_\ell^\star\rVert_{H^{-1/2}(\Gamma)},
\end{equation}
where $C_V$ is the operator norm of $V \colon H^{-1/2}(\Gamma) \to H^{1/2}(\Gamma)$.
\end{proposition}

\begin{proof} 
The proof is split into $2$ steps.

\textbf{Step~1 (Proof of inverse inequalities).}
Given $F \in \mathcal{T}_\ell^\Gamma$, $s \geq 0$, and $1 < \sigma,\nu < \infty$, we want to prove that there exists a constant $C_{\textnormal{inv}} > 0$ depending only on $\Omega$, $d$, $\kappa$-shape regularity of $\mathcal{T}_\ell$, $q$, $\sigma$, $\nu$, $s$, but not on $F$ such that
\begin{equation} \label{eq:invineq}
 \lVert \varphi_\ell\rVert_{W^{s,\sigma}(F)} 
 \leq 
 C_{\textnormal{inv}} h_F^{-s + (d-1)(1/\sigma - 1/\nu)} \lVert \varphi_\ell\rVert_{L^\nu(F)}
 \quad \text{for all } \varphi_\ell \in \mathbb{P}^q(F).
\end{equation}
We transform $F$ to the reference facet $F_{\mathrm{ref}}$ by means of the affine diffeomorphism $\Psi_F \colon F_{\mathrm{ref}} \to F$, $\Psi(x) \coloneqq A_F x + b_F$ with $\sqrt{\mathrm{det}(A_F^\intercal A_F)} \simeq |F| \simeq h_F^{d-1} \simeq \lVert A_F\rVert_{\mathrm{Frob}}^{d-1} \simeq \lVert A_F^{-1}\rVert_{\mathrm{Frob}}^{1-d}$. 
Noting that $h_F \leq \mathrm{diam}(\Omega)$, we establish
\begin{equation*}
 \lVert \varphi_\ell\rVert_{L^\sigma(F)}^\sigma
  \simeq
  h_F^{d-1} \lVert \varphi_\ell \circ \Psi_F\rVert_{L^\sigma(F_{\mathrm{ref}})}^\sigma
  \lesssim
  h_F^{d-1 - s \sigma} \lVert \varphi_\ell \circ \Psi_F\rVert_{L^\sigma(F_{\mathrm{ref}})}^\sigma
\end{equation*}
and
\begin{equation*}
 \begin{split}
  \lvert \varphi_\ell\rvert_{W^{s,\sigma}(F)}^\sigma
  &\simeq
  \int_F \int_F \frac{\lvert \varphi_\ell(x) - \varphi_\ell(y)\rvert^\sigma}{\lvert x-y\rvert^{d-1 + s\sigma}} \mathop{}\!\mathrm{d}x \mathop{}\!\mathrm{d}y \\
  &=
  h_F^{2(d-1)} \int_{F_{\mathrm{ref}}} \int_{F_{\mathrm{ref}}} \frac{\lvert \varphi_\ell \circ \Psi_F(x) - \varphi_\ell \circ \Psi_F(y)\rvert^\sigma}{\lvert A_F(x-y)\rvert^{d-1 + s\sigma}} \mathop{}\!\mathrm{d}x \mathop{}\!\mathrm{d}y \\
  &\simeq
  h_F^{2(d-1) - (d - 1 + s \sigma)} \int_{F_{\mathrm{ref}}} \int_{F_{\mathrm{ref}}} \frac{\lvert \varphi_\ell \circ \Psi_F(x) - \varphi_\ell \circ \Psi_F(y)\rvert^\sigma}{\lvert x-y\rvert^{d-1 + s\sigma}} \mathop{}\!\mathrm{d}x \mathop{}\!\mathrm{d}y \\
  &=
  h_F^{d - 1 - s \sigma} \lvert \varphi_\ell \circ \Psi_F\rvert_{W^{s,\sigma}(F_{\mathrm{ref}})}^\sigma.
 \end{split}
\end{equation*}
Since all norms are equivalent on the finite-dimensional space $\mathbb{P}^q(F_{\mathrm{ref}})$, we obtain
\begin{equation*}
 \begin{split}
  \lVert \varphi_\ell\rVert_{W^{s,\sigma}(F)}^\sigma
  &\lesssim
  h_F^{d-1 - s \sigma} \lVert \varphi_\ell \circ \Psi_F\rVert_{W^{s,\sigma}(F_{\mathrm{ref}})}^\sigma
  \simeq
  h_F^{d - 1 - s \sigma} \lVert \varphi_\ell\circ \Psi_F\rVert_{L^\nu(F_{\mathrm{ref}})}^\sigma.
 \end{split}
\end{equation*}
With
\begin{equation*}
 \lVert \varphi_\ell \circ \Psi_F\rVert_{L^\nu(F_{\mathrm{ref}})}^\nu
 \simeq
 h^{1-d} \lVert \varphi_\ell\rVert_{L^\nu(F)}^\nu,
\end{equation*}
we obtain
\begin{equation*}
 \lVert \varphi_\ell\rVert_{W^{s,\sigma}(F)}
 \lesssim
 h_F^{(d - 1)/\sigma - s} h_F^{(1 - d)/\nu } \lVert \varphi_\ell\rVert_{L^\nu(F)}
 =
 h_F^{-s + (d-1)(1/\sigma - 1/\nu)} \lVert \varphi_\ell\rVert_{L^\nu(F)}.
\end{equation*}
This concludes the proof of~\eqref{eq:invineq}.

\textbf{Step~2 (Proof of~(\ref{eq:boundarypoincare})).}
Let $\ell \in \mathbb{N}_0$. Note that $\widehat{J_\ell} f = 0$ if $T \in \mathcal{T}_\ell$ satisfies $T \cap \Gamma = \emptyset$. 
Hence, Lemma~\ref{lem:discnorm} and~\eqref{eq:invineq} with $s = 1/r'$ and $\sigma = \nu = r$ yield
\begin{equation} \label{eq:liftingcont}
 \begin{split}
  &\lVert \nabla \widehat{J_\ell} (f - Q_\ell^\Gamma f)\rVert_{L^r(\Omega)}^r 
  =
  \sum_{T \in \mathcal{T}_\ell[\Gamma]} \lVert \nabla  \widehat{J_\ell} (f - Q_\ell^\Gamma f)\rVert_{L^r(T)}^r \\
  &\qquad \stackrel{\eqref{eq:boundarypoincare}}{\lesssim}
  \sum_{T \in \mathcal{T}_\ell[\Gamma]} \Bigl( h_T^{-r/r'} \lVert J_\ell (f - Q_\ell^\Gamma f)\rVert_{L^r(\partial T \cap \Gamma)}^r + \lvert J_\ell (f - Q_\ell^\Gamma f)\rvert_{W^{1/r',r}(\partial T \cap \Gamma)}^r \Bigr) \\
  &\qquad \stackrel{\eqref{eq:invineq}}{\lesssim}
  \sum_{T \in \mathcal{T}_\ell[\Gamma]} h_T^{-r/r'} \lVert J_\ell (f - Q_\ell^\Gamma f)\rVert_{L^r(\partial T \cap \Gamma)}^r
  = 
  \lVert h_\ell^{-1/r'} J_\ell (f - Q_\ell^\Gamma f)\rVert_{L^r(\Gamma)}^r.
 \end{split}
\end{equation}
We define 
\begin{equation} \label{eq:srdef}
 s_r 
 \coloneqq 
 [r - 1 + (d-1)(r/2 - 1)]/r
 \quad \text{for any } r \geq 2
\end{equation}
and note that $s_2 = 1/2$. 
Let $r \geq 2$. 
We note that $r/r' = r - 1$ and use~\eqref{eq:invineq} with $s=0$, $\sigma = r$, and $\nu = 2$ to obtain
\begin{equation} \label{eq:longcomput}
 \begin{split}
  \lVert h^{-1/r'}_\ell J_\ell (f - Q_\ell^\Gamma f)\rVert_{L^r(\Gamma)}^r 
  &=  
  \sum_{F \in \mathcal{T}_\ell^{\Gamma}} h_F^{-r/r'} \lVert J_\ell (f - Q_\ell^\Gamma f)\rVert_{L^r(F)}^r \\
  &\stackrel{\eqref{eq:invineq}}{\lesssim} 
  \sum_{F \in \mathcal{T}_\ell^{\Gamma}} h_F^{1 - r} h_F^{(d-1)(1 - r/2)} \lVert J_\ell (f - Q_\ell^\Gamma f)\rVert_{L^2(F)}^r \\
  &\stackrel{\eqref{eq:srdef}}{=} \sum_{F \in \mathcal{T}_\ell^{\Gamma}} h_F^{-rs_r} \lVert J_\ell (f - Q_\ell^\Gamma f)\rVert_{L^2(F)}^r \\
  &= \sum_{F \in \mathcal{T}_\ell^\Gamma} \lVert h_\ell^{-s_r} J_\ell (f - Q_\ell^\Gamma f)\rVert_{L^2(F)}^{r - 2} \lVert h_\ell^{-s_r} J_\ell (f - Q_\ell^\Gamma f)\rVert_{L^2(F)}^2 \\
  &\leq
  \lVert h_\ell^{-s_r}J_\ell(f - Q_\ell^\Gamma f)\rVert_{L^2(\Gamma)}^{r - 2} \sum_{F \in \mathcal{T}_\ell^{\Gamma}} \lVert h_\ell^{-s_r} J_\ell (f - Q_\ell^\Gamma f)\rVert_{L^2(F)}^2 \\
  &=
  \lVert h_\ell^{-s_r}J_\ell (f - Q_\ell^\Gamma f)\rVert_{L^2(\Gamma)}^r.
 \end{split}
\end{equation}
Lemma~\ref{lem:SZ} and the properties~\eqref{eq:avgprop} of $Q_\ell^\Gamma$ yield
\begin{equation} \label{eq:avgapprox}
 \begin{split}
 &\lVert h_\ell^{-s} J_\ell(f - Q_\ell^\Gamma f)\rVert_{L^2(\Gamma)}^2
 = 
 \sum_{F \in \mathcal{T}_\ell^\Gamma} h_F^{-2s} \lVert J_\ell(f - Q_\ell^\Gamma f)\rVert_{L^2(F)}^2 
 \stackrel{\eqref{eq:SZtraceL2}}{\lesssim}
 \sum_{F \in \mathcal{T}_\ell^\Gamma} h_F^{-2s} \lVert f - Q_\ell^\Gamma f\rVert_{L^2(\Gamma_\ell[F])}^2 \\
 &\qquad \lesssim
 \sum_{F \in \mathcal{T}_\ell^\Gamma} h_F^{-2s} \lVert f - Q_\ell^\Gamma f\rVert_{L^2(F)}^2
 \stackrel{\eqref{eq:avgprop}}{\lesssim}
 \sum_{F \in \mathcal{T}_\ell^\Gamma} \lvert f\rvert_{H^{s}(F)}^2 
 \leq \lvert f\rvert_{H^{s}(\Gamma)}^2
 \quad \text{for all } 0 < s \leq 1.
 \end{split}
\end{equation}
It remains to verify that $0 < s_r \leq 1$. 
Indeed, elementary calculations show that $0 < s_r \leq 1$ if and only if $2 - 2/(d+1) < r \leq 2 + 2/(d-1)$. 
With $r_0 = 2 + 2/(d-1) > 2$, we are thus led to
\begin{equation*}
 \begin{split}
 \lVert \nabla \widehat{J_\ell} (f - Q_\ell^\Gamma f)\rVert_{L^r(\Omega)} 
 \stackrel{\eqref{eq:liftingcont}}{\lesssim}
 \lVert h^{-1/r'}J_\ell (f - Q_\ell^\Gamma f)\rVert_{L^r(\Gamma)} 
 &\stackrel{\eqref{eq:longcomput}}{\lesssim}
 \lVert h^{-s_r}J_\ell (f - Q_\ell^\Gamma f)\rVert_{L^2(\Gamma)} \\ 
 &\stackrel{\eqref{eq:avgapprox}}{\lesssim}
 \lvert f\rvert_{H^{s_r}(\Gamma)}.
 \end{split}
\end{equation*}
This proves~\eqref{eq:boundarypoincare}. 

\textbf{Step~3 (Proof of~(\ref{eq:boundarypoincare2})).}
By~\eqref{eq:weakdisc}, there holds $Q_\ell^\Gamma(g - V\phi_\ell^\star) = 0$. 
Hence, we can apply~\eqref{eq:boundarypoincare} with $f = g - V\phi_\ell^\star$ and $g_\ell = J_\ell f = J_\ell(f - Q_\ell^\Gamma f)$ for $r = 2$ and $s_r = 1/2$ to obtain~\eqref{eq:boundarypoincare2}. 
This concludes the proof.
\end{proof}

We now have all the ingredients to prove Theorem~\ref{thm:extconv}.

\begin{proof}[\textbf{Proof of Theorem~\ref{thm:extconv}}]
We define the limiting layer
\begin{equation*}
 \omega_\infty 
 \coloneqq 
 \bigcap_{\ell \in \mathbb{N}_0} \omega_\ell.
\end{equation*}
As our first step, we want to characterize the set $\omega_\infty$ in terms of the involved triangulations.

\textbf{Step~1 (Characterization of \boldsymbol{$\omega_\infty$}).} 
Recall the set $\mathcal{T}_\infty$ of all elements which remain unrefined after a finite amount of steps; see~\eqref{eq:Tinfty}. 
We want to prove that 
\begin{equation} \label{eq:omegainfty}
 \omega_\infty 
 \subseteq 
 \bigcup \{T \in \mathcal{T}_\infty \:\vert\: T \subset \overline{\omega}_\infty\} 
 \eqqcolon
 \widetilde{\omega}_\infty
 \subseteq
 \overline{\omega}_\infty.
\end{equation}
The inclusion $\widetilde{\omega}_\infty \subseteq \overline{\omega}_\infty$ is obvious. 
In order to show the converse inclusion, let $x \in \omega_\infty$. 
We aim to prove that $x \in \widetilde{\omega}_\infty$. 
Since $x \in \omega_\ell = \Omega_\ell^{[k]}[\Gamma]$ for every $\ell \in \mathbb{N}_0$, there exists a sequence $(\widetilde{T}_\ell)_{\ell \in \mathbb{N}_0}$ with $x \in \widetilde{T}_\ell \in \mathcal{T}_\ell^{(k)}[\Gamma]$ and $\widetilde{T}_{\ell + 1} \subseteq \widetilde{T}_\ell$ for all $\ell \in \mathbb{N}_0$. 
Furthermore, for all $\ell \in \mathbb{N}_0$, there exists $T_\ell \in \mathcal{T}_\ell$ with $T_\ell \cap \Gamma \neq 0$ and $\widetilde{T}_\ell \in \mathcal{T}_\ell^{[k]}[T_\ell]$. 
If $\widetilde{T}_\ell \in \mathcal{T}_\infty$ for some $\ell \in \mathbb{N}_0$ (and hence $\widetilde{T}_\ell = \widetilde{T}_k$ for all $k \geq \ell$), then $x \in \widetilde{\omega}_\infty$. 
If $\widetilde{T}_\ell \notin \mathcal{T}_\infty$ for all $\ell \in \mathbb{N}_0$, then $|\widetilde{T}_\ell| \xrightarrow{\ell \to \infty} 0$ by~\cite[Corollary~4.5]{Morin2008}. 
By use of $\kappa$-shape regularity, we obtain
\begin{equation*}
\mathrm{dist}(x,\Gamma) 
\leq 
\mathrm{dist}(x, \Gamma \cap \Omega_\ell^{[k]}(T_\ell)) 
\leq 
\mathrm{diam}(\Omega_\ell^{[k]}[T_\ell]) 
\simeq 
|\widetilde{T}_\ell|^{1/d} \xrightarrow{\ell \to \infty} 0.
\end{equation*}
This yields $x \in \Gamma$, contradicting $x \in \omega_\infty$ since $\omega_\infty \subseteq \Omega$ and $\Omega$ is open, i.e., $\Omega \cap \Gamma = \emptyset$. 
Hence, the case that $\widetilde{T}_\ell \notin \mathcal{T}_\infty$ for all $\ell \in \mathbb{N}_0$ cannot occur. 
This proves~\eqref{eq:omegainfty}.

\textbf{Step~2 (Construction of \boldsymbol{$\widehat{g}_\infty$}).}
By Lemma~\ref{lemma:apriori}, there exists $\phi_\infty^\star \in H^{-1/2}(\Gamma)$ such that $\lVert \phi^\star - \phi_\infty^\star\rVert_{H^{-1/2}(\Gamma)} \xrightarrow{\ell \to \infty} 0$. 
We define $g_{\infty,\ell} \coloneqq J_\ell(g - V\phi_\infty^\star)$ and denote by $\widehat{g}_{\infty,\ell}$ its discrete lifting. 
To construct $\widehat{g}_\infty \in H^1(\Omega)$, we first define $\widehat{g}_\infty|_{\Omega \setminus \omega_\infty} \coloneqq 0$. 
Second, since $\omega_\infty$ coincides with $\widetilde{\omega}_\infty$ up to a set of measure zero, it suffices to define $\widehat{g}_\infty|_T$ for all $T \in \mathcal{T}_\infty$ such that $T \subseteq \overline{\omega}_\infty$. 
We note that
\begin{equation*}
 \mathcal{T}_\infty = \displaystyle{\bigcup_{\ell \in \mathbb{N}_0} \mathcal{T}_\ell^0},
 \quad \text{where }
 \mathcal{T}_\ell^0 
 \coloneqq 
 \{T \in \mathcal{T}_\ell \:\vert\: \mathcal{T}_\ell[T] \subseteq \mathcal{T}_\infty\}
 \quad \text{for all } \ell \in \mathbb{N}_0;
\end{equation*}
see~\cite[Section~4.1]{Morin2008}.
Let $T \in \mathcal{T}_\infty$ with $T \subseteq \overline{\omega}_\infty$. 
Then, there is a smallest index $\ell_T \in \mathbb{N}_0$ such that $T \in \mathcal{T}_{\ell_T}^0$. 
We define $\widehat{g}_\infty|_T \coloneqq\widehat{g}_{\infty,\ell_T}|_T$. 
Note that if $T \cap \Gamma = \emptyset$, then $\widehat{g}_\infty|_T = 0$. 
Furthermore, if $\ell_T \leq \ell$, then $\mathcal{T}_{\ell_T}[T] = \mathcal{T}_\ell[T]$. 
Since $J_\ell(g - V\phi_\infty^\star)|_{\partial T \cap \Gamma}$ depends only on $(g - V\phi_\infty^\star)|_{\Gamma_\ell[\partial T]}$ (see Lemma~\ref{lem:SZ}), this yields $\widehat{g}_{\infty,\ell}|_T = \widehat{g}_{\infty, \ell_T}|_T$. 

\textbf{Step~3 (\boldsymbol{$\widehat{g}_\infty$} is weakly differentiable)}. 
Let $\varphi \in C^{\infty}_0(\Omega)^d$. Then, integration by parts proves
\begin{equation*}
 \begin{split}
  &\langle \widehat{g}_\infty \:,\: \operatorname{\mathrm{div}}\varphi\rangle_{L^2(\Omega)} 
  = 
  \sum_{\substack{ T \in \mathcal{T}_\infty \\ T \cap \Gamma \neq \emptyset}} \langle \widehat{g}_\infty \:,\: \operatorname{\mathrm{div}}\varphi\rangle_{L^2(T)} 
  = 
  \sum_{\substack{ T \in \mathcal{T}_\infty \\ T \cap \Gamma \neq \emptyset}} \langle \widehat{g}_{\infty,\ell_T} \:,\: \operatorname{\mathrm{div}}\varphi\rangle_{L^2(T)} \\
  &\quad = 
  \sum_{\substack{T \in \mathcal{T}_\infty \\ T \cap \Gamma \neq 0}} \langle \widehat{g}_{\infty,\ell_T} \:,\: \varphi \cdot n_T\rangle_{L^2(\partial T \setminus \Gamma)} 
  - 
  \sum_{\substack{ T \in \mathcal{T}_\infty \\ T \cap \Gamma \neq \emptyset}} \langle \nabla \widehat{g}_{\infty,\ell_T} \:,\: \phi\rangle_{L^2(T)}. 
 \end{split}
\end{equation*}
It remains to show that the first sum on the right-hand side vanishes. 
Let $T_1,T_2 \in \mathcal{T}_\infty$ with $T_1 \cap \Gamma \neq \emptyset \neq T_2 \cap \Gamma$ such that $F \coloneqq T_1 \cap T_2 \in \mathcal{F}_\ell$ for some sufficiently large $\ell \in \mathbb{N}_0$ with $T_1,T_2 \in \mathcal{T}_\ell$. 
Without loss of generality, it holds that $\ell_{T_1} \leq \ell_{T_2}$. 
Since $\mathcal{T}_{\ell_{T_1}}[T_1] \subseteq \mathcal{T}_\infty$ and $T_2 \in \mathcal{T}_\infty$, with $T_1 \cap T_2 \in \mathcal{F}_\ell$, it follows that $T_2 \in \mathcal{T}_{\ell_{T_1}}[T_1] = \mathcal{T}_{\ell_{T_2}}[T_1]$. Since $\widehat{g}_{\infty,\ell_{T_1}}$ depends only on the patch $\Omega_{\ell_{T_1}}[T_1]$ and $(g - V\phi_\infty^\star)|_{\Gamma_{\ell_{T_1}}[T_1]}$, it follows that $\widehat{g}_{\infty,\ell_{T_1}}|_{T_1} =\widehat{g}_{\infty,\ell_{T_2}}|_{T_1}$. 
Because of $\widehat{g}_{\infty,\ell_{T_2}} \in H^1(T_1 \cup T_2)$, we obtain 
\begin{equation*}
 \langle \widehat{g}_{\infty,\ell_{T_1}} \:,\: \varphi \cdot n_{T_1}\rangle_{L^2(F)} 
 = 
 \langle \widehat{g}_{\infty,\ell_{T_2}} \:,\: \varphi \cdot n_{T_1}\rangle_{L^2(F)}
 = 
 -\langle \widehat{g}_{\infty,\ell_{T_2}} \:,\: \varphi \cdot n_{T_2}\rangle_{L^2(F)}.
\end{equation*}
Since each interior facet is shared by precisely two elements, this guarantees that
\begin{equation*}
 \sum_{\substack{T \in \mathcal{T}_\infty \\ T \cap \Gamma \neq \emptyset}} \langle \widehat{g}_{\infty,\ell_T} \:,\: \varphi \cdot n_T\rangle_{L^2(\partial T \setminus \Gamma)} 
 = 
 0.
\end{equation*}
Therefore, $\widehat{g}_\infty$ is weakly differentiable with $\nabla \widehat{g}_\infty|_{\Omega \setminus \omega_\infty} = 0$ and $\nabla \widehat{g}_\infty|_T = \nabla\widehat{g}_{\infty,\ell_T}$ for all $T \in \mathcal{T}_\infty$ with $T \subseteq \overline{\omega}_\infty$ for all $\ell \in N_0$.

\textbf{Step~4 (\boldsymbol{$\widehat{g}_{\infty,\ell} \to \widehat{g}_\infty$} in \boldsymbol{$H^1(\Omega)$})}.
Let $\varepsilon > 0$. 
Since $L^2(\Gamma)$ is dense in $H^{-1/2}(\Gamma)$, there exists $\varphi_\varepsilon \in L^2(\Gamma)$ with $\lVert (\phi^\star - \phi_\infty^\star) - \varphi_\varepsilon\rVert_{H^{-1/2}(\Gamma)} \leq \varepsilon$. 
We start by showing the existence of $\ell_0 \in \mathbb{N}_0$ such that 
\begin{equation} \label{eq:intbound}
 \lVert \widehat{g}_{\infty,\ell} - \widehat{g}_\infty\rVert_{L^2(\Omega \setminus \omega_\infty)} 
 \lesssim
 \varepsilon
 \quad \text{for all } \ell \geq \ell_0.
\end{equation} 
We have
\begin{equation} \label{eq:intconv}
 \begin{split}
 &\lVert \nabla (\widehat{g}_{\infty,\ell} - \widehat{g}_\infty)\rVert_{L^2(\Omega \setminus \omega_\infty)} 
 = 
 \lVert \nabla\widehat{g}_{\infty,\ell}\rVert_{L^2(\omega_\ell \setminus \omega_\infty)} \\
 &\qquad \leq
 \lVert \nabla \widehat{J_\ell} (V\varphi_\varepsilon - Q_\ell^\Gamma V \varphi_\varepsilon)\rVert_{L^2(\omega_\ell \setminus \omega_\infty)} + \lVert \nabla (\widehat{g}_{\infty,\ell} - \widehat{J_\ell} (V \varphi_\varepsilon - Q_\ell^\Gamma \varphi_\varepsilon))\rVert_{L^2(\Omega)} 
 \end{split}
\end{equation}
We first consider the first summand of the right-hand side of~\eqref{eq:intconv}. 
Due to Proposition~\ref{prop:extbound}, there exists $r > 2$ and $1/2 < s_r \leq 1$ such that 
\begin{equation} \label{eq:phiepsesti}
 \lVert \nabla \widehat{J_\ell} (V\varphi_\varepsilon - Q_\ell^\Gamma V \varphi_\varepsilon)\rVert_{L^r(\Omega)} 
 \stackrel{\eqref{eq:boundarypoincare}}{\lesssim} 
 \lvert V \varphi_\varepsilon\rvert_{H^{s_r}(\Gamma)} 
 \lesssim
 \lVert \varphi_\varepsilon\rVert_{H^{s_r - 1}(\Gamma)} 
 \lesssim 
 \lVert \varphi_\varepsilon\rVert_{L^2(\Gamma)}.
\end{equation} 
The H\"older inequality yields
\begin{equation*}
 \begin{split}
  \lVert \nabla \widehat{J_\ell} (V\varphi_\varepsilon - Q_\ell^\Gamma V \varphi_\varepsilon)\rVert_{L^2(\omega_\ell \setminus \omega_\infty)} 
  &\leq 
  \lvert \omega_\ell \setminus \omega_\infty\rvert^{1/2 - 1/r} \lVert \nabla\widehat{J_\ell} (V\varphi_\varepsilon - Q_\ell^\Gamma V \varphi_\varepsilon)\rVert_{L^r(\omega_\ell \setminus \omega_\infty)} \\
  &\lesssim
  \lvert \omega_\ell \setminus \omega_\infty\rvert^{1/2 - 1/r} \lVert \varphi_\varepsilon\rVert_{L^2(\Gamma)}.
 \end{split}
\end{equation*}
Since $\omega_\infty = \displaystyle{\bigcap_{\ell \in \mathbb{N}_0} \omega_\ell}$ and $\varepsilon$ is fixed, there exists $\ell_0 \in \mathbb{N}_0$ such that $\lvert \omega_\ell \setminus \omega_\infty\rvert^{1/2 - 1/r}\lVert \varphi_\varepsilon\rVert_{L^2(\Gamma)} \leq \varepsilon$ for all $\ell \geq \ell_0$. 
For the second summand of~\eqref{eq:intconv}, we exploit $Q_\ell^\Gamma(V \phi^\star - V\phi_\infty^\star) = 0$ (see Lemma~\ref{lemma:apriori}) and apply Proposition~\ref{prop:extbound} with $r=2$ and $s_r = 1/2$ to obtain
\begin{equation} \label{eq:perturbation}
 \begin{split}
  &\lVert \nabla \widehat{g}_{\infty,\ell} - \widehat{J_\ell} (V \varphi_\varepsilon - Q_\ell^\Gamma \varphi_\varepsilon)\rVert_{L^2(\Omega)}
  =
  \lVert \nabla \widehat{J_\ell}(V(\phi^\star - \phi_\infty^\star - \varphi_\varepsilon) - Q_\ell^\Gamma V(\phi^\star - \phi_\infty^\star - \varphi_\varepsilon))\rVert_{L^2(\Omega)} \\
  &\qquad \stackrel{\eqref{eq:boundarypoincare}}{\lesssim} 
  \lvert V(\phi^\star - \phi_\infty^\star - \varphi_\varepsilon)\rvert_{H^{1/2}(\Gamma)} 
  \lesssim 
  \lVert \phi^\star - \phi_\infty^\star - \varphi_\varepsilon\rVert_{H^{-1/2}(\Gamma)} \leq \varepsilon.
 \end{split}
\end{equation}
Altogether, this yields~\eqref{eq:intbound}. Furthermore, we have that 
\begin{equation} \label{eq:omegainftyesti}
 \begin{split}  
  &\lVert \nabla(\widehat{g}_{\infty,\ell} - \widehat{g}_\infty)\rVert_{L^2(\omega_\infty)}^2 
  = 
  \sum_{\substack{T \in \mathcal{T}_\infty \\ T \cap \Gamma \neq \emptyset}} \lVert \nabla(\widehat{g}_{\infty,\ell} - \widehat{g}_\infty)\rVert_{L^2(T)}^2 
  = 
  \sum_{\substack{T \in \mathcal{T}_\infty \\ T \cap \Gamma \neq \emptyset \\ \ell_T > \ell}} \lVert \nabla (\widehat{g}_{\infty,\ell} - \widehat{g}_\infty)\rVert_{L^2(T)}^2 \\
  &\qquad = 
  \sum_{\substack{T \in \mathcal{T}_\infty \\ T \cap \Gamma \neq \emptyset \\ T \notin \mathcal{T}_\ell^0}} \lVert \nabla (\widehat{g}_{\infty,\ell} - \widehat{g}_\infty)\rVert_{L^2(T)}^2 
  \leq 
  \sum_{T \in \mathcal{T}_\infty \setminus \mathcal{T}_\ell^0} \lVert \nabla(\widehat{g}_{\infty,\ell} - \widehat{g}_\infty)\rVert_{L^2(T)}^2 \\
  &\qquad \lesssim 
  \sum_{T \in \mathcal{T}_\infty \setminus \mathcal{T}_\ell^0} \lVert \nabla \widehat{g}_{\infty,\ell}\rVert_{L^2(T)}^2 
  + 
  \sum_{T \in \mathcal{T}_\infty \setminus \mathcal{T}_\ell^0} \lVert \nabla \widehat{g}_\infty\rVert_{L^2(T)}^2 \\
  &\qquad = 
  \lVert \nabla \widehat{g}_{\infty,\ell}\rVert_{L^2(\bigcup[\mathcal{T}_\infty \setminus \mathcal{T}_\ell^0])}^2 + \sum_{T \in \mathcal{T}_\infty \setminus \mathcal{T}_\ell^0} \lVert \nabla \widehat{g}_\infty\rVert_{L^2(T)}^2.
 \end{split}
\end{equation}
Considering the first term right-hand side of~\eqref{eq:omegainftyesti}, we have
\begin{equation} \label{eq:omegainftyesti1}
 \begin{split}
  \lVert \nabla \widehat{g}_{\infty,\ell}\rVert_{L^2(\bigcup[\mathcal{T}_\infty \setminus \mathcal{T}_\ell^0])}
  &\leq
  \lVert \nabla \widehat{J_\ell} (V\varphi_\varepsilon - Q_\ell^\Gamma V \varphi_\varepsilon)\rVert_{L^2(\bigcup[\mathcal{T}_\infty \setminus \mathcal{T}_\ell^0])} \\
  &\qquad + \lVert \nabla (\widehat{g}_{\infty,\ell} - \widehat{J_\ell} (V \varphi_\varepsilon - Q_\ell^\Gamma \varphi_\varepsilon))\rVert_{L^2(\Omega)} \\
  &\stackrel{\eqref{eq:perturbation}}{\lesssim}
  \lVert \nabla \widehat{J_\ell} (V\varphi_\varepsilon - Q_\ell^\Gamma V \varphi_\varepsilon)\rVert_{L^2(\bigcup[\mathcal{T}_\infty \setminus \mathcal{T}_\ell^0])} + \varepsilon.
 \end{split}
\end{equation}
H\"older's inequality together with~\eqref{eq:phiepsesti} yields
\begin{equation*}
 \begin{split}
  &\lVert \nabla \widehat{J_\ell} (V\varphi_\varepsilon - Q_\ell^\Gamma V \varphi_\varepsilon)\rVert_{L^2(\bigcup[\mathcal{T}_\infty \setminus \mathcal{T}_\ell^0])} \\
  &\qquad \leq
  \big|\bigcup[\mathcal{T}_\infty \setminus \mathcal{T}_\ell^0]\big|^{1/2 - 1/r} \lVert \nabla \widehat{J_\ell} (V\varphi_\varepsilon - Q_\ell^\Gamma V \varphi_\varepsilon)\rVert_{L^r(\bigcup[\mathcal{T}_\infty \setminus \mathcal{T}_\ell^0])} \\
  &\qquad \lesssim
  \big|\bigcup[\mathcal{T}_\infty \setminus \mathcal{T}_\ell^0]\big|^{1/2 - 1/r} \lVert \varphi_\varepsilon\rVert_{L^2(\Gamma)}.
 \end{split}
\end{equation*}
Since $\mathcal{T}_\infty = \displaystyle{\bigcup_{\ell \in \mathbb{N}_0} \mathcal{T}_\ell^0}$ and $\varepsilon$ is fixed, there exists $\ell_1 \in \mathbb{N}_0$ such that 
\begin{equation*}
 \Big| \displaystyle{\bigcup [\mathcal{T}_\infty \setminus \mathcal{T}_\ell^0]}\Big|^{1/2 - 1/r} \lVert \varphi_\varepsilon\rVert_{L^2(\Gamma)} \leq \varepsilon \quad \text{for all } \ell \geq \ell_1.
\end{equation*} 
Hence, we see 
\begin{equation*}
 \lVert \nabla \widehat{g}_{\infty,\ell}\rVert_{L^2(\bigcup[\mathcal{T}_\infty \setminus \mathcal{T}_\ell^0])}
 \lesssim
 \varepsilon 
 \quad \text{for } \ell \geq \max\{\ell_0,\ell_1\}.
\end{equation*}
The second term on the right-hand side of~\eqref{eq:omegainftyesti1} can be estimated analogously as in~\eqref{eq:perturbation}. 
It remains to estimate the second term on the right-hand side of~\eqref{eq:omegainftyesti}. 
Due to the definition of $\widehat{g}_\infty$, the lifting estimate~\eqref{eq:liftcont} of Lemma~\ref{lem:discnorm}, Lemma~\ref{lem:SZ}, and the properties~\eqref{eq:avgprop} of $Q_\ell^\Gamma$, we obtain
\begin{equation*}
 \begin{split}
  &\lVert \nabla \widehat{g}_{\infty}\rVert_{L^2(T)}
  = 
  \lVert \nabla \widehat{g}_{\infty,\ell_T}\rVert_{L^2(T)}
  \stackrel{\eqref{eq:liftcont}}{\lesssim} 
  h_T^{-1/2} \lVert g_{\infty,\ell_T}\rVert_{L^2(\partial T \cap \Gamma)} + \lvert g_{\infty,\ell_T}\rvert_{H^{1/2}(\partial T \cap \Gamma)} \\
  &\qquad \stackrel{\eqref{lem:SZ}}{\lesssim}
  h_T^{-1/2} \lVert g - V\phi_\infty^\star\rVert_{L^2(\Gamma_{\ell_T}[T])} + \lvert g - V\phi_\infty^\star\rvert_{H^{1/2}(\Gamma_{\ell_T}[T])}
  \stackrel{\eqref{eq:avgprop}}{\lesssim}
  \lvert g - V\phi_\infty^\star\rvert_{H^{1/2}(\Gamma_{\ell_T}[T])}.
 \end{split}
\end{equation*}
Therefore, $\kappa$-shape regularity establishes
\begin{equation*}
 \sum_{T \in \mathcal{T}_\infty} \lVert \nabla \widehat{g}_\infty\rVert_{L^2(T)}^2 
 \lesssim 
 \sum_{T \in \mathcal{T}_\infty} \lvert g - V\phi_\infty^\star\rvert_{H^{1/2}(\Gamma_\ell[T])}^2 
 \lesssim 
 \lvert g - V\phi_\infty^\star\rvert_{H^{1/2}(\Gamma)}^2.
\end{equation*}
Since $\mathcal{T}_\infty = \displaystyle{\bigcup_{\ell \in \mathbb{N}_0} \mathcal{T}_\ell^0}$, we can choose $\ell_2 \in \mathbb{N}_0$ such that
\begin{equation*}
 \sum_{T \in \mathcal{T}_\infty \setminus \mathcal{T}_\ell^0} \lVert \nabla \widehat{g}_\infty\rVert_{L^2(T)}^2
 \leq
 \varepsilon
 \quad \text{for all } \ell \geq \ell_2.
\end{equation*}
Altogether, this proves $\lVert \nabla(\widehat{g}_{\infty,\ell} - \widehat{g}_\infty)\rVert_{L^2(\Omega)} \lesssim \varepsilon$ for all $\ell \geq \max\{\ell_0,\ell_1,\ell_2\}$ yielding $\widehat{g}_{\infty,\ell} \xrightarrow{\ell \to \infty} \widehat{g}_\infty$ and, in particular, $\widehat{g}_\infty \in H^1(\Omega)$.

\textbf{Step~5 (\boldsymbol{$\widehat{g}_{\ell} \to \widehat{g}_\infty$} in \boldsymbol{$H^1(\Omega)$})}. We have
\begin{equation*}
 \begin{split}
  \lVert \nabla(\widehat{g}_\ell - \widehat{g}_\infty)\rVert_{L^2(\Omega)} 
  &\leq 
  \lVert \nabla(\widehat{g}_{\infty,\ell} - \widehat{g}_\infty)\rVert_{L^2(\Omega)} + \lVert \nabla\widehat{J_\ell}V(\phi_\infty^\star - \phi_\ell^\star)\rVert_{L^2(\Omega)}.
 \end{split}
\end{equation*}
The first term on the right-hand side converges to zero as $\ell \to \infty$ by Step~4. 
Note that $Q_\ell^\Gamma V (\phi_\infty^\star - \phi_\ell^\star) = 0$. 
For the second term, we thus argue as in Step~4, and exploit the \sl{a~priori} convergence of Lemma~\ref{lemma:apriori} to obtain that
\begin{equation*}
 \begin{split}
  &\lVert \nabla\widehat{J_\ell}V(\phi_\infty^\star - \phi_\ell^\star)\rVert_{L^2(\Omega)}^2 
  = 
  \sum_{\substack{T \in \mathcal{T}_\ell \\ T \cap \Gamma \neq \emptyset}} \lVert \nabla\widehat{J_\ell}V(\phi_\infty^\star - \phi_\ell^\star)\rVert_{L^2(T)}^2 
  \lesssim 
  \sum_{F \in \mathcal{T}_\ell^\Gamma} \lvert  V(\phi_\infty^\star - \phi_\ell^\star)\rvert_{H^{1/2}(\Gamma_\ell[F])}^2 \\
  &\qquad \lesssim 
  \lvert V\phi_\infty^\star - V\phi_\ell^\star\rvert_{H^{1/2}(\Gamma)}^2 \xrightarrow{\ell \to \infty} 0.
 \end{split}
\end{equation*}
This concludes the proof.
\end{proof}

%%%%%%%%%%%%%%%%%%%%%%%%%%%%%%%%%%%%%%%%%%%%%%%%%%%%%%%%%%%%%%%%%%%%%%%%%%%%%%%%%%%
%%%%%%%%%%%%%%%%%%%%%%%%%%%%%%%%%%%%%%%%%%%%%%%%%%%%%%%%%%%%%%%%%%%%%%%%%%%%%%%%%%%
\section{Numerical experiments}
\label{section:numerics}
%%%%%%%%%%%%%%%%%%%%%%%%%%%%%%%%%%%%%%%%%%%%%%%%%%%%%%%%%%%%%%%%%%%%%%%%%%%%%%%%%%%
%%%%%%%%%%%%%%%%%%%%%%%%%%%%%%%%%%%%%%%%%%%%%%%%%%%%%%%%%%%%%%%%%%%%%%%%%%%%%%%%%%%

This section presents some numerical experiments in $2$D that illustrate the accuracy and performance of the proposed method. 
All computations are done by use of the MATLAB-toolboxes HILBERT from~\cite{Aurada2014} for BEM with $p \in \{0,1\}$,~\cite{Funken2011} for standard FEM for the functional upper bound~\eqref{eq:compupperbound} with $q \in \{1,2\}$, and~\cite{Bahriawati2005} for mixed FEM for the functional lower bound~\eqref{eq:lower-bound} with $q = 1$. 
Throughout, we consider Algorithm~\ref{algo:adap} for uniform ($\theta = 1$) and adaptive ($0 < \theta < 1$) mesh refinement.  

\begin{figure}[!ht]
  \resizebox{\textwidth}{!}{
   \subfloat{
    \stackunder{\stackunder[10pt]{\includegraphics{numerics/Plots/Square/squarereferencemesh.pdf}}{{\Huge$\# \mathcal{T}_\ell^\Gamma = 64$, $\#\mathcal{T}_\ell^\omega = 224$}}}{{\Huge $\ell=0$}}
   }
   \subfloat{
    \stackunder{\stackunder[10pt]{\includegraphics{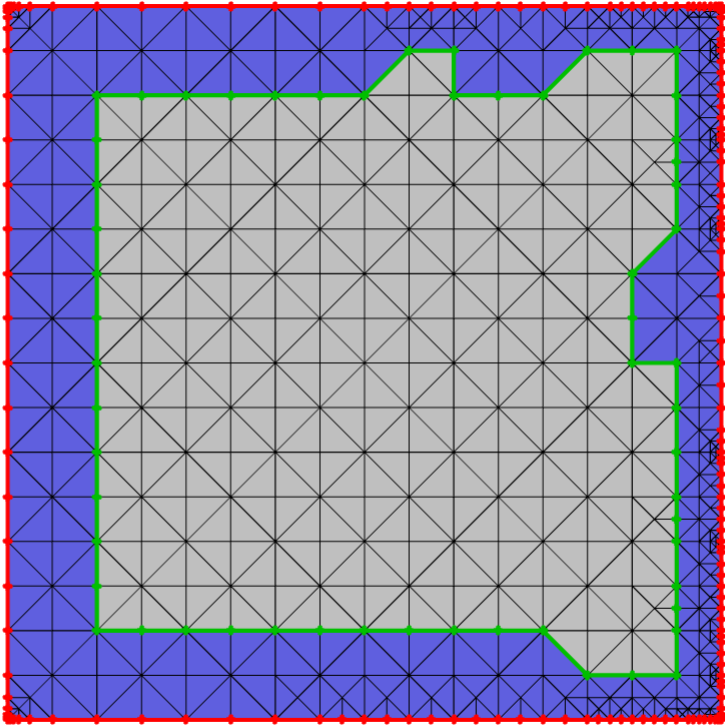}}{{\Huge $\# \mathcal{T}_\ell^\Gamma = 210$, $\#\mathcal{T}_\ell^\omega = 685$}}}{{\Huge $\ell=30$}}
   }
   
   \subfloat{
    \stackunder{\stackunder[10pt]{\includegraphics{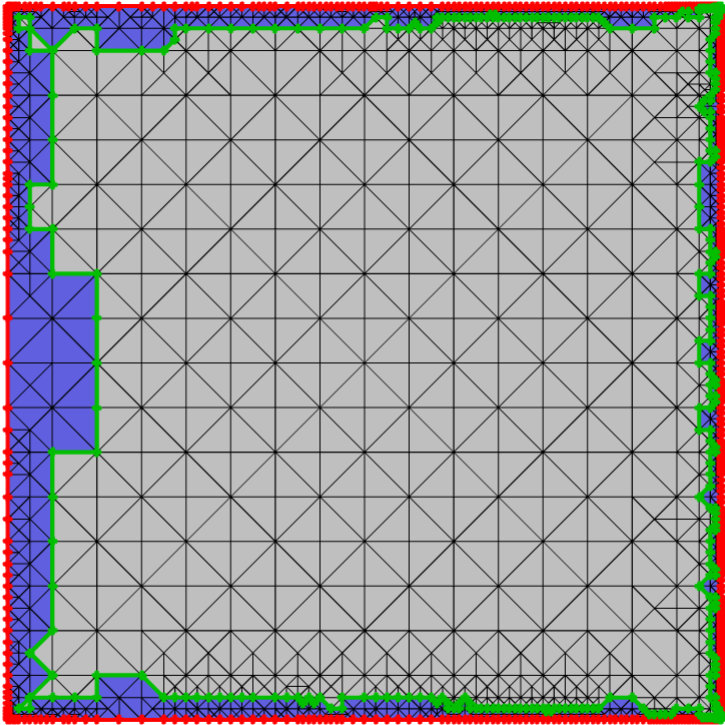}}{{\Huge $\# \mathcal{T}_\ell^\Gamma = 704$, $\#\mathcal{T}_\ell^\omega = 2709$}}}{{\Huge $\ell=50$}}
   }
   \subfloat{
    \stackunder{\stackunder[10pt]{\includegraphics{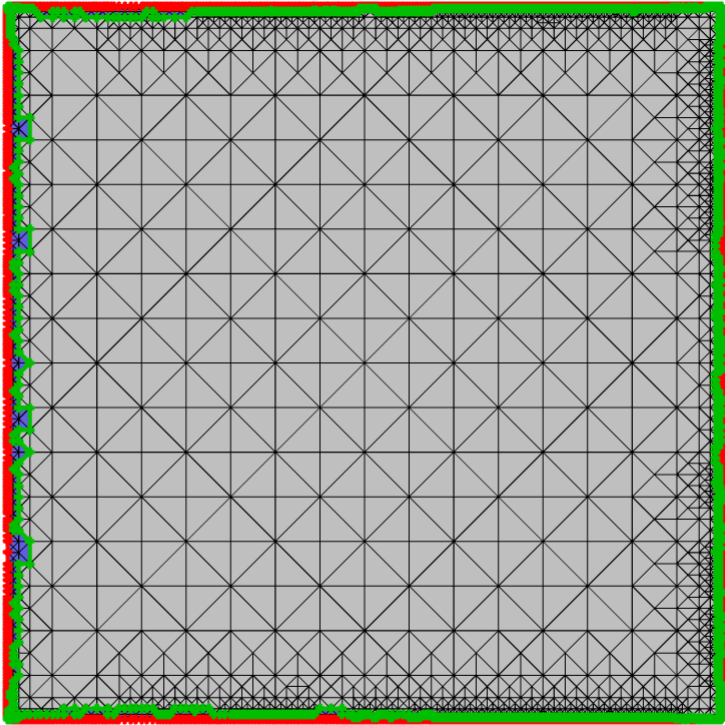}}{{\Huge $\# \mathcal{T}_\ell^\Gamma = 4936$, $\#\mathcal{T}_\ell^\omega = 20201$}}}{{\Huge $\ell=80$}}
   }
  }
   \caption{Adaptively generated meshes in Example~\ref{subsec:squareex} for $p=0$, $q=1$, $\theta = 0.5$, and $k = 3$. The triangles of the boundary strip $T \in \mathcal{T}_\ell^\omega$ are depicted in {\color{blue} blue}, while the triangles $T \in \mathcal{T}_\ell \setminus \mathcal{T}_\ell^\omega$ are indicated in {\color{gray} gray}. The outer boundary $\Gamma$ is shown in {\color{red} red} and the inner boundary $\Gamma_\ell^{\mathsf{c}}$ in {\color{green} green}.}
  
   \label{fig:squaremeshes}
 \end{figure}

%%%%%%%%%%%%%%%%%%%%%%%%%%%%%%%%%%%%%%%%%%%%%%%%%%%%%%%%%%%%%%%%%%%%%%%%%%%%%%%%%%%%%%%%%%%
\subsection{Example 1 (Square domain, smooth potential)} \label{subsec:squareex}
%%%%%%%%%%%%%%%%%%%%%%%%%%%%%%%%%%%%%%%%%%%%%%%%%%%%%%%%%%%%%%%%%%%%%%%%%%%%%%%%%%%%%%%%%%%

We consider the Laplace--Dirichlet problem~\eqref{eq:poisson} with the prescribed exact solution 
\begin{equation*}
 u^\star(x) 
 = 
 \sinh(2\pi x_1)\cos(2 \pi x_2)
 \quad \text{in }
 \Omega \coloneqq (0,1/2)^2
\end{equation*}
on the square domain $\Omega \subset \mathbb{R}^2$ which fulfills the scaling condition $\mathrm{diam}(\Omega) = 1/\sqrt{2} < 1$. 

Even though the exact solution $u^\star$ and the data $g \coloneqq u^\star|_{\Gamma}$ are smooth, one cannot expect that the density $\phi^\star$ stemming from the indirect boundary integral formulation~\eqref{eq:bie} is smooth, as it inherits generic singularities stemming from the interior and the exterior problem. 
As a consequence, uniform mesh-refinement will not lead to optimal convergence behavior $\lVert \phi^\star - \phi_\ell^\star\rVert_{H^{-1/2}(\Omega)} = \mathcal{O}((\#\mathcal{T}_\ell^\Gamma)^{-(p + 3/2)})$ for Galerkin BEM~\eqref{eq:weakdisc} with $p \in \mathbb{N}_0$.

We start Algorithm~\ref{algo:adap} with a coarse initial triangulation $\mathcal{T}_0$ comprised of $512$ triangles. 
The initial mesh, some adaptively refined meshes, and the corresponding boundary strips are visualized in Figure~\ref{fig:squaremeshes}.

Figure~\ref{fig:squarerates} shows the total upper bound $\eta_\ell + \mathrm{osc}_\ell$ for lowest-order BEM $p = 0$ computed with $k=3$ for $q=1,2$ and different marking parameters $0 < \theta < 1$. 
We observe that Algorithm~\ref{algo:adap} exhibits optimal convergence rates $\lVert \nabla (u^\star - u_\ell^\star)\rVert_{L^2(\Omega)} = \mathcal{O}(\eta_\ell + \mathrm{osc}_\ell) = \mathcal{O}((\#\mathcal{T}_\ell^\Gamma)^{-3/2})$ for any marking parameter $\theta \in \{0.2,0.4,0.6,0.8\}$ with only minimal change with respect to the particular choice of $\theta$. 

\begin{figure}[!ht]
 \resizebox{\textwidth}{!}{
  \subfloat{
   \includegraphics{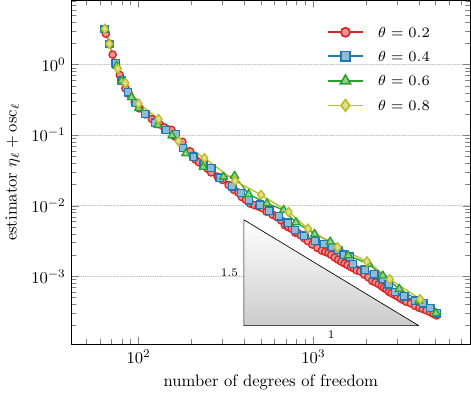}
  }
  \subfloat{
   \includegraphics{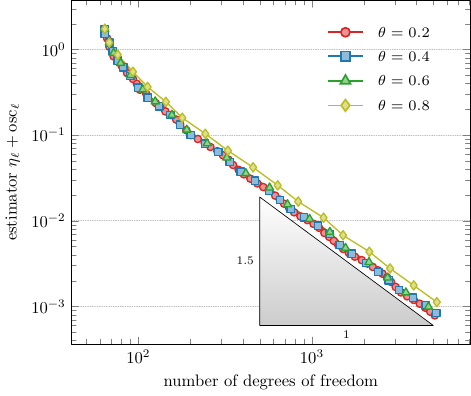}
  }
 }
  \caption{Convergence rates of the full error estimator $\eta_\ell + \mathrm{osc}_\ell$ in Example~\ref{subsec:squareex} for lowest-order BEM $p = 0$ with $k=3$ and different marking parameters $0 < \theta < 1$ for $q=1$ (left) and $q=2$ (right).}
  \label{fig:squarerates}
\end{figure}

A comparison of different strip depths $k \in \{1, \ldots , 6\}$ and how the choice of $k$ influences the functional estimators $\eta_\ell$ and $\tau_\ell$, as well as the oscillations $\mathrm{osc}_\ell$, is shown in Figure~\ref{fig:squarekerror}. 
Additionally, the (approximate) potential error $\lVert \nabla(u^\star - u_\ell^\star)\rVert_{L^2(\Omega)} \approx \lVert \nabla I_\ell(u^\star - u_\ell^\star)\rVert_{L^2(\Omega)}$, where $I_\ell \colon C(\overline{\Omega}) \to \mathcal{S}^2(\widehat{\mathcal{T}}_\ell)$ denotes the $\mathcal{P}^2$-nodal interpolation operator with respect to the mesh $\widehat{\mathcal{T}}_\ell$, which is obtained by uniformly refining $\mathcal{T}_\ell$, is shown. 
As expected, the quality of the estimator $\eta_\ell$ improves with increasing strip depth $k$; see also Remark~\ref{rem:whjustification}.

\begin{figure}[!ht]
 \resizebox{\textwidth}{!}{
  \subfloat{
    \includegraphics{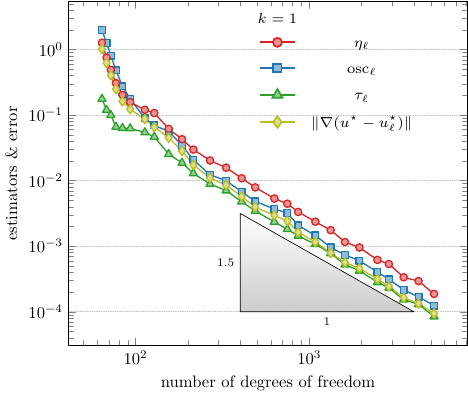}
  }
  \subfloat{
    \includegraphics{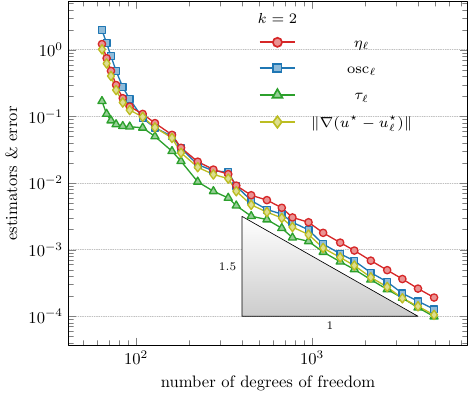}
  }
  
  \subfloat{
    \includegraphics{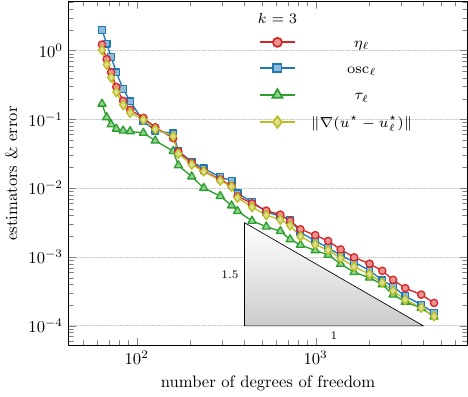}
  }
 }
  \\
  \resizebox{\textwidth}{!}{
  \subfloat{
    \includegraphics{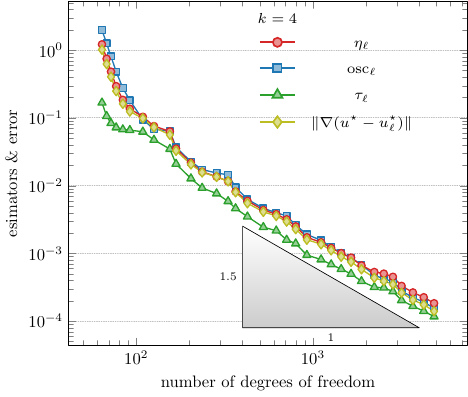}
  }
  \subfloat{
    \includegraphics{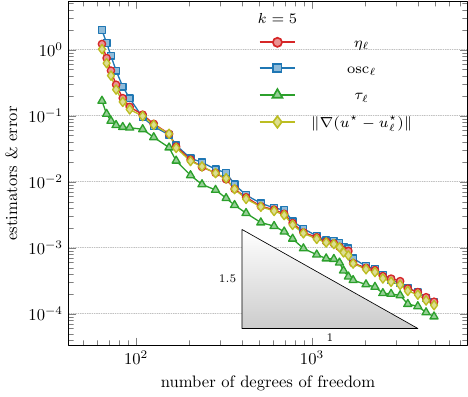}
  }
  \subfloat{
    \includegraphics{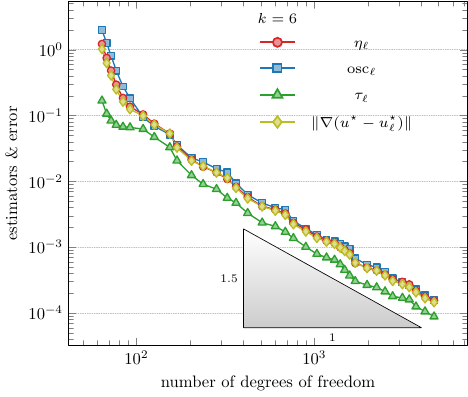}
  }
  }
  \caption{Functional error bounds for different strip depths $k \in \{1, \ldots , 6\}$ in Example~\ref{subsec:squareex} with $p=0$, $q=1$, and $\theta = 0.5$. We compare the estimators $\eta_\ell$, $\tau_\ell$, and $\mathrm{osc}_\ell$, as well as the (approximate) potential error $\lVert \nabla(u^\star - u_\ell^\star)\rVert_{L^2(\Omega)}$.}
  \label{fig:squarekerror}
\end{figure}

For a more thorough look at the accuracy of the estimator $\eta_\ell$, we refer to Figure~\ref{fig:squareefficiency}, where we compare the experimental reliability constants $\lVert \nabla(u^\star - u_\ell^\star)\rVert_{L^2(\Omega)}/\eta_\ell$ for different strip depths $k \in \{1, \ldots , 6\}$ and different polynomial degrees $q =1,2$. 
We observe that higher-order auxiliary problems achieve better accuracy even for lower strip depths $k$.

\begin{figure}[!ht]
  \resizebox{\textwidth}{!}{
    \subfloat{
      \includegraphics{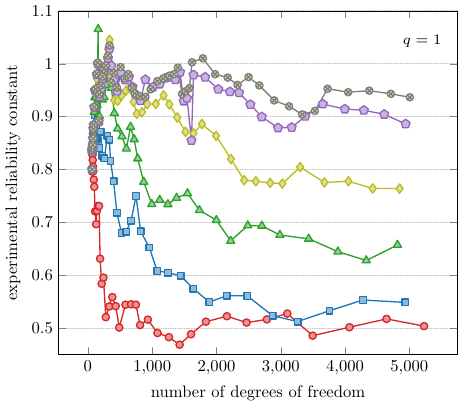}
    }

    \subfloat{
      \includegraphics{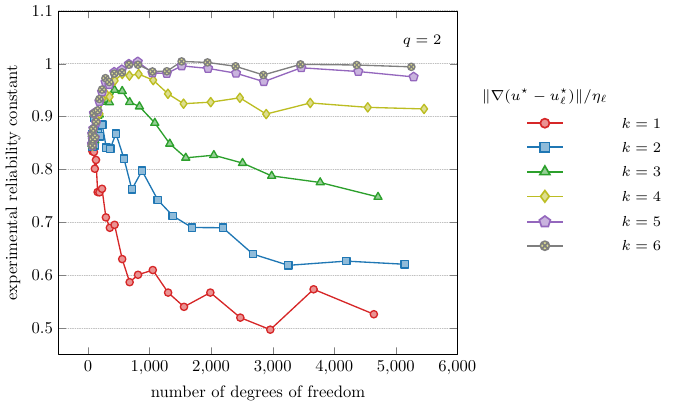}
    }
  }
 \caption{Experimental reliability constants $\lVert \nabla(u^\star - u_\ell^\star)\rVert_{L^2(\Omega)}/\eta_\ell$ in Example~\ref{subsec:squareex} with $p=0$, $\theta = 0.5$, $k \in \{1, \ldots , 6\}$, $q=1$ (left), and $q=2$ (right).}
 \label{fig:squareefficiency}
\end{figure}

Figure~\ref{fig:squarep1} depicts the convergence behaviour of $\mathcal{P}1$-BEM for $q=2$, $k=3$, and different adaptivity parameters $\theta \in \{0.2,0.4,0.6,0.8\}$, as well as the different functional estimators $\eta_\ell$, $\mathrm{osc}_\ell$, and $\tau_\ell$ together with the (approximate) potential error $\lVert \nabla (u^\star - u_\ell^\star)\rVert_{L^2(\Omega)}$. 
Similiar to the lowest-order case, we observe optimal convergence rates $\lVert \nabla(u^\star - u_\ell^\star)\rVert_{L^2(\Omega)} = \mathcal{O}(\eta_\ell + \mathrm{osc}_\ell) = \mathcal{O}((\#\mathcal{T}_\ell^\Gamma)^{-5/2})$. 
Moreover, we see that the functional upper bound $\eta_\ell$ achieves high accuracy also for $p=1$.

\begin{figure}[!ht]
\resizebox{\textwidth}{!}{
  \subfloat{
    \includegraphics{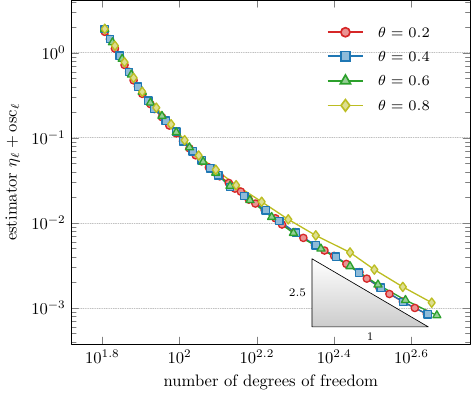}
  }

\subfloat{
    \includegraphics{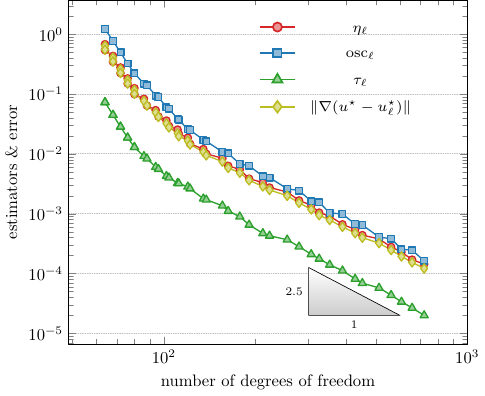}
  }
}
  \caption{Convergence rates of the full error estimator $\eta_\ell + \mathrm{osc}_\ell$ in Example~\ref{subsec:squareex} for $\mathcal{P}1$-BEM $p = 1$ with $k=3$, $q=2$, and different adaptivity parameters $\theta$ (left), as well as the estimators $\eta_\ell$, $\mathrm{osc}_\ell$, and $\tau_\ell$, as well as the (approximate) potential error $\lVert \nabla(u^\star - u_\ell^\star)\rVert_{L^2(\Omega)}$ for $\mathcal{P}1$-BEM with $\theta = 0.5$, $q=2$, and $k = 6$ (right).}
  \label{fig:squarep1}
\end{figure}

%%%%%%%%%%%%%%%%%%%%%%%%%%%%%%%%%%%%%%%%%%%%%%%%%%%%%%%%%%%%%%%%%%%%%%%%%%%%%%%%%%%%%%%%%%%
\subsection{Example 2 (Z-shaped domain, smooth potential)} \label{subsec:zshapeex}
%%%%%%%%%%%%%%%%%%%%%%%%%%%%%%%%%%%%%%%%%%%%%%%%%%%%%%%%%%%%%%%%%%%%%%%%%%%%%%%%%%%%%%%%%%%%

We consider the Laplace--Dirichlet problem~\eqref{eq:poisson} with constant Dirichlet data $g = 1$ on the Z-shaped domain
\begin{equation*}
 \Omega \coloneqq (-1/4,1/4)^2 \setminus \{(x,y) \in \mathbb{R}^2 \:\vert\:\mathopen{} -1/4 \leq x \leq 0, -1/4 \leq y \leq -1/4 - x\}.
\end{equation*}
Although the exact solution $u^\star$ is smooth (even constant), the solution $\phi^\star$ to the indirect boundary integral formulation~\eqref{eq:bie} inherits generic singularities from the exterior problem, which occur at the convex corners of the domain (but not at the reentrant corner). 
We start Algorithm~\ref{algo:adap} with an initial partition $\mathcal{T}_0$ into $1792$ triangles. 

Figure~\ref{fig:zshape1meshes} depicts some adaptively generated meshes, which show that the adaptive mesh-refinement algorithm is able to capture the singularities of the integral density $\phi^\star$. 

\begin{figure}[!ht]
  \resizebox{\textwidth}{!}{
   \subfloat{
    \stackunder{\stackunder[10pt]{\includegraphics{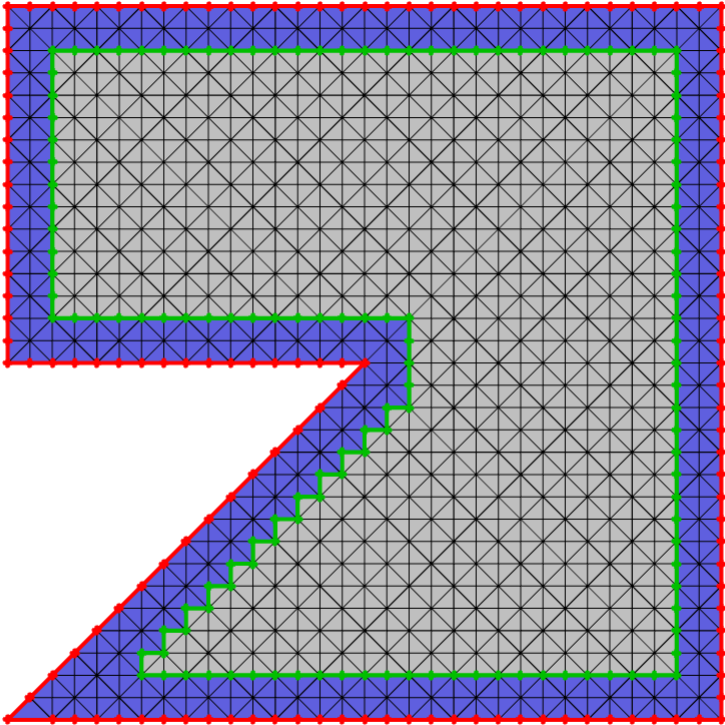}}{{\Huge $\# \mathcal{T}_\ell^\Gamma = 143$, $\#\mathcal{T}_\ell^\omega = 579$}}}{{\Huge $\ell=0$}}
   }
   \subfloat{
    \stackunder{\stackunder[10pt]{\includegraphics{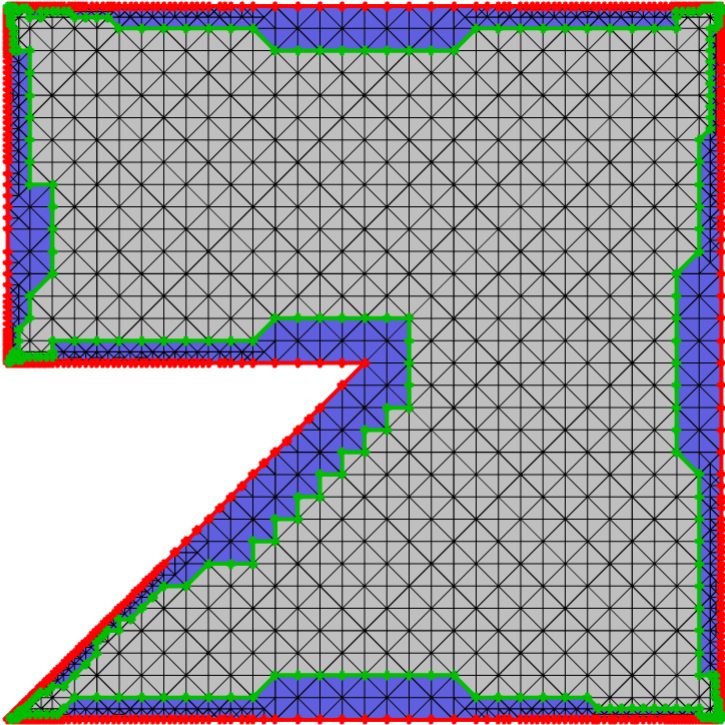}}{{\Huge $\# \mathcal{T}_\ell^\Gamma = 925$, $\#\mathcal{T}_\ell^\omega = 3248$}}}{{\Huge $\ell=50$}}
   }
   
   \subfloat{
    \stackunder{\stackunder[10pt]{\includegraphics{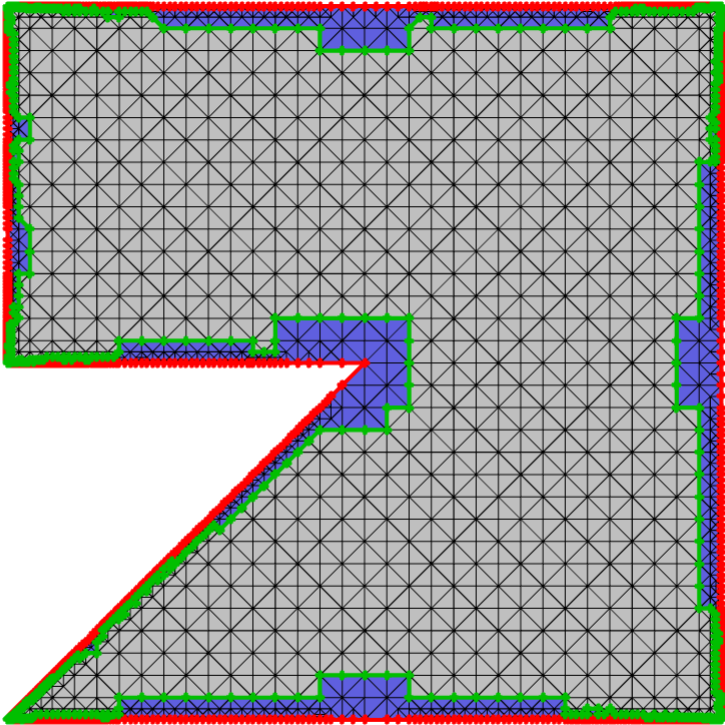}}{{\Huge $\# \mathcal{T}_\ell^\Gamma = 2247$, $\#\mathcal{T}_\ell^\omega = 8138$}}}{{\Huge $\ell=70$}}
   }
   \subfloat{
    \stackunder{\stackunder[10pt]{\includegraphics{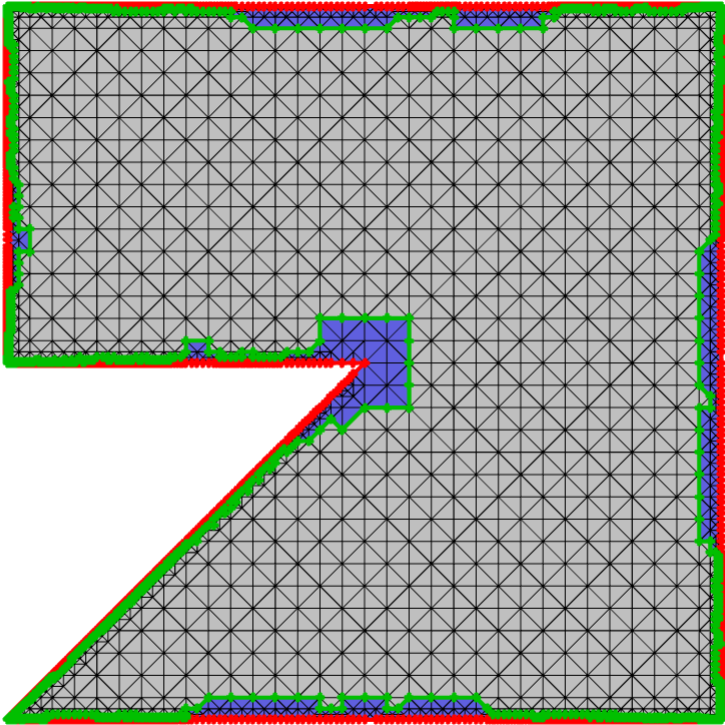}}{{\Huge $\# \mathcal{T}_\ell^\Gamma = 4313$, $\#\mathcal{T}_\ell^\omega = 15856$}}}{{\Huge $\ell=90$}}
   }
  }
   \caption{Adaptively generated meshes in Example~\ref{subsec:zshapeex} for $p=0$, $q=1$, $\theta = 0.5$, and $k = 3$. The triangles of the boundary strip $T \in \mathcal{T}_\ell^\omega$ are depicted in {\color{blue} blue}, while the triangles $T \in \mathcal{T}_\ell \setminus \mathcal{T}_\ell^\omega$ are indicated in {\color{gray} gray}. The outer boundary $\Gamma$ is shown in {\color{red} red} and the inner boundary $\Gamma_\ell^{\mathsf{c}}$ in {\color{green} green}.}
  
   \label{fig:zshape1meshes}
 \end{figure}

The convergence behaviour for different adaptivity parameters $\theta \in \{0.2,0.4,0.6,0.8\}$ is shown in Figure~\ref{fig:zshape1rates}. 
Since $q=1$ and $q=2$ lead to similar results in terms of convergence rates of $\mathcal{P}0$-BEM and $\mathcal{P}1$-BEM, we only show the results for $q=1$. 
As in Example~\ref{subsec:squareex}, we observe optimal convergence rates for $p = 0,1$.

\begin{figure}[!ht]
  \resizebox{\textwidth}{!}{
    \subfloat{
      \includegraphics{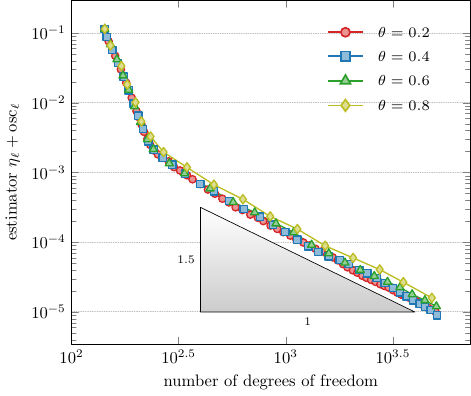}
    }
    \subfloat{
      \includegraphics{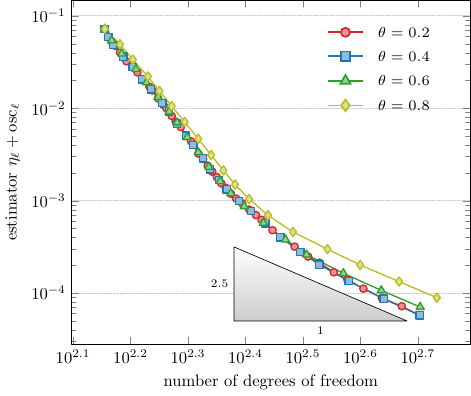}
    }
  }
  \caption{Convergence rates of the full error estimator $\eta_\ell + \mathrm{osc}_\ell$ in Example~\ref{subsec:zshapeex} for $p = 0$, $q = 1$, and $k = 3$ (left), as well as $p = 1$, $q=2$, and $k = 3$ (right) for different adaptivity parameters $\theta$.}
  \label{fig:zshape1rates}
\end{figure}

Figure~\ref{fig:zshapekerror} shows a comparison of different strip depths $k \in \{1, \ldots , 6\}$ for Example~\ref{subsec:zshapeex} for $q=2$. 
We observe that, even for small $k$, the functional estimator $\eta_\ell$ exhibits excellent accuracy.

\begin{figure}[!ht]
  \resizebox{\textwidth}{!}{
   \subfloat{
     \includegraphics{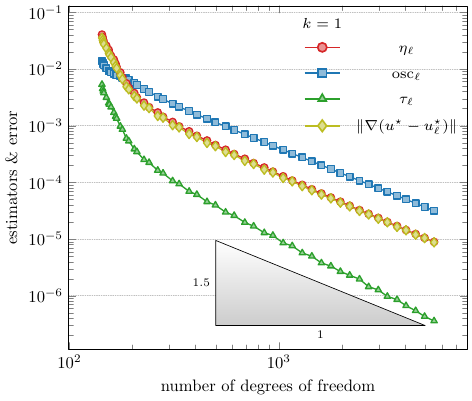}
   }
   \subfloat{
     \includegraphics{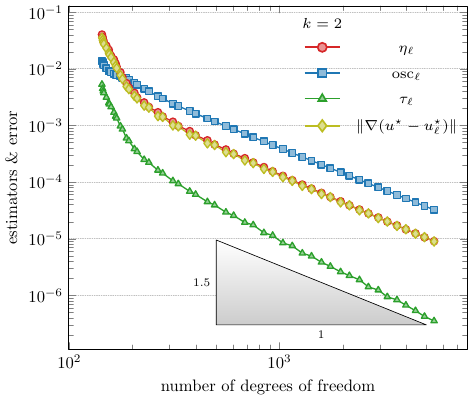}
   }
   
   \subfloat{
     \includegraphics{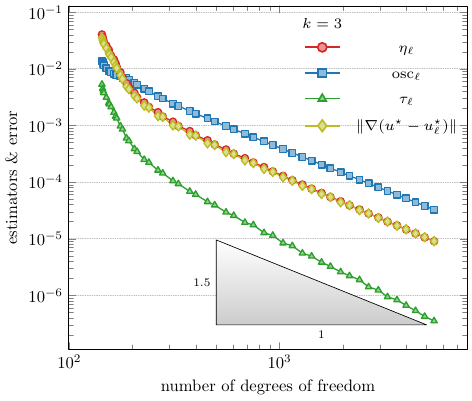}
   }
  }
   \\
   \resizebox{\textwidth}{!}{
   \subfloat{
     \includegraphics{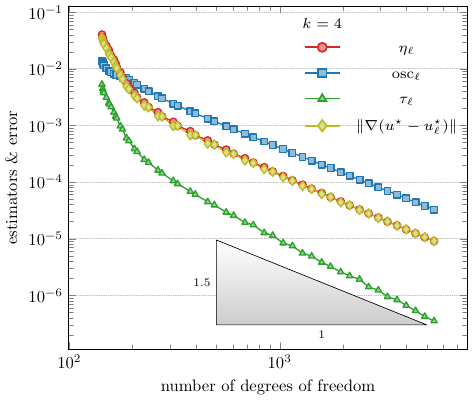}
   }
   \subfloat{
     \includegraphics{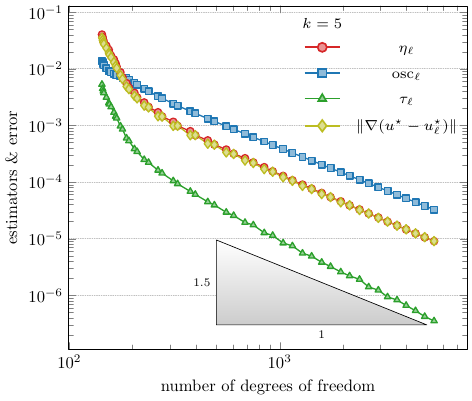}
   }
   \subfloat{
     \includegraphics{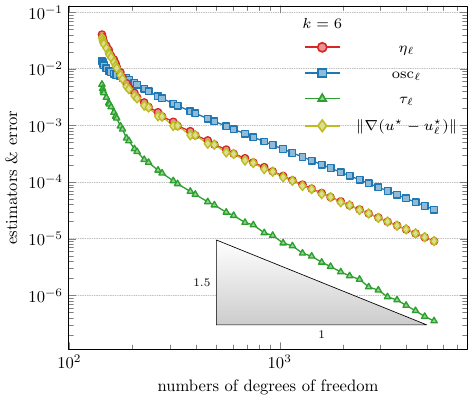}
   }
   }
   \caption{Functional error bounds for different strip depths $k \in \{1, \ldots , 6\}$ in Example~\ref{subsec:zshapeex} with $p=0$, $q=2$, and $\theta = 0.5$. We compare the estimators $\eta_\ell$, $\tau_\ell$, and $\mathrm{osc}_\ell$, as well as the (approximate) potential error $\lVert \nabla(u^\star - u_\ell^\star)\rVert_{L^2(\Omega)}$.}
   \label{fig:zshapekerror}
 \end{figure}

As expected, the proposed error estimation strategy does not depend on the shape of $\Omega$, as it exhibits optimal convergence rates for the Z-shaped domain as well.

%%%%%%%%%%%%%%%%%%%%%%%%%%%%%%%%%%%%%%%%%%%%%%%%%%%%%%%%%%%%%%%%%%%%%%%%%%%%%%%%%%%%%%%%%%%
\subsection{Example 3 (Z-shaped domain, singular potential)} \label{subsec:zshapeex2}
%%%%%%%%%%%%%%%%%%%%%%%%%%%%%%%%%%%%%%%%%%%%%%%%%%%%%%%%%%%%%%%%%%%%%%%%%%%%%%%%%%%%%%%%%%%

We consider the Laplace--Dirichlet problem~\eqref{eq:poisson} with the prescribed exact solution
\begin{equation*}
  u^\star(x) 
  = 
  r^{4/7}\cos(4\varphi/7)
\end{equation*}
given in polar coordinates $(r,\varphi)$ on the Z-shaped domain 
\begin{equation*}
 \Omega \coloneqq (-1/4,1/4)^2 \setminus \{(x,y) \in \mathbb{R}^2 \:\vert\:\mathopen{} -1/4 \leq x \leq 0 \text{ and } -1/4 \leq y \leq -1/4- x\}.
\end{equation*}
Unlike Example~\ref{subsec:squareex} and Example~\ref{subsec:zshapeex}, the potential $u^\star$ is singular at the reentrant corner $(0,0)$. 
Figure~\ref{fig:zshape2meshes} shows adaptively generated meshes, which show that the adaptive refinement strategy is able to resolve the singularity at the reentrant corner.

\begin{figure}[!ht]
\resizebox{\textwidth}{!}{
  \subfloat{
    \stackunder{\stackunder[10pt]{\includegraphics{numerics/Plots/Zshape1/zshapereferencemesh.pdf}}{{\Huge $\# \mathcal{T}_\ell^\Gamma = 143$, $\#\mathcal{T}_\ell^\omega = 579$}}}{{\Huge $\ell=0$}}
  }
  \subfloat{
    \stackunder{\stackunder[10pt]{\includegraphics{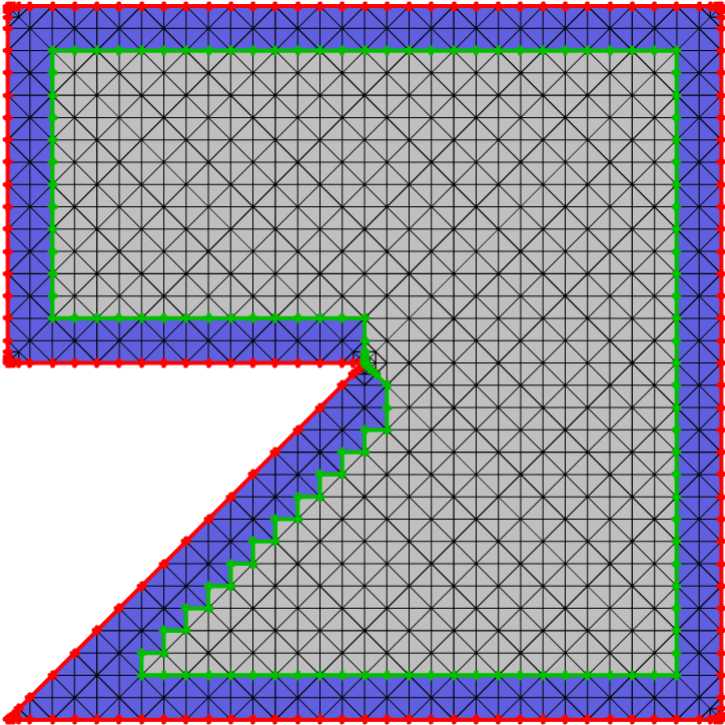}}{{\Huge $\# \mathcal{T}_\ell^\Gamma = 204$, $\#\mathcal{T}_\ell^\omega = 724$}}}{{\Huge $\ell=20$}}
  }
  
  \subfloat{
    \stackunder{\stackunder[10pt]{\includegraphics{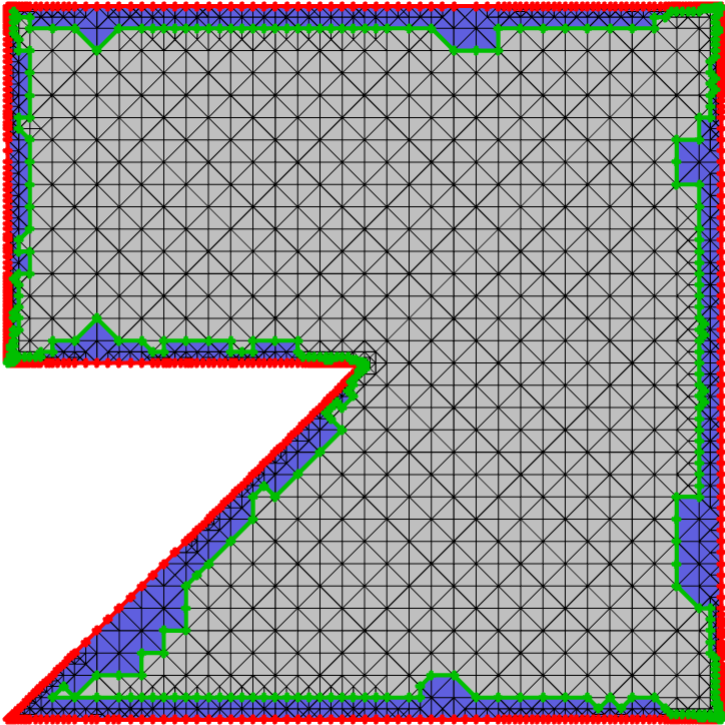}}{{\Huge $\# \mathcal{T}_\ell^\Gamma = 990$, $\#\mathcal{T}_\ell^\omega = 4667$}}}{{\Huge $\ell=50$}}
  }
  \subfloat{
    \stackunder{\stackunder[10pt]{\includegraphics{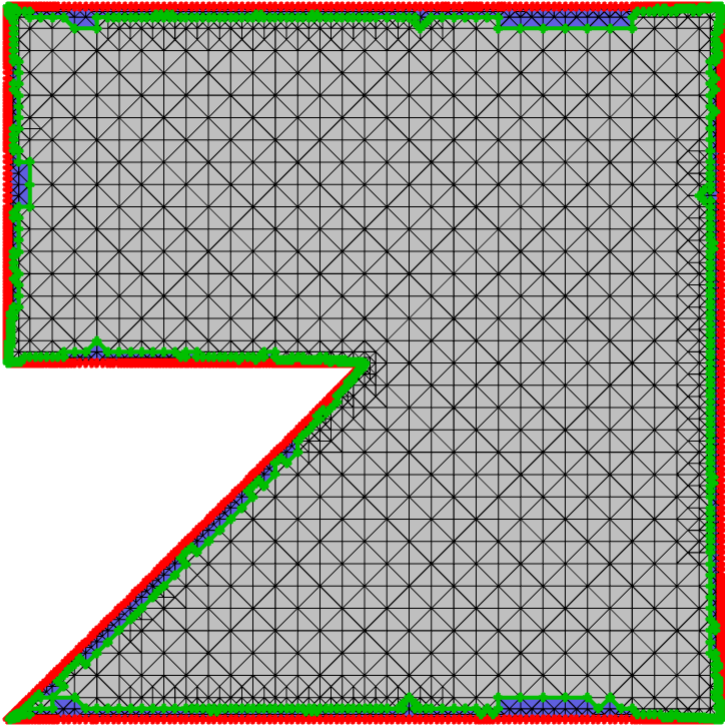}}{{\Huge $\# \mathcal{T}_\ell^\Gamma = 3454$, $\#\mathcal{T}_\ell^\omega = 13512$}}}{{\Huge $\ell=70$}}
  }
}
  \caption{Adaptively generated meshes in Example~\ref{subsec:zshapeex2} for $p=0$, $q=1$, $\theta = 0.5$, and $k = 2$. The triangles of the boundary strip $T \in \mathcal{T}_\ell^\omega$ are depicted in {\color{blue} blue}, while the triangles $T \in \mathcal{T}_\ell \setminus \mathcal{T}_\ell^\omega$ are indicated in {\color{gray} gray}. The outer boundary $\Gamma$ is shown in {\color{red} red} and the inner boundary $\Gamma_\ell^{\mathsf{c}}$ in {\color{green} green}.}
  
  \label{fig:zshape2meshes}
\end{figure}

Figure~\ref{fig:zshape2rates} shows the convergence behaviour for different adaptivity parameters $\theta \in \{0.2,0.4,0.6,0.8\}$ for BEM with $p=0,1$. 
As in Example~\ref{subsec:squareex} and Example~\ref{subsec:zshapeex}, we observe optimal convergence rates.

\begin{figure}[!ht]
  \resizebox{\textwidth}{!}{
    \subfloat{
      \includegraphics{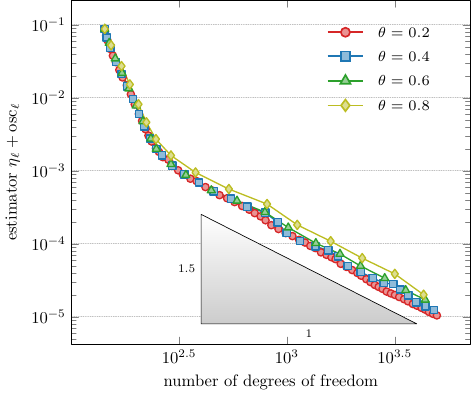}
    }
    \subfloat{
      \includegraphics{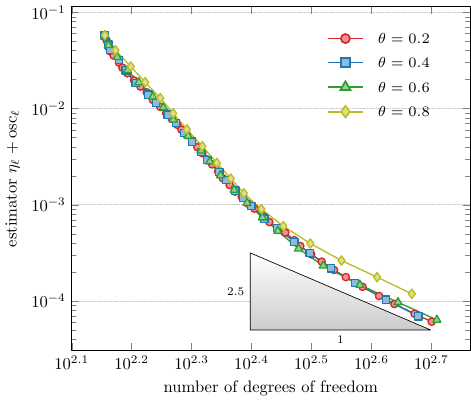}
    }
  }
  \caption{Convergence rates of the full error estimator $\eta_\ell + \mathrm{osc}_\ell$ in Example~\ref{subsec:zshapeex2} for $p = 0$ and $q= 1$ (left), as well as $p=1$ and $q=2$ (right).}
  \label{fig:zshape2rates}
\end{figure}

The estimators $\eta_\ell$, $\tau_\ell$, and $\mathrm{osc}_\ell$ for $p \in \{0,1\}$, $q \in \{1,2\}$, and $k = 3$ are shown in Figure~\ref{fig:zshape2efficiency}. 
As in Example~\ref{subsec:zshapeex}, we observe that $\eta_\ell$ exhibits excellent accuracy and that higher polynomial degrees $q$ improve the accuracy of the functional error estimator. 

\begin{figure}[!ht]
  \resizebox{\textwidth}{!}{
    \subfloat{
      \includegraphics{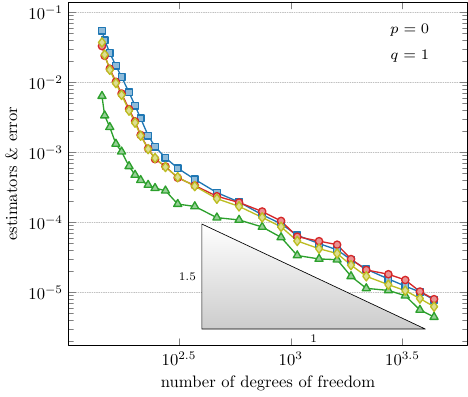}
    }
    \subfloat{
      \includegraphics{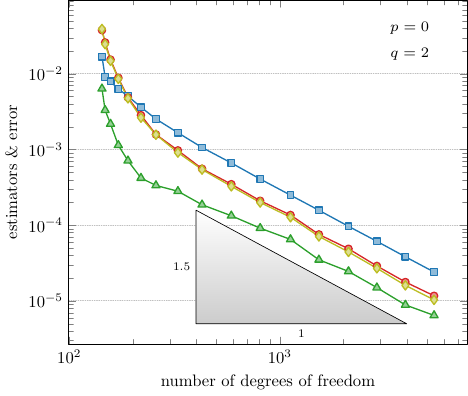}
    }
    \subfloat{
      \includegraphics{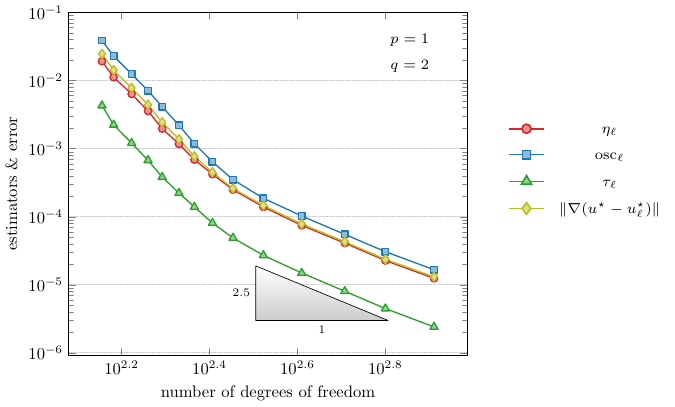}
    }
  }
  \caption{Functional error bounds in Example~\ref{subsec:zshapeex2} with $\theta = 0.5$ and $k = 3$ for $p=0$ and $q=1$ (left), $p=0$ and $q=2$ (middle), as well as $p=1$ and $q=2$ (right).}
  \label{fig:zshape2efficiency}
\end{figure}

%%%%%%%%%%%%%%%%%%%%%%%%%%%%%%%%%%%%%%%%%%%%%%%%%%%%%%%%%%%%%%%%%%%%%%%%%%%%%%%%%%%
%%%%%%%%%%%%%%%%%%%%%%%%%%%%%%%%%%%%%%%%%%%%%%%%%%%%%%%%%%%%%%%%%%%%%%%%%%%%%%%%%%%

%

%%%%%%%%%%%%%%%%%%%%%%%%%%%%%%%%%%%%%%%%%%%%%%%%%%%%%%%%%%%%%%%%%%%%%%%%%%%%%%%%%%%
%%%%%%%%%%%%%%%%%%%%%%%%%%%%%%%%%%%%%%%%%%%%%%%%%%%%%%%%%%%%%%%%%%%%%%%%%%%%%%%%%%%
\printbibliography
\end{document}